\documentclass[10pt]{article}

\usepackage{amsmath,amssymb,amsthm,mathtools}
\usepackage{aliascnt}
\usepackage[margin=0.9in]{geometry}
\usepackage{enumitem}
\usepackage{needspace}
\setlist{topsep=3pt plus 1pt minus 1pt,partopsep=0pt,itemsep=0.5pt plus 0.5pt,parsep=0pt}
\setlist[itemize]{before=\Needspace{4\baselineskip}}
\usepackage{array,longtable}
\usepackage{hyperref}
\AtBeginDocument{%
  \setlength{\topsep}{4pt plus 1pt minus 1pt}%
  \setlength{\abovedisplayskip}{4pt plus 1pt minus 1pt}%
  \setlength{\belowdisplayskip}{4pt plus 1pt minus 1pt}%
  \setlength{\abovedisplayshortskip}{0pt plus 1pt}%
  \setlength{\belowdisplayshortskip}{2pt plus 1pt minus 1pt}%
}

\theoremstyle{plain}
\newtheorem{theorem}{Theorem}[section]
\newaliascnt{conjecture}{theorem}
\newtheorem{conjecture}[conjecture]{Conjecture}
\aliascntresetthe{conjecture}
\newaliascnt{lemma}{theorem}
\newtheorem{lemma}[lemma]{Lemma}
\aliascntresetthe{lemma}
\newaliascnt{proposition}{theorem}
\newtheorem{proposition}[proposition]{Proposition}
\aliascntresetthe{proposition}
\newaliascnt{corollary}{theorem}
\newtheorem{corollary}[corollary]{Corollary}
\aliascntresetthe{corollary}
\theoremstyle{definition}
\newaliascnt{definition}{theorem}
\newtheorem{definition}[definition]{Definition}
\aliascntresetthe{definition}
\theoremstyle{remark}
\newaliascnt{remark}{theorem}
\newtheorem{remark}[remark]{Remark}
\aliascntresetthe{remark}
\numberwithin{equation}{section}

\newcolumntype{G}[1]{>{\raggedright\arraybackslash}p{#1}}

\newcommand{\PSL}{\operatorname{PSL}}
\newcommand{\PGL}{\operatorname{PGL}}
\newcommand{\PSU}{\operatorname{PSU}}
\newcommand{\PGU}{\operatorname{PGU}}
\newcommand{\PSp}{\operatorname{PSp}}
\newcommand{\PCSp}{\operatorname{PCSp}}

\newcommand{\POmega}{\mathrm{P}\Omega}
\newcommand{\PGammaL}{\operatorname{P\Gamma L}}
\newcommand{\GL}{\operatorname{GL}}
\newcommand{\SL}{\operatorname{SL}}
\newcommand{\SU}{\operatorname{SU}}
\newcommand{\GU}{\operatorname{GU}}
\newcommand{\Sp}{\operatorname{Sp}}
\newcommand{\AGL}{\operatorname{AGL}}

\newcommand{\Soc}{\operatorname{soc}}
\newcommand{\Aut}{\operatorname{Aut}}
\newcommand{\Hom}{\operatorname{Hom}}
\newcommand{\Core}{\operatorname{Core}}
\newcommand{\Stab}{\operatorname{Stab}}
\newcommand{\Fix}{\operatorname{Fix}}
\newcommand{\Sym}{\operatorname{Sym}}
\newcommand{\ord}{\operatorname{ord}}
\newcommand{\diag}{\operatorname{diag}}

\newcommand{\F}{\mathbb{F}}
\newcommand{\Z}{\mathbb{Z}}
\newcommand{\gen}[1]{\langle #1\rangle}

\title{A proof of the Galt--Tyutyanov conjecture on Skiba's question\thanks{ \scriptsize\emph{E-mail address:}
 zhangshengmin26@stu.pku.edu.cn\,(Zhang S-M).}}

\author{Zhang Sheng-Min\thanks{Corresponding author.} \\ \small School of Mathematical Sciences, Peking University}
\date{}

\begin{document}
\maketitle

\begin{abstract}
A subgroup $A$ of a finite group $G$ is \emph{hereditarily $G$-permutable} if, for every subgroup $E$ with $A\leq E\leq G$ and every subgroup $B$ of $E$, some $G$-conjugate of $B$ permutes with $A$ in $E$. We give a new criterion for a finite group to be simple. A nontrivial finite group $G$ is simple if and only if it has no proper nontrivial hereditarily $G$-permutable subgroup. This proves the conjecture of Galt and Tyutyanov and answers Skiba's question in part \textup{(b)} of Problem 17.112 in the Kourovka Notebook. Earlier work settled the conjecture for alternating groups, sporadic groups, exceptional groups of Lie type, and the groups ${\rm PSL}_2(q)$ and ${\rm PSU}_3(q)$. We prove the conjecture for all classical groups. The proof uses descent through parabolic subgroups and the classification of factorizations of finite simple groups.


\end{abstract}

\section{Introduction}\label{sec:introduction}

Throughout the paper all groups are finite, and we write $A^x=x^{-1}Ax$.

Subgroup permutability is a basic and important property in finite group theory. Its study began with Ore's work \cite[\S1]{GT2025} on permutable subgroups. Skiba \cite{Skiba2003,Skiba2006} introduced generalized forms of this property. We use the following definitions.

\begin{definition}\label{def:g-permutable}
Let $A\leq G$. We say that $A$ is \emph{$G$-permutable} in $G$ if for every subgroup $B\leq G$ there is an element $x\in G$ with
$$
        AB^x=B^xA.
$$
We say that $A$ is \emph{hereditarily $G$-permutable} in $G$ if $A$ is $E$-permutable in $E$ for every subgroup $E$ with $A\leq E\leq G$. These definitions follow Skiba \cite[Definitions 2.1 and 2.2]{Skiba2006}.
\end{definition}

Guo, Skiba and Shum \cite{GuoSkibaShum2007} used generalized permutability conditions to characterize solvable, supersolvable and nilpotent groups. In 2006, Tyutyanov and Skiba \cite[Question 2.41]{Skiba2006} asked which finite nonabelian simple groups have a proper nontrivial ordinary $G$-permutable subgroup. Both the ordinary and hereditary questions appeared as Problem 17.112 in the 2010 issue of the Kourovka Notebook \cite{Kourovka2014}. Part \textup{(b)} concerns hereditary permutability. Tyutyanov and Bychkov \cite{TB2008} answered the hereditary question for sporadic groups. Vasil'ev and Tyutyanov \cite{VT2012} proved the corresponding result for alternating groups. Galt and Tyutyanov \cite{GT2022,GT2023} obtained the same result for the Tits group, exceptional simple groups and $\PSL_2(q)$. Their 2025 paper \cite[Theorem 2 and Section 4]{GT2025} answers the hereditary question for $\PSU_3(q)$. The classical groups of higher rank remained.

In their 2023 paper, Galt and Tyutyanov formulated the following conjecture.

\begin{conjecture}[{\cite[Conjecture]{GT2023}}]\label{conj:gt}
A nontrivial finite group $G$ is simple if and only if $G$ has no proper nontrivial hereditarily $G$-permutable subgroup.
\end{conjecture}

We prove the conjecture for all classical groups. Together with the results for alternating, sporadic and exceptional simple groups, our classical results give the following theorem.

\begin{theorem}\label{thm:main}
No nonabelian finite simple group $G$ has a proper nontrivial hereditarily $G$-permutable subgroup.
\end{theorem}

Theorem \ref{thm:main}, together with Corollary \ref{cor:simplicity-criterion}, proves Conjecture \ref{conj:gt}. Equivalently, a nontrivial finite group is simple if and only if it has no proper nontrivial hereditarily permutable subgroup.

\begin{corollary}\label{cor:simplicity-criterion}
A nontrivial finite group $G$ is simple if and only if $G$ has no proper nontrivial hereditarily $G$-permutable subgroup.
\end{corollary}

\begin{proof}
If $G$ is not simple, then $G$ has a proper nontrivial normal subgroup $N$. For every subgroup $B\leq G$ the product $NB=BN$ is again a subgroup, and the equality $NB=BN$ holds inside every subgroup $E$ with $N\leq E\leq G$ because $N\unlhd E$. Hence $N$ is a proper nontrivial hereditarily $G$-permutable subgroup. Conversely, suppose that $G$ is simple. If $G$ is abelian, then $G$ is cyclic of prime order and has no proper nontrivial subgroup at all. If $G$ is nonabelian, then we use Theorem \ref{thm:main}. In each case $G$ has no proper nontrivial hereditarily $G$-permutable subgroup.
\end{proof}

Our arguments use descent through parabolic subgroups and the classification of factorizations of finite simple groups. We use both ordinary $G$-permutability and hereditary $G$-permutability at several steps of the proofs. The distinction between these two notions is especially relevant in Section \ref{sec:unitary}.

We also use the theorem of Yang and Galt on $G$-permutable subgroups of alternating groups of prime degree. Its only use is inside the unitary argument, in the proof leading to Theorem \ref{thm:unitary-simplex-image}, where \emph{ordinary} $G$-permutability is needed. The ordinary theorem of Yang and Galt does not replace the theorem on hereditarily $G$-permutable subgroups of alternating groups.

We treat the linear, unitary, symplectic, and orthogonal families in that order and then combine their results in Section \ref{sec:assembly}. The unitary arguments use the linear theorem (Theorem \ref{thm:split-linear}). The orthogonal cases of low rank use the linear, unitary and symplectic theorems (Theorems \ref{thm:split-linear}, \ref{thm:unitary} and \ref{thm:symplectic}). Section \ref{sec:assembly} uses all four theorems. Each section uses only results proved in the sections on which it depends or quoted from the literature.

We write $\Soc(G)$ for the socle of $G$ and $\Core_G(H)=\bigcap_{g\in G}H^g$ for the core of $H$ in $G$. We say that a subgroup is \emph{core-free} when its core in $G$ is trivial.

For orthogonal groups we write $\POmega(V)$ for the projective group attached to a nondegenerate quadratic space $V$. In the linear section (Section \ref{sec:split-linear}) $P_i$ denotes the stabilizer in the group under consideration of an $i$-dimensional subspace of its natural module. In the unitary, symplectic and orthogonal sections (Sections \ref{sec:unitary}, \ref{sec:symplectic} and \ref{sec:orthogonal}) $P_i$ denotes the standard parabolic subgroup whenever such a parabolic exists. This subgroup stabilizes a totally singular $i$-dimensional subspace of the natural module for the group under consideration.

For a group of Lie type, a \emph{Borel subgroup} is the subgroup $B$ of the standard (split or twisted) $BN$-pair. A \emph{parabolic subgroup} is a subgroup of the same group of Lie type containing a Borel subgroup. For a parabolic subgroup, the Levi quotient is the quotient of that parabolic subgroup by its unipotent radical. When we specify a projective action of a Levi factor, we use the image of that action. Equivalently, we quotient that Levi factor by the kernel of the specified projective action. We work throughout in the finite group of Lie type. We follow Carter \cite{Carter1985} and Malle and Testerman \cite{MalleTesterman2011} for the parabolic and Levi structures.

\section{Preliminaries}\label{sec:preliminaries}

The following lemmas will be used throughout the paper.

\begin{lemma}\label{lem:generated-hereditary-formulation}
For $A\leq G$, the hereditary condition of {\rm Definition \ref{def:g-permutable}} is equivalent to the following: for every subgroup $B\leq G$ there exists $x\in \gen{A,B}$ with $AB^x=B^xA$.
\end{lemma}

\begin{proof}
Suppose that $A$ is hereditarily $G$-permutable. For $B\leq G$, we use hereditary permutability of $A$ inside $E=\gen{A,B}$. Consequently, the conjugating element provided by this hereditary permutability condition lies in $\gen{A,B}$. Conversely, suppose that the condition in the lemma holds, and let $A\leq E\leq G$ with $B\leq E$. The conjugating element $x$ for $A$ and $B$ lies in $\gen{A,B}\leq E$ and shows that $A$ permutes with $B^x$ inside $E$.
\end{proof}

\begin{lemma}\label{lem:maximal-product}
Let $A\leq G$ be hereditarily $G$-permutable and let $M\leq G$. Then there is $x\in \gen{A,M}$ for which $AM^x$ is a subgroup of $G$. If moreover $M$ is maximal in $G$, then either $A\leq M^x$ or $G=AM^x$.
\end{lemma}

\begin{proof}
By hereditary $G$-permutability of $A$ for $M$, we have $x\in\gen{A,M}$ with $AM^x=M^xA$, so $AM^x$ is a subgroup. If $M$ is maximal, then $AM^x$, which contains $M^x$, equals $M^x$ or $G$.
\end{proof}

\begin{lemma}\label{lem:ordinary-maximal-product}
Let $A\leq G$ be ordinarily $G$-permutable and let $M\leq G$. Then there is $x\in G$ for which $AM^x$ is a subgroup of $G$. If $M$ is maximal in $G$, then either $A\leq M^x$ or $G=AM^x$.
\end{lemma}

\begin{proof}
We repeat the argument of Lemma \ref{lem:maximal-product}, with $x\in G$ chosen by ordinary $G$-permutability of $A$.
\end{proof}

\begin{lemma}\label{lem:factorization-coset-transitivity}
Let $G$ act on the right cosets of a subgroup $M$ by right multiplication. For $A\leq G$, the group $A$ is transitive on the right cosets of $M$ in $G$ if and only if $G=MA=AM$.
\end{lemma}

\begin{proof}
The $A$-orbit of the coset $M$ is $\{Ma:a\in A\}$, which is the whole coset space exactly when every $g\in G$ lies in some $Ma$, i.e., when $G=MA=AM$.
\end{proof}

\begin{lemma}\label{lem:overgroup-restriction}
Let $A\leq E\leq G$. If $A$ is hereditarily $G$-permutable, then $A$ is hereditarily $E$-permutable.
\end{lemma}

\begin{proof}
For $B\leq E$, by hereditary $G$-permutability of $A$ we have $x\in\gen{A,B}$ with $AB^x=B^xA$, and $\gen{A,B}\leq E$ since $A,B\leq E$. Consequently, the equality $AB^x=B^xA$ establishes hereditary $E$-permutability of $A$.
\end{proof}

\begin{lemma}\label{lem:quotient-inheritance}
Let $A\leq E$ and $N\unlhd E$. If $A$ is hereditarily $E$-permutable, then $AN/N$ is hereditarily $(E/N)$-permutable.
\end{lemma}

\begin{proof}
Let $\rho:E\to E/N$ be the quotient map. For a subgroup $T\leq E/N$, choose $B=\rho^{-1}(T)$. Then $BN/N\leq E/N$ and $T=BN/N$. By hereditary $E$-permutability of $A$ for $B$, we have $x\in\gen{A,B}$ with $AB^x=B^xA$. Taking the images of both sides of this equality in the quotient of $E$ by $N$, namely $E/N$, gives
$$
        (AN/N)(BN/N)^{xN}=(BN/N)^{xN}(AN/N),
$$
and $xN\in\gen{AN/N,BN/N}$.
\end{proof}

\begin{lemma}\label{lem:almost-simple-core-free}
Let $S\leq E\leq \Aut(S)$ with $S$ nonabelian simple. A subgroup $H\leq E$ is core-free in $E$ if and only if $S\nleq H$.
\end{lemma}

\begin{proof}
If $S\leq H$, then $S\leq\Core_E(H)$, so $H$ is not core-free in $E$. Conversely, suppose that $N=\Core_E(H)\neq1$. Since $N\unlhd E$, the subgroup $N\cap S$ is normal in $S$. If $N\cap S=1$, then $[N,S]\leq N\cap S=1$, so $N$ centralizes $S$. Together with $N\leq\Aut(S)$, this centralization implies that $N=1$, a contradiction. Hence $N\cap S=S$, and $S\leq N\leq H$.
\end{proof}

\begin{lemma}\label{lem:normal-intersection-transfer}
Let $A\leq G$ be hereditarily $G$-permutable, let $N\unlhd G$, and let $K=A\cap N$. Then $K$ is ordinarily $N$-permutable in $N$. For every $H\leq N$ there is $n\in N$ with
$$
        KH^n=H^nK.
$$
\end{lemma}

\begin{proof}
Fix $H\leq N$ and let $R=\gen{A,H}$. By hereditary $G$-permutability of $A$, we choose $x\in R$ with $AH^x=H^xA$. Since $N\unlhd G$, we have $R\leq AN=NA$, so we may write $x=nl$ with $n\in N$, $l\in A$. We take the intersection of both sides of $AH^x=H^xA$ with $N$. As $H^x\leq N$, this intersection gives
$$
        (A\cap N)H^x=H^x(A\cap N).
$$
Now $H^x=(H^n)^l$ and $A\cap N\unlhd A$. Conjugating the equality $(A\cap N)H^x=H^x(A\cap N)$ by $l^{-1}$, we obtain that $KH^n=H^nK$ with $n\in N$. Thus, the equality establishes ordinary $N$-permutability of $K$ only, and says nothing about hereditary permutability of the various subgroups of $A$.
\end{proof}

\begin{remark}\label{rem:actual-levi}
At each parabolic descent step, we study the image of $A$ in the full Levi quotient specified at that step. In particular, we include the diagonal part of this quotient as well as its simple socle. In the linear case the quotient is a group $E$ with
$$
        \PSL_m(q)\leq E\leq \PGL_m(q),
$$
and by Lemma \ref{lem:quotient-inheritance} the image of $A$ is hereditarily $E$-permutable. Consequently, we state the lemmas describing the descent and its alternatives for $E$ rather than for the socle of $E$, $\PSL_m(q)$. The parabolic and Levi structure used in these steps is standard. For this parabolic and Levi structure, we follow Carter \cite{Carter1985}, Malle and Testerman \cite{MalleTesterman2011}, and Kleidman and Liebeck \cite{KleidmanLiebeck1990}.
\end{remark}

\section{Linear groups}\label{sec:split-linear}

Following Kleidman and Liebeck \cite[\S2.1]{KleidmanLiebeck1990}, let $V$ be the natural space of dimension $n$ over $\F_q$ for $\SL_n(q)$. A \emph{line}, a \emph{plane} and a \emph{hyperplane} in $V$ have dimensions $1$, $2$ and $n-1$, respectively. In projective geometry, the lines of $V$ are the projective points. Since scalar matrices fix every subspace as a subspace, the action of $\SL_n(q)$ on subspaces induces an action of $\PSL_n(q)$. More generally, a linear action on a vector space induces its \emph{projective action} on the lines of that space.

A \emph{flag} in $V$ is a chain $V_1<\cdots<V_t$ of nonzero proper subspaces, with each inclusion strict. A subgroup stabilizes this chain if it fixes each $V_i$ as a subspace. More explicitly, an element $g\in\GL(V)$ of this stabilizer satisfies $g(V_i)=V_i$. However, the element $g$ may move vectors within $V_i$. The subspace $V_i$ is fixed pointwise by $g$ if and only if $g(v)=v$ for every $v\in V_i$. An extension of $V_1<\cdots<V_t$ is a flag $W_1<\cdots<W_s$ in $V$ with $\{V_1,\ldots,V_t\}\subsetneq\{W_1,\ldots,W_s\}$. The original subspaces $V_i$ are terms of the new flag, and $s>t$. A \emph{complete flag} contains a subspace of every dimension from $1$ to $n-1$. Adding the zero subspace $0$ and the whole space $V$ to this complete flag gives the chain
$$
        0=V_0<V_1<\cdots<V_n=V.
$$
In this chain, each $V_i$ has dimension $i$ for $0\leq i\leq n$. For example, the flag $\ell<H$ is complete in dimension three, with full chain $0<\ell<H<V$.

For the complete flag $0=V_0<\cdots<V_n=V$, choose an $\F_q$-basis $(e_1,\ldots,e_n)$ of $V$ such that $V_i=\gen{e_1,\ldots,e_i}$ for $1\leq i\leq n$. In this basis, the stabilizer $B$ of the flag in $\PSL_n(q)$ is represented by upper triangular matrices of determinant $1$. Among these matrices, those with every diagonal entry equal to $1$ represent the unipotent radical $R_u(B)$. More generally, for a parabolic subgroup $P$ we write $R_u(P)$ for its unipotent radical. For actions on quotient spaces, we use $W/Y$ with $Y\leq W\leq V$. This space consists of the cosets $w+Y$ with $w\in W$. If $g\in\GL(V)$ fixes $Y$ and $W$ as subspaces, then its induced action on $W/Y$ is $w+Y\mapsto g(w)+Y$.

\begin{theorem}\label{thm:split-linear}
For every nonabelian simple group $G=\PSL_n(q)$ there is no proper nontrivial hereditarily $G$-permutable subgroup.
\end{theorem}

For $n=2$, the conclusion is the known $\PSL_2(q)$ theorem \cite[Proposition 1]{GT2023}, so we assume that $n\geq3$ throughout the section. In this section, we shall always assume that $A<G$ is a proper nontrivial hereditarily $G$-permutable subgroup, where $G = \PSL_n (q) = \PSL (V)$.

\subsection{\texorpdfstring{$A$ lies in a proper parabolic}{A lies in a proper parabolic}}

\begin{proposition}\label{prop:sl-parabolic-containment}
$A$ is contained in a proper parabolic subgroup of $G$.
\end{proposition}

\begin{proof}
Suppose that $A$ is not contained in a proper parabolic subgroup of $G$. Let $B$ be a Borel subgroup of $G$. By hereditary permutability of $A$ in $G$ there is $x\in G$ with $AB^x=B^xA$, so $AB^x$ is a subgroup. If $AB^x<G$, then a maximal subgroup containing $AB^x$ would contain the Borel subgroup $B^x$ and hence be parabolic. This would contradict the assumption that $A$ lies in no proper parabolic subgroup of $G$. Therefore, $G=AB^x$. Because $A<G$, we choose a maximal subgroup $M$ with $A\leq M<G$. Since $A\leq M$, we have a factorization
$$
        G=MB^x
$$
with solvable factor $B^x$.

We now use the classification of factorizations with a solvable factor, due to Li and Xia \cite[Theorem 1.1 and Tables 1.1--1.2]{LiXia2022}. The relevant linear rows are Table 1.1, rows 1--2, and Table 1.2, rows 6--9, which are all the rows of \cite[Tables 1.1--1.2]{LiXia2022} with linear socle $\PSL_n(q)$, $n\geq3$. Let $p$ be the characteristic of the field. In each of these listed linear rows the solvable factor has a smaller $p$-part than a Borel subgroup of $\PSL_n(q)$. The $p$-part for a Borel subgroup is
$$
        |B|_p=q^{n(n-1)/2},
$$
since type $A_{n-1}$ has $n(n-1)/2$ positive roots \cite{Carter1985}. The following table compares the two $p$-parts.
$$
\begin{array}{c|c|c|c}
\text{socle} & \text{solvable-factor row} & \text{solvable }p\text{-part} & \text{Borel }p\text{-part}\\\hline
\PSL_n(q) & \widehat{\GL}_1(q^n):n & |n|_p & q^{n(n-1)/2}\\
\PSL_4(q) & q^3:((q^3-1)/(4,q-1))\cdot 3 & q^3 |3|_p & q^6\\
\PSL_4(3) & 2^4:5:4 & 1 & 3^6\\
\PSL_4(3) & 3^3:13:3 & 3^4 & 3^6\\
\PSL_4(4) & 2^6:63:3 & 2^6 & 2^{12}\\
\PSL_5(2) & 31:5 & 1 & 2^{10}
\end{array}
$$
We have $|n|_p<q^{n(n-1)/2}$ for $n\geq3$, and $q^3|3|_p<q^6$ in the second row (including $p=3$). Thus, $B^x$ cannot be one of the solvable factors in the listed linear rows, a contradiction.
\end{proof}

\subsection{From a parabolic to a line and a hyperplane}

\begin{lemma}\label{lem:sl-line-hyperplane}
Suppose that $A\leq P_j$, where $P_j$ is the stabilizer in $G$ of a $j$-dimensional subspace $U$ of $V$ with $1\leq j\leq n-1$. Then, after relabelling, $A$ stabilizes a line $\ell$ and a hyperplane $H$ of $V$.
\end{lemma}

\begin{proof}
We use Lemma \ref{lem:maximal-product} for the point stabilizer $P_1$ in $G$, which is maximal in $G$. If $G=AP_1^x$, then $A$ is transitive on the projective points of $V$. However, $A$ stabilizes $U$ and so maps the nonempty proper set of projective points in $U$ to itself. This contradicts the transitivity of $A$. Hence $A\leq P_1^x$ for some $x$, i.e., $A$ stabilizes a line of $V$. Similarly, hereditary permutability of $A$ with the hyperplane stabilizer in $G$, together with the nontransitivity of $A$ on hyperplanes, shows that $A\leq P_{n-1}^y$ for some $y$.
\end{proof}

After Lemma \ref{lem:sl-line-hyperplane} we distinguish whether or not the line $\ell$ is contained in the hyperplane $H$.

\subsubsection*{The case $\ell\not\leq H$}
If $\ell\not\leq H$, then $V=\ell\oplus H$. We consider the action of the stabilizer of $\ell$ in $G$ on $V/\ell$, identified with $H$ by the decomposition $V=\ell\oplus H$. The identification is the map $H\to V/\ell$ given by $h\mapsto h+\ell$. Write $P_\ell=\Stab_G(\ell)$ and let $\rho_\ell:P_\ell\to\PGL(H)$ be the induced projective action. Let $K_\ell=\ker\rho_\ell$. We quotient the stabilizer of $\ell$ in $G$ by the kernel of its projective action on $H$. The kernel consists of the elements whose induced action on $H$ is scalar. By \cite[Proposition 26.1]{MalleTesterman2011}, this projective factor of the Levi quotient has socle $\PSL_{n-1}(q)$, and is a group $E$ with
$$
        \PSL_{n-1}(q)\leq E\leq \PGL_{n-1}(q).
$$ For $E=P_\ell/K_\ell$, the image of $A$ is $\bar A=AK_\ell/K_\ell$. By Lemma \ref{lem:quotient-inheritance} the image $\bar A $ of $A$ in the quotient $E$ determined by this projective action is hereditarily $E$-permutable.

\subsubsection*{The case $\ell\leq H$}
If $\ell\leq H$, then $A$ stabilizes the flag $\ell<H$ with $\dim\ell=1$ and $\dim H=n-1$. Let
$$
        W=H/\ell,~ d=\dim W=n-2.
$$
We take the stabilizer of the flag $\ell<H$ in $G$ and its projective action on $W$. Write $P_{\ell,H}=\Stab_G(\ell<H)$ and let $\rho_{\ell,H}:P_{\ell,H}\to\PGL(W)$ be this action on the lines of $W=H/\ell$. Let $K_{\ell,H}=\ker\rho_{\ell,H}$. We quotient this flag stabilizer by the elements acting as scalars on $W$, obtaining $E=P_{\ell,H}/K_{\ell,H}$. The group $E$ is the projective factor of the Levi quotient acting on $W$, with socle $\PSL_d(q)$ and $\PSL_d(q)\leq E\leq\PGL_d(q)$ \cite[Proposition 26.1]{MalleTesterman2011}. By Lemma \ref{lem:quotient-inheritance}, the image $\bar{A}$ of $A$ in $E$ is hereditarily $E$-permutable.

We next describe the descent for the projective images of $A$. At each step, let $\mathcal F_0$ be the flag in $V$ stabilized by $A$. Let $P_0=\Stab_G(\mathcal F_0)$. After adding the zero subspace $0$ and the whole space $V$, write $\mathcal F_0$ as $0=F_0<F_1<\cdots<F_s=V$. Let $U_0$ be the subquotient of $V$ corresponding to the linear factor at that step. For the index $j$ of this factor, $U_0=F_j/F_{j-1}$. The projective action of $P_0$ on $U_0$ is its action on the subspaces of $U_0$ of dimension one. Let $\pi:P_0\to E$ be the quotient map onto the image of this action. Write $C=\pi(A)$. Initially, $\mathcal F_0=(\ell)$ and $U_0=V/\ell\cong H$ when $\ell\not\leq H$. When $\ell\leq H$, the initial flag is $\mathcal F_0=(\ell<H)$ and $U_0=H/\ell$. We then give the base of rank one for the descent. Throughout the rest of this section, $\ell$ denotes a line in $V$ stabilized by $A$ and $H$ denotes a hyperplane in $V$ stabilized by $A$. The next lemma gives the alternatives for a quotient $E$ with socle $\PSL_m(q)$, where $m\geq3$.

\begin{lemma}\label{lem:sl-actual-quotient-trichotomy}
Let $m\geq3$ and let $S=\PSL_m(q)$. Suppose that $S\leq E\leq\PGL_m(q)$. Let $C\leq E$ be hereditarily $E$-permutable. Then we have three alternatives for $C$.
\begin{enumerate}[label=\textup{(\arabic*)}]
\item $C\geq S$.
\item $C$ lies in the normalizer in $E$ of a proper parabolic subgroup of $S$. Let $Q$ denote this proper parabolic subgroup. We take the inverse image of this proper parabolic subgroup of $S$ under $\pi:P_0\to E$. Let $U_1<\cdots<U_a$ be the flag in $U_0=F_j/F_{j-1}$ stabilized by $Q$. For the quotient map $\tau_j:F_j\to U_0$, write $Y_i=\tau_j^{-1}(U_i)$. The new flag inserts the chain $F_{j-1}<Y_1<\cdots<Y_a<F_j$ between the same two terms $F_{j-1}$ and $F_j$ of $\mathcal F_0$. All other terms of $\mathcal F_0$ remain in the new flag.
\item After every possible extension of the flag through the parabolic subgroups in (2), let $\mathcal F$ be the extended flag in $U_0$, the natural module for $E$. Let $P_{\mathcal F}=\operatorname{Stab}_E(\mathcal F)$. The Levi quotient is $P_{\mathcal F}/R_u(P_{\mathcal F})$, which we denote by $L_{\mathcal F}$. After adding the zero subspace $0$ and the whole space $U_0$, write $\mathcal F$ as $0=W_0<W_1<\cdots<W_b=U_0$. The linear factors of $L_{\mathcal F}$ correspond to the actions on $W_i/W_{i-1}$. Under these actions, the image of $C$ in each simple factor of the semisimple part of $L_{\mathcal F}$ is trivial. Consequently, we continue with the Borel argument for the extended flag.
\end{enumerate}
\end{lemma}

\begin{proof}
Suppose that $C\not\geq S$ and $C$ does not lie in the normalizer in $E$ of a proper parabolic subgroup of $S$. Suppose also that the semisimple image of $C$ in the projective Levi quotient $E$ with simple socle $S$ is nontrivial. We must derive a contradiction from these assumptions. We let $B_S$ be a Borel subgroup of $S$ and $B_E=N_E(B_S)$. Since a diagonal projective element can be represented by a diagonal matrix and so normalizes the standard Borel subgroup of $S$, $B_E$ maps onto $E/S$ under the quotient map from $E$ to $E/S$. By hereditary $E$-permutability of $C$ for $B_E$, we have a conjugate $B_E^x$ with
$$
        H=CB_E^x
$$
a subgroup of $E$. If $H<E$, then take a maximal subgroup $M<E$ containing $H$. Since $B_E^x$ maps onto $E/S$, $M$ cannot contain $S$: otherwise $M/S=E/S$ and $M=E$. Hence $M\cap S$ is a proper subgroup of $S$ containing the Borel subgroup $B_S^x$. Therefore, it is parabolic. Consequently, $C\leq M\leq N_E(M\cap S)$, so $C$ lies in the normalizer in $E$ of a proper parabolic subgroup of $S$, contrary to assumption. Therefore,
$$
        E=CB_E^x,
$$
a factorization of the almost simple group $E$ with solvable factor $B_E^x$. Neither $C$ nor $B_E ^x$ contains $S$, so both are core-free in $E$ (Lemma \ref{lem:almost-simple-core-free}), and $B_E^x\cap S$ has the same $p$-part of its order as a Borel subgroup of $S$. As shown in Proposition \ref{prop:sl-parabolic-containment}, none of the linear rows of Li and Xia has a solvable factor whose order has a $p$-part at least as large as that of a Borel subgroup of $S$ when $m\geq3$. This contradiction completes the proof.
\end{proof}

When $\pi$ has image $E$ with $\PSL_2(q)\leq E\leq\PGL_2(q)$, we do not use Lemma \ref{lem:sl-actual-quotient-trichotomy}. If the image of $A$ lies in a parabolic subgroup of $E$, then it fixes a line $K<U_0$ as a subspace. For $\tau_j:F_j\to U_0=F_j/F_{j-1}$, we extend the flag by inserting $F_{j-1}<\tau_j^{-1}(K)<F_j$ between $F_{j-1}$ and $F_j$. We handle an image of $A$ with no invariant line by Lemma \ref{lem:sl-root-line-module-obstruction}, and an image of $A$ contained in a Borel subgroup of $E$ by Lemma~\ref{lem:pgl2-borel}.
\subsection{Restrictions on the projective image}

\begin{lemma}\label{lem:sl-root-line-module-obstruction}
Let $V$ be a vector space over $K=\F_q$, written additively, and suppose that
$$
        F=VA
$$
is a semidirect product with $A\cap V=1$ and $A$ acting $K$-linearly on $V$. We write $\bar A$ for the image of the action of $A$ on $V$ in $\GL_K(V)$. If $A$ is hereditarily $F$-permutable and $U\leq V$ is a $K$-line, then $\bar A$ stabilizes some $K$-line of $V$. In particular, if $\bar A$ has no invariant $K$-line on $V$, then no subgroup $A$ satisfying these hypotheses exists.
\end{lemma}

\begin{proof}
By hereditary $F$-permutability of $A$ with $U$, we have $x\in F$ with $AU^x=U^xA$ a subgroup of $F$. Since $V\unlhd F$ and the action of $A$ on $V$ is $K$-linear, $U^x$ is again a $K$-line. The equality $A\cap V=1$ also gives
$$
        (AU^x)\cap V=U^x
$$
Hence, for every $a\in A$,
$$
        aU^xa^{-1}\leq (AU^x)\cap V=U^x,
$$
so $A$ stabilizes the $K$-line $U^x$. Therefore, $\bar A$ stabilizes the line $U^x$ as well. The existence of this invariant line is impossible if $\bar A$ has no invariant $K$-line.
\end{proof}

\begin{lemma}\label{lem:sl-nonincident-full-cover}
Suppose that $A \leq G = {\rm PSL}_n (q)={\rm PSL} (V)$ stabilizes the pair $\ell, H$ (i.e. $\ell$ is a line and $H$ is a hyperplane in $V$) with $V=\ell\oplus H$. If $\dim H\geq3$ and the image $\bar{A}$ of $A$ on $H$ contains $\PSL(H)\cong\PSL_{\dim H}(q)$, then these conditions are incompatible. If $\dim H=2$ and the image $\bar{A}$ of $A$ on $H$ stabilizes no line, then these conditions are also incompatible.
\end{lemma}

\begin{proof}
Suppose that $\dim H\geq3$, and test hereditary permutability of $A$ in $G$ against a $2$-space stabilizer $P_2$ in $G$, which is maximal in $G$. Since $A$ fixes $\ell$, the group $A$ maps the nonempty proper set of $2$-spaces containing $\ell$ to itself. Using Lemma \ref{lem:maximal-product} we get either $G = A P_2 ^x$ or $A \leq P_2 ^x$. If $G = A P_2 ^x$, then $A$ acts transitively on the 2-dimensional subspaces of $V$, a contradiction. Hence $G = A P_2 ^x$ cannot occur and $A$ stabilizes some $2$-space $X$. Under the quotient map $V\to V/\ell$, the image of $X$ is $(X+\ell)/\ell$. This subspace of $V/\ell\cong H$ has dimension $1$ or $2$, so the image $ \bar{A} \geq \PSL(H)$ of $A$ would stabilize a proper nonzero subspace of $H$, a contradiction.

Now let $\dim H=2$, so $G=\PSL_3(q)$. If the image $C$ of $A$ on $H$ fixes a line $K\leq H$, then we extend the flag $(\ell)$ to $(\ell<\ell\oplus K)$ and continue the descent. Its full chain is $0<\ell<\ell\oplus K<V$. Otherwise let $C$ be the image of $A$ on $H$, with no invariant line. Consider the block diagonal stabilizer of the decomposition $\ell\oplus H$ in $G$. The symbol $\Hom(\ell,H)$ denotes the additive group of $\F_q$-linear maps from $\ell$ to $H$. We work in the subgroup generated by this stabilizer and the unipotent radical of $\Stab_G(H)$, written as
$$
        V_+=\Hom(\ell,H).
$$ The parabolics $\Stab_G(H)$ and $\Stab_G(\ell)$ are opposite. Their common Levi subgroup is the block diagonal stabilizer of $\ell\oplus H$ in $G$. The subgroup generated by this Levi subgroup and $V_+$ is a semidirect product. Every nonidentity element of $V_+$ (viewed as a subgroup of $G$) moves the line $\ell$, sending a spanning vector $v$ to $v+\varphi(v)$ with $\varphi(v)\in H$. Since $A$ stabilizes $\ell$, we have $A\cap V_+=1$. The induced $C$-module $V_+$ is the natural module of dimension two for $C$. Using Lemma \ref{lem:sl-root-line-module-obstruction} for a line over $\F_q$ in $V_+$ forces $C$ to fix a line of $H$, contrary to assumption.
\end{proof}

\begin{lemma}\label{lem:sl-incident-full-cover-large}
Suppose that $A \leq G$ stabilizes the flag $\ell<H$ in $V$ and the image of $A$ on $W=H/\ell$ contains $\PSL_d(q)$ with $d=n-2\geq3$. These conditions on the flag and the image of $A$ are incompatible.
\end{lemma}

\begin{proof}
We let $Y\leq V$ be an $A$-stable subspace. Because $A$ stabilizes $\ell<H$, the dimensions
$$
        \dim(Y\cap\ell),~
        \dim\bigl(((Y\cap H)+\ell)/\ell\bigr),~
        \dim\bigl((Y+H)/H\bigr)
$$
are invariant under the action of $A$ on the three factors $\ell$, $W=H/\ell$, and $V/H$. The image of $A$ on $W$ contains $\PSL_d(q)$. Under the hypotheses of the lemma, the natural module of $\PSL_d(q)$ is irreducible. Hence the subquotient $((Y\cap H)+\ell)/\ell$ is $0$ or all of $W$. The subspaces $Y\cap\ell$ and $(Y+H)/H$, corresponding to the first and third dimensions, are of dimension at most one. Hence every $A$-stable subspace of $V$ has dimension in
$$
        \{0,1,2,d,d+1,d+2\}.
$$

For a $k$-subspace stabilizer $P_k$ in $G$, Lemma \ref{lem:maximal-product} gives either an $A$-stable $k$-subspace or $G=AP_k^x$. By Lemma \ref{lem:factorization-coset-transitivity}, the equality $G=AP_k^x$ would make $A$ transitive on $k$-subspaces $X$ of $V$. Since $A$ does not change the pair $(\dim(X\cap\ell),\dim(X\cap H))$, transitivity is impossible for $2\leq k\leq n-1$. So $A$ stabilizes a $k$-subspace for every $k$ in the range $2\leq k\leq n-1$.

If $d\geq4$, then take $k=3$. Then $3\notin\{0,1,2,d,d+1,d+2\}$, a contradiction. If $d=3$, then by testing hereditary permutability of $A$ with $P_2$, we have an $A$-stable $2$-subspace $X$ (by the conclusion for $k =2$ obtained from Lemma \ref{lem:maximal-product}). If $X\leq H$, then $X$ or $(X+\ell)/\ell$ is a proper nonzero $A$-invariant subspace of $W$, impossible. If $X\not\leq H$ and $\ell\not\leq X$, then $((X\cap H)+\ell)/\ell$ is an $A$-invariant line of $W$, again impossible. The only remaining position of $X$ relative to $\ell$ and $H$ is
$$
        \ell<X\not\leq H.
$$
Under the quotient map from $V$ to $V/\ell$, the image $X/\ell$ is a line not contained in the hyperplane $H/\ell$, while the image of $A$ on $H/\ell$ contains $\PSL_3(q)$. Therefore, the line $X/\ell$ and the hyperplane $H/\ell$ in $V/\ell$ give the configuration described in Lemma \ref{lem:sl-nonincident-full-cover}, a contradiction.
\end{proof}

\begin{lemma}\label{lem:sl-incident-full-cover-d2}
Let $G=\PSL_4(q) = \PSL (V)$, and suppose that $A$ stabilizes a flag
$$
        \ell<H,~ \dim\ell=1,~ \dim H=3,
$$
with the image $\bar{A}$ of $A$ on $W=H/\ell$ stabilizing no line of $W$. These conditions on the flag and the image of $A$ are incompatible.
\end{lemma}

\begin{proof}
Choose a complement to $\ell$ in $H$. The map $H\to H/\ell$ identifies this complement with $W=H/\ell$, and we denote the complement by the same symbol $W$. We choose a line $m$ complementary to $H$ in $V$. We fix a decomposition
$$
        V=\ell\oplus W\oplus m,~ H=\ell\oplus W.
$$
Let $P=\Stab_G(\ell<H)$ be the parabolic subgroup of $G$ stabilizing the flag $\ell<H$. Inside $P$, the unipotent subgroups used in this proof are
$$
        U_-=\Hom(W,\ell),~ U_0=\Hom(m,\ell),~ U_+=\Hom(m,W),
$$
and the image of $A$ on $W$ acts on $U_-$ as the dual natural module and on $U_+$ as the natural module.

Testing hereditary permutability of $A$ in $G$ against a $2$-subspace stabilizer $P_2$ in $G$ forces $A$ to fix some $2$-subspace $X$, i.e., $A \leq P_2 ^x$ for some $x \in G$. Otherwise $G = A P_2 ^x$ and $A$ would be transitive on $2$-subspaces of $V$. This is impossible since $A$ fixes the flag $\ell<H$. By hypothesis, the image $\bar{A}$ of $A$ on $W$ has no invariant line. Each of the positions $\ell<X\leq H$ and ($X\cap H$ a line with $\ell\not\leq X$) would produce an $A$-invariant line of $W=H/\ell$. Hence neither position is possible. There are two possible positions of $X$ relative to the flag $\ell<H$.
\begin{enumerate}[label=\textup{(\arabic*)}]
\item $X\leq H$ and $\ell\not\leq X$. Then $H=\ell\oplus X$, and fixing $X$ makes the $U_-=\Hom(X,\ell)$ component of $A$ trivial. The quotient $P/(U_0U_+)$ of the flag stabilizer in $G$ by $U_0U_+$ is a semidirect product $U_-\rtimes C$. The image $C$ of $A$ on $U_-$ must be trivial. Also, $C$ has no invariant line on the dual natural module $U_-$.
\item $\ell\leq X$ and $X\not\leq H$. Then $X=\ell\oplus m'$ with a line $m'$ mapping nontrivially to $V/H$. Since $A$ fixes $X$, the $U_+=\Hom(m,W)$ component of $A$ is trivial. The quotient $P/(U_0U_-)$ of the flag stabilizer in $G$ by $U_0U_-$ is the semidirect product $U_+\rtimes C$. The image $C$ of $A$ on $U_+$ meets $U_+$ trivially. Also, $C$ has no invariant line on the natural module $U_+$.
\end{enumerate}
The module of dimension two is $U_-$ for position (1) of $X$ and $U_+$ for position (2). We choose a line over $\F_q$ in the module corresponding to the position of $X$. Lemma \ref{lem:sl-root-line-module-obstruction} then forces $C$ to fix a line of the natural or dual natural module, contradicting the hypothesis that the image of $A$ on $W$ has no invariant line.
\end{proof}

\begin{remark}
Lemma \ref{lem:sl-incident-full-cover-d2} proves the incompatibility for $d=2$, where the $\PSL_2(q)$ theorem alone is insufficient. The theorem \cite[Proposition 1]{GT2023} forbids proper hereditarily permutable subgroups of $\PSL_2(q)$. However, the image of $A$ on $W$ may equal the whole group $E$ induced by the projective action of $\Stab_G(\ell<H)$ on $W$, so that image need not be proper. Lemma \ref{lem:sl-root-line-module-obstruction}, used with a line in the corresponding natural or dual natural module, forces this image of $A$ to fix a line of $W$, contrary to the hypothesis in the dimension two case.
\end{remark}

\Needspace{9\baselineskip}
\subsection{The Borel case}

\begin{lemma}\label{lem:pgl2-borel}
Let
$$
       S= \PSL_2(q)\leq E\leq\PGL_2(q),~ C\leq B_E,
$$
where $B_E$ is a point stabilizer in $E$ on the projective line, and suppose that $E\not\cong\PGL_2(2)$. If $C$ is hereditarily $E$-permutable, then $C=1$.
\end{lemma}

\begin{proof}
For ease of reading, we divide the proof into three steps.

\begin{itemize}[leftmargin=0pt,labelwidth=4em,labelsep=0.5em,itemindent=4.5em,align=left,topsep=3pt,partopsep=0pt,parsep=0pt,itemsep=0pt]
\item[{\bf Step 1.}] $C=1$ if $q=3$ or $C\leq S$.
\end{itemize}
We let $S=\PSL_2(q)$. We begin with the case $q=3$, where $S\cong A_4$ is not simple, so the $\PSL_2(q)$ theorem does not cover this group and we argue directly. We have $E\in\{A_4,S_4\}$ and $C$ lies in a point stabilizer $B_E$ of $E$ on the projective line. If $E\cong A_4$, then $B_E\cong C_3$. For the only nontrivial candidate $C=B_E$, permuting with the subgroup generated by an involution of $A_4$ would produce a subgroup of order $6$, and $A_4$ has none. Therefore, the candidate $C=B_E$ is not hereditarily $E$-permutable. If $E\cong S_4=\PGL_2(3)$, then $B_E\cong S_3$, and we test each nontrivial subgroup of $S_3$ for hereditary $S_4$-permutability. By Lemma \ref{lem:generated-hereditary-formulation} the conjugating element may be taken in $\gen{C,H}$, and for each candidate $C$ we choose a subgroup $H$ such that no conjugate of $H$ in $\gen{C,H}$ permutes with $C$. For a candidate $C$ of order $3$ and a subgroup $H$ generated by an involution without fixed points, $\gen{C,H}=A_4$. This group has no subgroup of order $6$, so the order-$6$ product $CH^x$ is not a subgroup for any $x \in \langle C,H \rangle$. The subgroup generated by an involution fixing a point would instead lie in the same point stabilizer $S_3$ as $C$ and permute with $C$. For a candidate $C$ generated by a transposition, choose a subgroup $H$ generated by a double transposition for which $\gen{C,H}$ is a Sylow $2$-subgroup $D_8$. Since no conjugate of $H$ in $\gen{C,H}$ permutes with $C$, the subgroup $C$ fails this permutability test with $H$. Finally consider $C=B_E$ and a subgroup $H$ generated by a double transposition. We have $\gen{C,H}=S_4$ and $B_E$ contains odd permutations, while every conjugate $H^x$ meets $B_E$ trivially, so $|B_EH^x|=12$ for all $x$. The unique subgroup of $S_4$ of order $12$ is $A_4$, which has no odd permutation, so $B_EH^x$ is not a subgroup for any $x$. Comparing $B_E$ with a second point stabilizer does not test failure of hereditary permutability, since the second point stabilizer is $S_4$-conjugate to $B_E$. These checks are confirmed by the computation of Appendix \ref{app:pgl2-small-certificate}, whose enumeration includes $q=3$. Hence $C=1$ when $q=3$, and we assume that $q\geq4$ from now on.

If $q$ is even, then $E=S$, and the conclusion $C=1$ follows from the $\PSL_2(q)$ theorem \cite[Proposition 1]{GT2023}. More generally, we have $C=1$ whenever $E=S$. If $q$ is odd, then $q\geq5$ and $S=\PSL_2(q)$ is simple. If moreover $C\leq S$, then restricting the hereditary condition from $E$ to $S$ makes $C$ hereditarily $S$-permutable (Lemma \ref{lem:overgroup-restriction}), so $C=1$ by the $\PSL_2(q)$ theorem (\cite[Proposition 1]{GT2023}).

\begin{itemize}[leftmargin=0pt,labelwidth=4em,labelsep=0.5em,itemindent=4.5em,align=left,topsep=3pt,partopsep=0pt,parsep=0pt,itemsep=0pt]
\item[{\bf Step 2.}] If $q>11$, then $|C|\mid4$.
\end{itemize}
Therefore, we may assume that
$$
        q\geq5\text{ odd},~ E=\PGL_2(q),~ C\nleq S.
$$

Suppose that $q>11$. Let $D_-,D_+\leq S$ be the normalizers in $S$ of a split and of a nonsplit maximal torus of $S$, so
$$
        |D_-|=q-1,~ |D_+|=q+1.
$$
By Dickson's subgroup theorem \cite[\S260, Theorem 262]{Dickson1901}, $D_-$ and $D_+$ are maximal in $S$ for $q>11$. By hereditary $E$-permutability of $C$ with $D_-$, for some conjugate $D_-^x$, the product $H=CD_-^x$ is a subgroup. Since $C\leq B_E$, we have $|C|\leq|B_E|=q(q-1)$, so
$$
        |C|\,|D_-^x|\leq q(q-1)^2<q(q-1)(q+1)=|E|,
$$
and $H<E$. A maximal subgroup $M\geq H$ of $E$ does not contain $S$ (the subgroup $M$ contains an element outside $S$), so $M\cap S=D_-^x$ by maximality of $D_-^x$ in $S$. As $S\unlhd E$, $M\cap S\unlhd M$, whence
$$
        C\leq M\leq N_E(D_-^x),~ |C|\mid 2(q-1).
$$
Repeating the hereditary permutability test for $C$ with $D_+$ and using the order bound for $C$ from $|C|\mid 2(q-1)$, we obtain
$$
        |C|\,|D_+^y|\leq 2(q-1)(q+1)<|E|,
$$
so $CD_+^y<E$ and $C\leq N_E(D_+^y)$, and hence $|C|\mid 2(q+1)$. Therefore,
$$
        |C|\mid\bigl(2(q-1),2(q+1)\bigr)=4.
$$

\begin{itemize}[leftmargin=0pt,labelwidth=4em,labelsep=0.5em,itemindent=4.5em,align=left,topsep=3pt,partopsep=0pt,parsep=0pt,itemsep=0pt]
\item[{\bf Step 3.}] $C=1$ in all cases.
\end{itemize}
Since $C\leq B_E\cong q:(q-1)$ and $|C|$ is prime to the characteristic $q$, $C$ fixes a point of the affine line and so is conjugate inside $B_E$ into the split torus of $E$ contained in $B_E$. As $C\nleq S$, the only possibilities for $C$ are a subgroup generated by an involution outside $S$ or a cyclic group of order $4$ not contained in $S$.

Suppose that $|C|=2$. After conjugating $C$ in $E$ we may take $C=\gen{c}$ with $c\colon z\mapsto-z$. With
$$
        h\colon z\mapsto\frac{z+1}{z-1},~ H_0=\gen{h},~ F=\gen{c,h},
$$
we have $F\cong D_8$, with the generators $c$ and $h$ in the two different conjugacy classes of reflections. No $F$-conjugate of $H_0$ equals or commutes with $C$. Two subgroups of order $2$ permute if and only if their nonidentity elements commute, so $C$ is not hereditarily $F$-permutable, a contradiction.

Suppose that $|C|=4$. Since $q$ is odd, the unipotent radical of $B_E$ has odd order, so $C$ embeds in the cyclic split torus of $E$ in $B_E$ of order $q-1$. Therefore, $4\mid q-1$. Let $r\colon z\mapsto iz$ with $i$ of order $4$ and $i^2=-1$. Then $C=\gen{r}$ is the unique order-$4$ subgroup of this torus. As $C\nleq S$, the generator $i$ is a nonsquare in $\F_q^\times$ (a square generator would mean that $C\leq S$). We consider this cyclic subgroup of order four only when an element $i$ of order four that is a nonsquare in $\F_q^\times$ exists. With the element $h$ acting by $z\mapsto\frac{z+1}{z-1}$, the group $F=\gen{r,h}$ is the subgroup $S_4$ of $\PGL_2(q)$. Let $H_0\leq F$ have order $3$. If $CH_0^x=H_0^xC$ for some $x\in F$, then $CH_0^x$ is a subgroup of $S_4$ of order $12$ containing an element of order $4$. However, the only subgroup of $S_4$ of order $12$ is $A_4$, which has no element of order $4$. This contradiction settles $|C|=4$.

Finally, for the field sizes $q\in\{5,7,9,11\}$ we verify $C=1$ by a direct computation in $\PGL_2(q)$. We enumerate the subgroups of a point stabilizer in $E$. For each nontrivial candidate $C$, the computation finds an element $h\in E$ and tests $C$ against the cyclic subgroup $\gen{h}$ inside $\gen{C,h}$. No $x\in\gen{C,h}$ satisfies $C\gen{h}^x=\gen{h}^xC$. Consequently, every nontrivial candidate $C$ fails to be hereditarily $E$-permutable in $E$. The subgroup $\gen{C,h}$ is generated by $C$ and this chosen element $h$ (Appendix \ref{app:pgl2-small-certificate}). Hence $C=1$ for every field size $q$ covered by the lemma.
\end{proof}

\begin{lemma}\label{lem:sl-borel}
Let
$$
        \PSL_m(q)\leq E\leq\PGL_m(q),~ m\geq2,
$$
and let $B_E$ be the normalizer in $E$ of the upper triangular Borel subgroup of the socle of $E$. Apart from $(m,q,E)=(2,2,\PGL_2(2))$, if $C\leq B_E$ is hereditarily $E$-permutable, then $C=1$.
\end{lemma}

\begin{proof}
We induct on $m$. For $m=2$, the conclusion follows from Lemma \ref{lem:pgl2-borel}, apart from the $\PGL_2(2)$ case. The case $m=3$, $q=2$ is handled by the isomorphism $\PSL_3(2)\cong\PSL_2(7)$ \cite[p. 3]{Wilson2009} together with the $\PSL_2(7)$ theorem \cite[Proposition 1]{GT2023}, which avoids the (false) $\PGL_2(2)$ case.

Assume that $m\geq3$ and $(m,q)\neq(3,2)$. The parabolic subgroups $P_1$ and $P_{m-1}$ of $E$ stabilize a line and a hyperplane, respectively. Choose an $\F_q$-basis $(e_1,\ldots,e_m)$ of the natural space $\F_q^m$ in which $B_E$ is upper triangular. Write $Q_1=\F_q^m/\gen{e_1}$ and $Q_{m-1}=\gen{e_1,\ldots,e_{m-1}}$. The lower right block acts on $Q_1$, and the upper left block acts on $Q_{m-1}$. Let $\rho_i:P_i\to\PGL(Q_i)$ be these projective actions, and let $E_i=\rho_i(P_i)$ for $i\in\{1,m-1\}$. We take the images of $C$ under the projective actions of these subgroups on the lower right and upper left blocks, respectively. For each of $P_1$ and $P_{m-1}$, we quotient the parabolic subgroup by the elements acting as scalars on the corresponding block. These projective factors of the Levi quotients are projective linear of dimension $m-1$ \cite[Proposition 26.1(b)]{MalleTesterman2011}. By Lemma \ref{lem:quotient-inheritance}, the images $\rho_i(C)$ of $C$ in these quotients are hereditarily $E_i$-permutable for $i\in\{1,m-1\}$. By induction, both images of $C$ in the projective actions on the lower right and upper left $(m-1)\times(m-1)$ blocks are trivial. Hence the diagonal part of $C$ and its components in all root subgroups except the subgroup corresponding to the highest root are trivial. In upper triangular terms,
$$
        C\leq U_{1m}.
$$
If $q\geq4$, then the opposite root subgroups $U_{1m}$ and $U_{m1}$ generate a subgroup $R\cong\SL_2(q)$ \cite{Carter1985}. A nontrivial scalar in the centre of $\SL_m(q)$ acts nontrivially on the middle coordinates $2,\dots,m-1$ (nonempty as $m\geq3$). By contrast, the group $R$ fixes these middle coordinates, so the embedding of $R$ in $E$ is faithful. Since $C\leq U_{1m}$ is unipotent, $C\cap Z(R)=1$. By Lemmas \ref{lem:overgroup-restriction} and \ref{lem:quotient-inheritance}, the image $\bar C$ of $C$ in $R/Z(R)\cong\PSL_2(q)$ is hereditarily $\PSL_2(q)$-permutable and isomorphic to $C$. The $\PSL_2(q)$ theorem \cite[Proposition 1]{GT2023} (valid as $q\geq4$) shows that $\bar C=1$, hence $C=1$. There remain $q\in\{2,3\}$, where $\PSL_2(q)$ is not simple. If $q=2$, then $m\geq4$. Pick $j$ with $1<j<m$. The root subgroups on coordinates $1,j,m$ generate a subgroup with simple quotient $\PSL_3(2)\cong\PSL_2(7)$ \cite[p. 3]{Wilson2009}. Lemma \ref{lem:overgroup-restriction} and the $\PSL_2(7)$ theorem \cite[Proposition 1]{GT2023} then show that $C=1$. If $q=3$, then $\bar C$ is the image of $C$ obtained by projecting through the parabolic subgroups $P_1$ and $P_{m-1}$ of $E$. In this case, $\bar C$ is a hereditarily $(R/Z(R))$-permutable subgroup of $R/Z(R)\cong\PSL_2(3)\cong A_4$. This image has order dividing $3$. Were $\bar C\neq1$, the subgroup $\bar C$ would permute with the subgroup generated by an involution of $A_4$ to produce a subgroup of order $6$, which $A_4$ lacks. Hence $\bar C=1$ and $C=1$.
\end{proof}

\subsection{Proof of \texorpdfstring{Theorem \ref{thm:split-linear}}{the linear theorem}}

\begin{proof}[Proof of Theorem \ref{thm:split-linear}]
Let $A<G=\PSL_n(q)$ be proper, nontrivial, and hereditarily $G$-permutable. For $n=2$, we use the known theorem for linear groups of dimension two \cite[Proposition 1]{GT2023}, so we assume that $n\geq3$. By Proposition \ref{prop:sl-parabolic-containment}, $A$ lies in a proper parabolic subgroup of $G$. Lemma \ref{lem:sl-line-hyperplane} gives, after relabelling, a line $\ell$ and a hyperplane $H$ in $V$ stabilized by $A$.

Suppose that $\ell\not\leq H$, so $V=\ell\oplus H$ and the image $C$ of $A$ on $H$ is hereditarily $E$-permutable in a group $\PSL_{n-1}(q)\leq E\leq\PGL_{n-1}(q)$. For a flag $\mathcal F_H=(H_1<\cdots<H_t)$ in $H$ fixed by $C$, the corresponding flag in $V$ extends $(\ell)$ and has full chain $0<\ell<\ell\oplus H_1<\cdots<\ell\oplus H_t<V$. For each additional term $K$ of the flag in $H$, the corresponding additional term in $V$ is $\ell\oplus K$. If $n-1\geq3$, then we use Lemma \ref{lem:sl-actual-quotient-trichotomy}. The case in which $C$ contains the socle of $E$ cannot occur by Lemma \ref{lem:sl-nonincident-full-cover}. When $C$ lies in the normalizer in $E$ of a proper parabolic subgroup of the socle of $E$, we extend the flag. Suppose that $C$ lies in no such normalizer in $E$ and does not contain the socle. Testing $C$ through the normalizer in $E$ of a Borel subgroup of the socle of $E$ gives a contradiction in this case. This follows the argument in Proposition \ref{prop:sl-parabolic-containment}. Let $\mathcal F_H$ be the extended flag in $H$, and let $P_H=\Stab_E(\mathcal F_H)$. Its Levi quotient is $L_H=P_H/R_u(P_H)$. When the image of $C$ in each simple factor of the semisimple part of $L_H$ is trivial, we use Lemma \ref{lem:sl-borel}. If $n-1=2$, then we extend the flag when $C$ lies in the normalizer in $E$ of a proper parabolic subgroup of the socle of $E$. An image $C$ with no invariant line is impossible by the argument in dimension two of Lemma \ref{lem:sl-nonincident-full-cover}. When the image of $C$ in every simple factor of this Levi quotient is trivial, we again use Lemma \ref{lem:sl-borel}.

Next suppose that $\ell\leq H$, and let $W=H/\ell$. For the quotient map $\tau:H\to W$, a flag $W_1<\cdots<W_t$ in $W$ fixed by the image of $A$ gives the extended flag $\ell<\tau^{-1}(W_1)<\cdots<\tau^{-1}(W_t)<H$ in $V$. This extends the original flag $(\ell<H)$ by inserting the inverse images of the new quotient subspaces. If $n=3$, then $A$ stabilizes the complete flag $0<\ell<H<V$, whose nonzero terms have dimensions $1,2,3$. Consequently, it lies in a Borel subgroup of $G$ and must be trivial by Lemma \ref{lem:sl-borel}. So assume that $n-2\geq2$. The image $C$ of $A$ on $W$ is hereditarily $E$-permutable in a group $\PSL_{n-2}(q)\leq E\leq\PGL_{n-2}(q)$. If $n-2\geq3$, then we use Lemma \ref{lem:sl-actual-quotient-trichotomy}. The possibility that $C$ contains the socle of $E$ contradicts Lemma \ref{lem:sl-incident-full-cover-large}. When $C$ lies in the normalizer in $E$ of a proper parabolic subgroup of the socle of $E$, we extend the flag. Suppose that $C$ lies in no such normalizer and does not contain the socle. We test $C$ through the normalizer in $E$ of a Borel subgroup of the socle of $E$. Let $\mathcal F_W$ be the extended flag in $W$, and let $P_W=\Stab_E(\mathcal F_W)$. Its Levi quotient is $L_W=P_W/R_u(P_W)$. When the image of $C$ in every simple factor of the semisimple part of $L_W$ is trivial, we use Lemma \ref{lem:sl-borel}. If $n-2=2$, then we proceed as follows. When $C$ lies in the normalizer in $E$ of a proper parabolic subgroup of the socle of $E$, we extend the flag. Every image of $A$ lying in no such normalizer and having no invariant line contradicts Lemma \ref{lem:sl-incident-full-cover-d2}. If the image of $C$ in every simple factor of this Levi quotient is trivial, then $A$ must again be trivial by Lemma \ref{lem:sl-borel}.

At each step at which the descent continues, the dimension of the space on which the linear Levi quotient at that step acts strictly decreases. Therefore, the descent reaches either the $\PSL_2(q)$ base \cite[Proposition 1]{GT2023} or Lemma \ref{lem:sl-borel}. In either case, these results forbid a proper nontrivial subgroup with the hereditary permutability property in the image of the projective action at that step of the descent. This contradiction proves the theorem.
\end{proof}
\section{Unitary groups}\label{sec:unitary}

The natural unitary space $V$ has dimension $n$ over $\F_{q^2}$. We write $h(u,v)=\langle u,v\rangle$ for its nondegenerate Hermitian form, which is linear in the first variable and satisfies $h(v,u)=h(u,v)^q$. For a subspace $X\leq V$, the orthogonal complement is
$$
        X^\perp=\{v\in V:h(v,x)=0\text{ for every }x\in X\}.
$$
A subspace $X$ is \emph{nondegenerate} if $X\cap X^\perp=0$. By contrast, $X$ is \emph{totally isotropic} if $h(x,y)=0$ for all $x,y\in X$. Following Kleidman and Liebeck \cite[\S2.1]{KleidmanLiebeck1990}, we also call a totally isotropic subspace \emph{totally singular}. In particular, a line $\gen{v}$ is singular if $h(v,v)=0$.

A \emph{totally singular flag} is a flag whose terms are totally singular. The \emph{Witt index} $r$ of $V$ is the largest dimension of a totally singular subspace. For this unitary space, $r=\lfloor n/2\rfloor$. A complete totally singular flag has the form $0=F_0<F_1<\cdots<F_r$, with $\dim F_i=i$ for $0\leq i\leq r$. In particular, its last term $F_r$ is a maximal totally singular subspace of $V$. For a singular line $\ell$, we have $\ell\leq\ell^\perp$. The quotient map sends $v\in\ell^\perp$ to the coset $v+\ell$. For $u,v\in\ell^\perp$ and $z_1,z_2\in\ell$, the equalities $h(u,z_2)=h(z_1,v)=h(z_1,z_2)=0$ give $h(u+z_1,v+z_2)=h(u,v)$. Consequently, $h_\ell(u+\ell,v+\ell)=h(u,v)$ defines the induced Hermitian form on $\ell^\perp/\ell$. A subgroup stabilizes this flag if its elements $g$ satisfy $g(F_i)=F_i$ for $0\leq i\leq r$. Individual vectors in $F_i$ need not be fixed by $g$.

An orthogonal decomposition $V=L_1\perp\cdots\perp L_n$ into nondegenerate lines defines a \emph{self-polar simplex} consisting of $L_1,\ldots,L_n$. Choose a basis $e_1,\ldots,e_n$ with $L_i=\gen{e_i}$. A monomial transformation sends $e_i$ to $\lambda_i e_{\sigma(i)}$, with $\lambda_i\in\F_{q^2}^\times$ and $\sigma\in S_n$. The permutation $\sigma$ acts on the index set $\{1,\ldots,n\}$. In $\PSU_n(q)$ or $\PGU_n(q)$, the subgroup fixing every $L_i$ as a subspace is the diagonal subgroup for this decomposition. A subgroup stabilizing the simplex also induces permutations of $\{L_1,\ldots,L_n\}$. The group of these permutations is the permutation image of that subgroup.

The proof uses self-polar simplices and the permutation image of $A$ on their lines. We first consider the cases of low rank, the toral and unipotent cases, and the permutation images of $A$ on the simplex lines. We then use the results for odd composite dimension and even dimension. For each projective action, we test hereditary permutability of the image of $A$ in the full image of that action.

\begin{theorem}\label{thm:unitary}
For the simple group $G=\PSU_n(q)$ there is no proper nontrivial hereditarily $G$-permutable subgroup.
\end{theorem}

A simple unitary group $G=\PSU_n(q)$ has no proper nontrivial ordinary $G$-permutable subgroup \cite[Corollary 4.17]{ZhangOrdinaryPreprint}. Theorem \ref{thm:unitary-odd-prime-noncentral} includes an intermediate result on ordinary $G$-permutable subgroups. By the Levi quotient of a parabolic subgroup we mean the quotient of that parabolic subgroup by its unipotent radical. When we specify a projective action of a Levi factor on a subspace or quotient space, we use the image of that action. Equivalently, we quotient that Levi factor by the elements whose restriction to the specified subspace or quotient space is scalar.

\subsection{Base cases of low rank}

\begin{theorem}\label{thm:unitary-low-bases}
We use three results of low rank.
\begin{enumerate}[label=\textup{(\arabic*)}]
\item For $Q\geq4$, $\PSL_2(Q)$ has no proper nontrivial hereditarily $\PSL_2(Q)$-permutable subgroup.
\item For $Q\geq4$, a proper ordinary $\PSL_2(Q)$-permutable subgroup of $\PSL_2(Q)$, if such an ordinary permutable subgroup exists, has order $2$.
\item For $Q\geq3$, $\PSU_3(Q)$ has no proper nontrivial hereditarily $\PSU_3(Q)$-permutable subgroup.
\end{enumerate}
\end{theorem}

\begin{proof}
The nonexistence of proper nontrivial hereditarily $\PSL_2(Q)$-permutable subgroups is proved in \cite[Proposition 1]{GT2023}. For $Q>4$, the conclusion about subgroups of order two for ordinary $\PSL_2(Q)$-permutable subgroups is \cite[Theorem 1]{GT2025}. For $Q=4$, we use $\PSL_2(4)\cong A_5$, for which the examples of subgroups of order two are described in \cite[Example 1]{GT2022}. The nonexistence of proper nontrivial hereditarily permutable subgroups of $\PSU_3(Q)$ for $Q\geq3$ is proved in \cite[Theorem 2]{GT2025}. (The inheritance of hereditary permutability by quotient images is \cite[Lemma 5]{GT2025}. Lemma \ref{lem:quotient-inheritance} provides the short proof we use.)
\end{proof}

We also use results on self-polar simplex stabilizers, Singer normalizers, permutation groups, and factorizations with references stated where each result is used.

\emph{Stabilizers of self-polar simplices.} For a self-polar simplex in the projective unitary group, let $D$ be its projective diagonal unitary torus and let $S_m$ permute its lines. By Dye's theorem \cite[Theorem 5]{Dye1985}, the stabilizer of this simplex is maximal and has the form $D:S_m$.

\emph{Singer normalizers.} Consider a unitary Singer torus corresponding to a partition with one part in $\PSU_r(q)$, where $r\geq5$ is an odd prime and $(r,q+1)=1$. This torus has order $(q^r+1)/(q+1)$. Its normalizer in $\PSU_r(q)$, modulo this torus, has order $r$. By Bereczky's theorem \cite[Main Theorem]{Bereczky2000}, the maximal overgroups of this torus are the Singer normalizers, apart from the single boundary case $\PSU_5(2)$. That theorem rests on the analysis of primitive prime divisors by Guralnick, Penttila, Praeger and Saxl \cite[Main Theorem]{GPPS1999}. (For $r=3$ the base case in dimension three is treated separately in Theorem \ref{thm:unitary-low-bases}.)

\emph{Permutation groups.} We use Burnside's theorem \cite[Theorem 3.5B]{DixonMortimer1996} on transitive groups of prime degree. We also use the theorems of Livingstone and Wagner \cite[Theorem 2]{LivingstoneWagner1965} and Kantor \cite[Theorem 1]{Kantor1972} on $k$-homogeneous permutation groups. Further results are the transitive groups of degree seven from Butler and McKay \cite[Table 7A]{ButlerMcKay1983} and the theorem of Yang and Galt \cite[Theorem 2]{YG2023} on $G$-permutable subgroups of $A_p$ for $p$ prime. For subset actions and setwise stabilizers we follow Dixon and Mortimer \cite[\S2.1]{DixonMortimer1996}.

\emph{Factorization classifications.} We use Li and Xia \cite{LiXia2022} for factorizations with a solvable factor and Li, Wang, and Xia \cite{LiWangXia2024} for the factorizations of the finite classical groups, in each case only after permutability has produced the corresponding factorization.

\begin{remark}
We consider the permutation image of $A$ on the lines of a self-polar simplex. In this argument, we use the \emph{ordinary} theorem of Yang and Galt on $A_p$-permutable subgroups of alternating groups of prime degree. The theorem of Yang and Galt applies to the intersection of the permutation image of $A$ with the alternating group of prime degree. Indeed, Lemma \ref{lem:unitary-alt-transfer} establishes ordinary $A_p$-permutability of this intersection. The theorem on hereditary permutability in alternating groups is used in the proof of the main theorem, after the four classical families have been treated.
\end{remark}

\subsection{Subgroups of a diagonal torus}

\begin{lemma}\label{lem:unitary-weyl-normalization}
Let $E=\bar D\rtimes S_m$ be the quotient of a self-polar simplex stabilizer over $\F_{Q^2}$ by the kernel of its projective monomial action on the unitary space spanned by the simplex lines, with $m\geq3$. Let $\bar A\leq\bar D$ be hereditarily $E$-permutable. Then $\bar A$ is normalized by the group $S_m$ acting by permutations of the coordinate lines.
\end{lemma}

\begin{proof}
It suffices to show that a transposition $(ij)$ normalizes $\bar A$. As $m \geq 3$, choose a third coordinate $k$. On the three corresponding anisotropic lines take a monomial unitary element $w$. This element swaps the $i$-line and $j$-line, multiplies the $k$-line by $-1$, and fixes the remaining coordinate lines. In odd characteristic the swap has determinant $-1$ and the extra factor $-1$ restores determinant $1$. In characteristic $2$ the swap has determinant $1$. The element $w$ represents an element of the projective special unitary group with $w^2=1$ inducing $(ij)$.

In this paragraph, $E=\bar D\rtimes\gen{w}$ denotes the subgroup generated by $\bar D$ and $w$ inside $\bar D\rtimes S_m$. By hereditary permutability, $\bar A$ permutes with a conjugate of $\gen{w}$ in $E$. Write a generator of the chosen conjugate of $\gen{w}$ as $y=tw$ with $t\in\bar D$. Since $\bar A\gen{y}$ is a subgroup, for each $a\in\bar A$ the element $ya$ equals $a'y$ for some $a'\in\bar A$ (compare the nontrivial coset of $\bar D$ in $E/\bar D$). Therefore, $y\bar A y^{-1}=\bar A$. As $t\in\bar D$ centralizes the diagonal subgroup $\bar A$, we get $w\bar Aw^{-1}=\bar A$. Since coordinate transpositions generate $S_m$, this proves the lemma.
\end{proof}

\begin{lemma}\label{lem:unitary-rank-two-ratio-bound}
In the notation of {\rm Lemma \ref{lem:unitary-weyl-normalization}}, represent an element of $\bar A$ on the coordinate lines by
$$
        \diag(\lambda_1,\ldots,\lambda_m),~ \lambda_i^{Q+1}=1.
$$
If $Q\geq3$, then $(\lambda_i/\lambda_j)^4=1$ for every pair $i,j$.
\end{lemma}

\begin{proof}
We fix $i\neq j$ and let $s=(ij)$, which normalizes $\bar A$ by Lemma \ref{lem:unitary-weyl-normalization}. For $a\in\bar A$ let $\bar c=a^{-1}a^s\in\bar A$. On the unitary subspace of dimension two $U=L_i\perp L_j$, the image of $\bar c$ in $\PGU(U)$ is $\diag(\delta,\delta^{-1})$ with $\delta=\lambda_j\lambda_i^{-1}$. This image of $\bar c$ lies in $N=\PSU(U)\cong\PSL_2(Q)$ (using $\PSU_2(Q)\cong\PSL_2(Q)$ \cite[p. 66]{Wilson2009}), and the projective order of this image equals the multiplicative order of $\delta^2$.

Let $\bar A_U$ be the image of $\bar A$ in $\PGU(U)$. Let $J_U$ denote the stabilizer of the orthogonal decomposition $U\perp U^\perp$ used for the hereditary condition of $\bar{A}$. We restrict this condition to $J_U$. Let $K_U$ be the complementary unitary factor of $J_U$ on $U^\perp$, acting on $U^\perp$ and trivially on $U$. To study the action on $U$, we take the quotient $J_U/K_U$ and use its induced projective action $\rho_U:J_U/K_U\to\PGU(U)$ on $U$. By Lemma \ref{lem:quotient-inheritance}, the image $\bar A_U$ satisfies the hereditary permutability condition in the group $\rho_U(J_U/K_U)$ acting on the space $U$ of dimension two. Applying Lemma \ref{lem:normal-intersection-transfer} to $\bar A_U$ for $N\unlhd\PGU(U)$ shows that the subgroup $K=\bar A_U\cap N$ is ordinary $N$-permutable. If $Q>4$, then Theorem \ref{thm:unitary-low-bases} shows that $|K|\leq2$ when $K\neq1$. Since $K$ lies in the diagonal torus of $\PGU(U)$ determined by the two coordinate lines, $K$ is abelian, hence proper in $N$. Therefore, the image of $\bar c$ in $\PGU(U)$ has order at most $2$ and $\delta^4=1$. If $Q=4$, then $N\cong A_5$, and a nontrivial subgroup of this diagonal torus of $\PGU(U)$ has order $5$. Testing the nontrivial subgroup in the diagonal torus of $\PGU(U)$ against a subgroup of order $3$ would require a subgroup of order $15$ in $A_5$. Since $3\nmid4$, this subgroup would be cyclic and contain an element of order $15$. However, such an element does not exist in $A_5$. So $\bar c=1$ and $\delta^4=1$. If $Q=3$, then the scalars of norm one have order $Q+1=4$, so $\delta^4=1$ directly.
\end{proof}

\begin{lemma}\label{lem:unitary-rank-three-toral-collapse}
In the notation of {\rm Lemma \ref{lem:unitary-weyl-normalization}}, the projective image of $\bar A$ on the span of any three coordinate lines is trivial.
\end{lemma}

\begin{proof}
If $Q=2$, then we consider each nontrivial subgroup of the projective diagonal torus of $\PGU_3(2)$. A direct computation in $\PGU_3(2)$ finds a cyclic subgroup $H$ for each chosen toral subgroup. For this $H$, no conjugate of $H$ permutes with that toral subgroup inside the group generated by $H$ and the toral subgroup (Appendix \ref{app:pgu3-2-toral-certificate}). Hence no such toral subgroup is hereditarily $\PGU_3(2)$-permutable, and the projective image of $\bar{A}$ on the span of the three coordinate lines is trivial. Therefore, we may assume that $Q\geq3$.

By Lemma \ref{lem:unitary-rank-two-ratio-bound}, the projective image of $\bar A$ on the span of any three coordinate lines is a $2$-group. By \cite[p. 66]{Wilson2009}, for a unitary space of dimension three over $\F_{Q^2}$,
$$
        \PGU_3(Q)/\PSU_3(Q)\cong C_{(3,Q+1)}
$$
(the quotient being trivial when $(3,Q+1)=1$), so the diagonal $2$-group acting on the three coordinates has trivial image in this quotient and lies in $\PSU_3(Q)$. By Lemma \ref{lem:overgroup-restriction} the projective image of $\bar A$ on the span of the three coordinate lines is hereditarily $\PSU_3(Q)$-permutable in $\PSU_3(Q)$. Since this image is a diagonal $2$-group, it is proper. Hence the result for $\PSU_3$ in Theorem \ref{thm:unitary-low-bases} shows that this image on three coordinates is trivial.
\end{proof}

\begin{proposition}\label{prop:unitary-toral}
Let $A$ be a hereditarily permutable subgroup contained in the projective diagonal torus of a self-polar simplex stabilizer $\Delta=D:S_m$ over $\F_{Q^2}$, with $m\geq3$. Then $A=1$.
\end{proposition}

\begin{proof}
We let $\bar A\leq\bar D$ be the image of $A$ under the projective monomial action on the space spanned by the $m$ simplex lines. By Lemma \ref{lem:unitary-weyl-normalization} the image $\bar A$ is normalized by $S_m$, and by Lemma \ref{lem:unitary-rank-three-toral-collapse} the projective image of $\bar A$ on the span of any three coordinate lines is trivial. Hence all coordinate ratios in $\bar A$ are trivial and $\bar A=1$. Projectively, $A=1$.
\end{proof}

\begin{remark}
The diagonal group $\bar D$ being abelian does not by itself show that the group $S_m$ acting by permutations of the coordinate lines normalizes $\bar A$. We prove this normalization by testing hereditary permutability of $\bar A$ against the cyclic subgroup generated by $w$ inside $\bar D\rtimes\gen{w}$. The involution $w$ constructed in the proof of Lemma \ref{lem:unitary-weyl-normalization} swaps the two chosen coordinate lines. Since $\bar A$ is diagonal, the conjugation calculation in this semidirect product shows that $w$ normalizes $\bar A$.
\end{remark}
\subsection{Unipotent subgroups fixing a nondegenerate line}

\begin{lemma}\label{lem:unitary-singer-line-orbit-stabilizer}
Let $F=\F_Q$, $E=\F_{Q^u}$, and let $\lambda\in E^\times$ have prime order $s$ with $\ord_s(Q)=u$. If an $F$-linear map $T\in\GL_F(E)$ stabilizes every line $F\lambda^i$ for $0\leq i<s$, then $T$ is scalar over $F$.
\end{lemma}

\begin{proof}
Write $T$ as a linearized polynomial $T(X)=\sum_{k=0}^{u-1}a_kX^{Q^k}$ with $a_k\in E$. The condition that $T$ stabilizes the line $F\lambda^m$ means that $T(\lambda^m)=c_m\lambda^m$ with $c_m\in F$, and since $\lambda^s=1$ the scalar $c_m$ depends only on $m\pmod s$. Dividing by $\lambda^m$, we have, for every $m\in\Z/s$,
$$
        \sum_{k=0}^{u-1}a_k\,\lambda^{me_k}=c_m,~ e_k:=Q^k-1\!\!\pmod s .
$$
If $e_j=e_k$ with $j<k$, then $Q^{k-j}\equiv1\pmod s$. This congruence, together with $\ord_s(Q)=u$, forces $u\mid k-j$. Therefore, the residues $e_0=0,e_1,\dots,e_{u-1}$ are pairwise distinct. We view each side of $\sum_{k=0}^{u-1}a_k\,\lambda^{me_k}=c_m$ as a function $\Z/s\to E$ and expand both functions in the characters $\chi_e\colon m\mapsto\lambda^{me}$ $(e\in\Z/s)$. Write $p=\operatorname{char}F$ and let $\mu_s$ be the group of $s$th roots of unity. Since $s\neq p$ and $\mu_s\subseteq E^\times$, these characters are linearly independent over $E$. The identity reads $\sum_k a_k\chi_{e_k}=c$, where $c$ is the function $m\mapsto c_m$ on $\Z/s$. The set of exponents with nonzero coefficients is $\Sigma:=\{e_k:a_k\neq0\}$.

We now use the Frobenius map $x\mapsto x^Q$ on the character identity $\sum_k a_k\chi_{e_k}=c$. Since $\chi_e^Q=\chi_{Qe}$ and each $c_m\in F$ satisfies $c_m^Q=c_m$, the character identity becomes $\sum_k a_k^Q\chi_{Qe_k}=c$. We compare the nonzero coefficients in the two expressions for $c$. In this comparison, character independence and $a_k^Q\neq0\iff a_k\neq0$ give $Q\Sigma=\Sigma$. Hence $\Sigma$ is a union of $\gen{Q}$-orbits on $\Z/s$. The subgroup $\gen{Q}\leq(\Z/s)^\times$ has order $u$ and acts freely on the nonzero residues, so each orbit on the residues different from $0$ has size $u$. But $\Sigma\setminus\{0\}\subseteq\{e_1,\dots,e_{u-1}\}$ has at most $u-1$ elements. This is too few to contain an entire orbit on the nonzero residues. Therefore, $\Sigma\subseteq\{0\}$, i.e.\ $a_k=0$ for all $k\geq1$. Finally $T(1)=a_0=c_0\in F$, so $T=a_0I$ is scalar over $F$.
\end{proof}

\begin{proposition}\label{prop:unitary-unipotent}
Let $G=\PSU_u(Q)$ with $Q=p^f$ and $u\geq3$ odd, $(u,Q)\neq(3,2)$. Suppose that $A<G$ is proper, nontrivial, and hereditarily $G$-permutable, is a $p$-subgroup, and fixes a nondegenerate line $L_0$ of the natural module pointwise. Then no such $A$ exists.
\end{proposition}

\begin{proof}
Since $L_0$ is nondegenerate, $V=L_0\perp L_0^\perp$ is an $A$-invariant orthogonal decomposition of $V$, and $A$ fixes $L_0$ pointwise. Hence each nonidentity element of $A$, being unipotent and acting as the identity on the direct summand $L_0$, has a Jordan block of size $1$ on the natural unitary module $V$.

Choose a primitive prime divisor $s$ of $Q^{2u}-1$. Since $u\geq3$ and $(u,Q)\neq(3,2)$, such a divisor exists by Zsigmondy's theorem \cite{Zsigmondy1892}. This prime satisfies $\ord_s(Q)=2u$ and $s\equiv1\pmod{2u}$. Let $R=\gen{r}$ be the subgroup of order $s$ in an irreducible unitary Singer torus of $G$ acting on the natural module $V$ \cite[Table 1, Main Theorem]{Bereczky2000}. Since $A$ is a $p$-group and $s\neq p$, we have $A\cap R=1$.

By hereditary $G$-permutability of $A$, we choose a conjugate of $R$, again written $R$, with $H=AR$ a subgroup. Since $A\cap R=1$, the action of $H$ on the cosets of $A$ has prime degree $s$ and $R$ is regular on these cosets. We let $C=\Core_H(A)$. Because $A$ is a $p$-group and $|R|=s$ is a prime distinct from $p$, $|H|=|A|\,s=p^a s$, so $H$ is solvable by Burnside's $p^aq^b$ theorem \cite[Theorem 3.10]{DixonMortimer1996}. In the coset action of $H/C$, a solvable transitive group of prime degree $s$ has a normal Sylow $s$-subgroup \cite[Theorem 3.5B]{DixonMortimer1996}, so $RC/C\unlhd H/C$. In particular the nonsolvable possibility does not occur, and only two cases remain.

\emph{$R\unlhd H$.} Let $y\in A$ have order $p$. If $y$ centralized $R$, then $y$ would lie in the centralizer of the element $r$. The element $r$ acts irreducibly on the natural module, and its centralizer consists of the transformations acting by multiplication in $\F_{Q^{2u}}$, identified with the natural module $V$ over $\F_{Q^2}$. The unitary transformations in this centralizer form a torus of elements of norm one, a $p'$-group \cite[Table 1]{Bereczky2000}. Hence $y$ acts nontrivially on $R$,
$$
        yry^{-1}=r^a,~ a\in(\Z/s\Z)^\times,~ |a|=p.
$$
The eigenvalues of $r$ form one orbit under the map $\lambda\mapsto\lambda^{-Q}$,
$$
        \lambda,\lambda^{-Q},\ldots,\lambda^{(-Q)^{u-1}},
$$
and the eigenvalue multiset of $r^a$ must equal the eigenvalue multiset of $r$. So exponentiation by $a$ sends $\lambda$ into the $\gen{-Q}$-orbit of $\lambda$, and unitary semisimple conjugacy forces
$$
        a\in\gen{-Q},~ |\gen{-Q}|=u.
$$
Equivalently, $N_G(R)/C_G(R)$ is cyclic of order $u$, generated by the map on eigenvalues $\lambda\mapsto\lambda^{-Q}$ \cite[Main Theorem]{Bereczky2000}. Since the central scalars lie in $C_G(R)$, taking the quotient of $\SU_u(Q)$ by $Z(\SU_u(Q))$ to obtain $\PSU_u(Q)$ leaves $N_G(R)/C_G(R)$ cyclic of order $u$ in the projective group $G$, with generator induced by the map on eigenvalues $\lambda\mapsto\lambda^{-Q}$. Therefore, the automorphism induced by $y$ on $R$ lies in $\gen{-Q}$. As $|a|=p$ and $\gen{-Q}$ is cyclic of order $u$, we have $p\mid u$. If $p\nmid u$, then there is no element $a$ of order $p$ in $\gen{-Q}$, so $R\not\unlhd H$. The element $y$ permutes the eigenlines of $r$ in cycles of length $p$. On the span of each cycle the matrix of $y$ is a cyclic permutation matrix. In characteristic $p$ this matrix is a single Jordan block of size $p$, not of size $1$. So $y$ has no Jordan block of size $1$, contradicting the fact that $A$ fixes $L_0$ pointwise. Hence $R\not\unlhd H$.

\emph{$RC/C\unlhd H/C$ but $R\not\unlhd H$.} Since $H=AR$, the core of $A$ in $H$ is
$$
        C=\bigcap_{i=0}^{s-1}A^{r^i}.
$$
As $A$ fixes the line $L_0$, the subgroup $C$ stabilizes every line in the $R$-orbit $L_0,L_0^r,\ldots,L_0^{r^{s-1}}$. By Lemma \ref{lem:unitary-singer-line-orbit-stabilizer} (with $Q$ replaced by $Q^2$, so its fields are $F=\F_{Q^2}$ and $E=\F_{Q^{2u}}$), $C$ is scalar on the natural module $V$. But $C\leq A$ is a $p$-group, while the scalars are elements of norm one of $p'$-order. Therefore, $C=1$. Then $RC/C\unlhd H/C$ implies that $R\unlhd H$. But $R\unlhd H$ is incompatible with the Jordan blocks forced by the fixed line $L_0$. Both cases are impossible, so no such $A$ exists.
\end{proof}

\begin{lemma}\label{lem:unitary-unipotent-root-sl2}
Let $G=\PSU_n(q)$ with $n\geq3$ and $q\geq3$, and let $A<G$ be proper, nontrivial, and hereditarily $G$-permutable. If $A$ is contained in a single subgroup $T_{\langle v\rangle}=\{\,x\mapsto x+\lambda\langle x,v\rangle v:\lambda^q=-\lambda\,\}$ corresponding to a long root ($v$ isotropic), then $A=1$.
\end{lemma}

\begin{proof}
We choose $w$ so that the space $\langle v,w\rangle$ spanned by $v$ and $w$ is a hyperbolic plane and let $R=\langle T_{\langle v\rangle},T_{\langle w\rangle}\rangle\cong\SU_2(q)\cong\SL_2(q)$ \cite[p.~66]{Wilson2009}, so $Z(R)=\{\pm I\}$. A nonidentity element of $T_{\langle v\rangle}$ is unipotent, hence not $-I$, so $A\cap Z(R)=1$. By Lemma \ref{lem:overgroup-restriction} ($A\leq R\leq G$) and Lemma \ref{lem:quotient-inheritance} (modulo $Z(R)$), the image $\bar A$ of $A$ in $R/Z(R)\cong\PSL_2(q)$ is hereditarily $\PSL_2(q)$-permutable and isomorphic to $A$. The subgroup $\bar A$ lies in the image of $T_{\langle v\rangle}$ in $R/Z(R)$, a point stabilizer (Borel subgroup) of $\PSL_2(q)$. Since $q\geq3$, we have $\PSL_2(q)\not\cong\PGL_2(2)$, so Lemma \ref{lem:pgl2-borel} shows that $\bar A=1$. Hence $A=1$.
\end{proof}

\subsection{Orders of quotients associated with extension fields}

\begin{lemma}\label{lem:unitary-extension-quotient}
Let $n=ub$ with $b$ an odd prime, and let
$$
        E_b=N_G(S_b),~ S_b\cong\PSU_u(q^b),~ G=\PSU_n(q).
$$
Then every prime divisor of $|E_b/S_b|$ divides $b(q^b+1)/(q+1)$.
\end{lemma}

\begin{proof}
The normalizer $E_b=N_G(S_b)$ comes from the subgroup in class $\mathcal C_3$ of Kleidman and Liebeck \cite[\S4.3, Table 4.3.A]{KleidmanLiebeck1990} that fixes an extension field structure as a whole. This structure comes from $\GU_u(q^b)$ inside $\GU_n(q)$, with the natural space viewed over $\F_{q^{2b}}$. If $\delta\in C_{q^b+1}$ is the determinant over $\F_{q^{2b}}$, then the determinant over $\F_{q^2}$ is
$$
        N_{\F_{q^{2b}}/\F_{q^2}}(\delta)=\delta^{1+q^2+\cdots+q^{2(b-1)}}.
$$
Set $M=q^b+1$ and $N=1+q^2+\cdots+q^{2(b-1)}$. Since $b$ is odd,
$$
        N=\frac{q^{2b}-1}{q^2-1}=\frac{q^b-1}{q-1}\cdot\frac{q^b+1}{q+1}.
$$
Since $1+q+\cdots+q^{b-1}\equiv1\pmod{q+1}$, we have $\gcd(q+1,(q^b-1)/(q-1))=1$. Hence $\gcd(M,N)=(q^b+1)/(q+1)$, so the norm kernel on $C_{q^b+1}$ has order $(q^b+1)/(q+1)$. The special unitary determinant condition restricts $\delta$ to this norm kernel and removes the $(q+1)$-part. The automorphism of the extension $\F_{q^{2b}}/\F_{q^2}$ contributes only the factor $b$. Passing from $\SU_n(q)$ to $G=\PSU_n(q)$ means taking the quotient by $Z(\SU_n(q))$, which cannot introduce new prime divisors. Therefore, every prime divisor of $|E_b/S_b|$ divides $b(q^b+1)/(q+1)$.
\end{proof}

\begin{lemma}\label{lem:unitary-visibility-gcd}
If $b\nmid q+1$, then
$$
        \gcd\!\left(b\,\frac{q^b+1}{q+1},\,q+1\right)=1.
$$
\end{lemma}

\begin{proof}
Since $b\nmid q+1$, we have $(b,q+1)=1$, and
$$
        \frac{q^b+1}{q+1}=q^{b-1}-q^{b-2}+\cdots-q+1\equiv b\pmod{q+1}.
$$
A common prime divisor of the integers $b(q^b+1)/(q+1)$ and $q+1$ would divide $b$, but no prime divisor of $b$ divides $q+1$.
\end{proof}

\subsection{The permutation action on the lines of a self-polar simplex}

Let $M$ be a self-polar simplex stabilizer in ${\rm PSU}_m (Q)$. Its projective diagonal unitary subgroup $D$ is the kernel of the permutation action on the simplex lines. With this notation, $M=D:S_m$ and $M/D\cong S_m$. Let $A\leq M$, and write $B=AD/D\leq S_m$ for the image of $A$ in this action on the simplex lines. By Lemma \ref{lem:quotient-inheritance}, $B$ is hereditarily $S_m$-permutable in every symmetric quotient $S_m$ in which we consider the image of $A$.

\begin{lemma}\label{lem:unitary-alt-transfer}
Let $m\geq7$ and let $B\leq S_m$ be hereditarily $S_m$-permutable. Then $K=B\cap A_m$ is ordinary $A_m$-permutable.
\end{lemma}

\begin{proof}
Let $C\leq A_m$. Hereditary $S_m$-permutability of $B$ ensures that $BC^x=C^xB$ for some $x\in\gen{B,C}$. If $x\in A_m$, then we use $x$ itself as the conjugating element. If $x\notin A_m$, then $B\not\leq A_m$ (else $\gen{B,C}\leq A_m$). We pick $b\in B\setminus A_m$ and set $a=xb\in A_m$. Conjugating $BC^x$ by $b$ produces $BC^a$, and since $C^a\leq A_m$,
$$
        (BC^a)\cap A_m=(B\cap A_m)C^a=KC^a
$$
is a subgroup, so $KC^a=C^aK$. Since $C\leq A_m$ was arbitrary, $K$ is ordinary $A_m$-permutable.
\end{proof}

\begin{lemma}\label{lem:unitary-subset}
Let $B\leq S_m$ be hereditarily $S_m$-permutable, and let $H=S_k\times S_{m-k}$ be the setwise stabilizer in $S_m$ of a subset of size $k$ with $1\leq k<m/2$. Then either $B\leq H^x$ for some $x$, or $S_m=BH^x$ for some $x$. In the factorization case $S_m=BH^x$, the group $B$ is transitive on $k$-subsets and $\binom{m}{k}\mid|B|$.
\end{lemma}

\begin{proof}
We use Lemma \ref{lem:maximal-product} in $S_m$ for the maximal subgroup $H$. The factorization $S_m=BH^x$ makes $B$ transitive on the cosets of $H$, i.e., on $k$-subsets.
\end{proof}

We consider the possible permutation images $B=AD/D \leq S_m$ of $A$ on the simplex lines as follows. Solvable images are handled directly, in Lemma \ref{lem:unitary-solvable-low-socle-simplex} for composite degree and Lemma~\ref{lem:unitary-solvable-high-prime-simplex} for prime degree. The prime degree $5$ is handled in Lemma \ref{lem:s5-no-solvable-hereditary}. For nonsolvable $B$ we restrict the action of $B$ to an orbit with nonsolvable image and reduce to a transitive image on that orbit. Lemma \ref{lem:unitary-transitive-alternating} forces this transitive image to contain the alternating group. We then prove that no unitary subgroup $A$ with such a permutation image on the entire set of coordinates can satisfy hereditary permutability (Lemmas \ref{lem:unitary-transitive-alternating}, \ref{lem:unitary-full-support-simplex}, and \ref{lem:unitary-small-simplex}). For $k$-homogeneous permutation groups, we use the theorems of Livingstone and Wagner \cite[Theorem 2]{LivingstoneWagner1965} and Kantor \cite[Theorem 1]{Kantor1972}. We also use the classification of multiply transitive groups \cite[\S7.4]{DixonMortimer1996} and the transitive groups of small degree from Butler and McKay \cite{ButlerMcKay1983}. The remaining result is the theorem of Yang and Galt \cite[Theorem 2]{YG2023}. The geometric inputs include the maximality of the unitary parabolics $P_1$ and $P_2$ stabilizing totally singular lines and planes \cite[Table 3.5.B]{KleidmanLiebeck1990}, \cite[Proposition 26.1(b)]{MalleTesterman2011}. We also use Zsigmondy's theorem \cite{Zsigmondy1892} and the elementary module calculation in Lemma \ref{lem:alternating-permutation-no-two}.

\begin{lemma}\label{lem:no-solvable-four-homogeneous}
For $N\geq8$ no solvable subgroup of $S_N$ is $4$-homogeneous.
\end{lemma}

\begin{proof}
Let $B\leq S_N$ be $4$-homogeneous with $N\geq8$, so that $4\leq N/2$. By the Livingstone and Wagner theorem \cite[Theorem 2]{LivingstoneWagner1965}, we have that $B$ is $3$-transitive, in particular $2$-transitive and primitive. If $B$ is $4$-transitive, then by the classification of the multiply transitive groups \cite[\S7.4]{DixonMortimer1996}, $B$ is alternating, symmetric, or one of the Mathieu groups $M_{11},M_{12},M_{23},M_{24}$. If $B$ is $4$-homogeneous but not $4$-transitive, then by Kantor's classification of such groups \cite[Theorem 1]{Kantor1972}, $B$ is $\PGL_2(8)$ or $\PGammaL_2(8)$ of degree $9$, or $\PGammaL_2(32)$ of degree $33$. Every entry of both lists contains a nonabelian simple group (an alternating group $A_N$ with $N\geq 8$, a Mathieu group, or $\PSL_2(8)$ or $\PSL_2(32)$), so $B$ is nonsolvable. Hence no solvable subgroup of $S_N$ is $4$-homogeneous.
\end{proof}

\begin{lemma}\label{lem:solvable-simplex-prime-bound}
Let $N\geq5$ be odd and let $B\leq S_N$ be solvable and hereditarily $S_N$-permutable. Then $B$ stabilizes a partition of the $N$ points into parts of size $4$ together with one part of size $5$ or $7$. $B$ acts trivially on the part of size $5$ or $7$. Every prime divisor of $|B|$ is $2$ or $3$.
\end{lemma}

\begin{proof}
We induct on the odd integer $N$. If $N\in\{5,7\}$, then Lemma \ref{lem:s5-no-solvable-hereditary} for $N=5$ and Lemma \ref{lem:unitary-solvable-high-prime-simplex} for $N=7$ (with $u=7$) show that $B=1$. We take the whole set of $N$ points as the single part of the partition. This partition satisfies the conclusion.

We let $N\geq9$ be odd. Since $B$ is solvable and $N\geq8$, Lemma \ref{lem:no-solvable-four-homogeneous} shows that $B$ is not $4$-homogeneous. As $4<N/2$, the setwise stabilizer $H=S_4\times S_{N-4}$ of a subset of size $4$ is maximal in $S_N$. The product alternative $S_N = BH^x$ of Lemma \ref{lem:unitary-subset} would make $B$ transitive on $4$-sets and hence $4$-homogeneous. Therefore, this product alternative cannot occur. Hence $B\leq H^x$ for some $x$, i.e., $B$ stabilizes a subset $S_0$ of size $4$. Then $B\leq S_{S_0}\times S_{\Omega}$ with $\Omega$ the complementary $(N-4)$-set. We restrict hereditary permutability of $B$ from $S_N$ to $S_{S_0}\times S_{\Omega}$ (Lemma \ref{lem:overgroup-restriction}). To obtain the action on $\Omega$, we take the quotient of the symmetric group $S_{S_0} \times S_{\Omega}$ by its direct factor $S_{S_0}$, which acts on $S_0$ and fixes $\Omega$ pointwise (Lemma \ref{lem:quotient-inheritance}). Therefore, the image $B'$ of $B$ in $S_\Omega$ is solvable and hereditarily $S_{N-4}$-permutable. As $N-4\geq5$ is odd, induction shows that $B'$ stabilizes a partition of $\Omega$ into parts of size $4$ and one part of size $5$ or $7$. Moreover, $B'$ acts trivially on the part of size $5$ or $7$. Induction also shows that $|B'|$ is divisible by no prime other than $2$ or $3$. Adjoining $S_0$ produces the partition of all $N$ points, with the same part of size $5$ or $7$ as in the partition of $\Omega$. Because the action of $B$ on $\Omega$ factors through $B'$, $B$ acts trivially on the part of size $5$ or $7$. Finally the kernel of $B\to B'$ is supported on $S_0$, hence lies in $S_{S_0}\cong S_4$, a $\{2,3\}$-group, and $|B'|$ is a $\{2,3\}$-number, so $|B|$ is a $\{2,3\}$-number.
\end{proof}

\begin{lemma}\label{lem:unitary-solvable-low-socle-simplex}
Let $G=\PSU_n(q)$ with $n=ub$ odd, where $b$ is an odd prime with $b\nmid q+1$ and $u\geq3$. Let $E=N_G(S)$ with $S\cong\PSU_u(q^b)$ be the normalizer in $G$ of the subgroup defined by the extension field $\F_{q^{2b}}$. Suppose that a hereditarily $G$-permutable subgroup $A$ lies in $E$ and in a self-polar simplex stabilizer $\Delta=D:S_n$. If the image of $A$ in $E/S$ is nontrivial and the permutation image $B=AD/D$ of $A$ on the simplex lines is solvable, then no such subgroup $A$ exists.
\end{lemma}

\begin{proof}
We write $Q=AS/S\neq1$ for the nontrivial image of $A$ in the quotient $E/S$ of the extension field normalizer $E$ by its socle $S$. The permutation image $B=AD/D$ of $A$ on the simplex lines is hereditarily $S_n$-permutable (Lemmas \ref{lem:overgroup-restriction} and \ref{lem:quotient-inheritance}). By hypothesis $B$ is solvable, and $n=ub$ is odd with $n\geq9$, so Lemma \ref{lem:symmetric-solvable-hereditary-trivial} shows that $B=1$. Hence $A$ lies in the diagonal torus $D$ in $G$, and Proposition \ref{prop:unitary-toral} shows that $A=1$. But then $Q=AS/S=1$, contradicting $Q\neq1$. So no such subgroup $A$ exists.
\end{proof}

\begin{lemma}\label{lem:unitary-transitive-alternating}
Let $d\geq7$ and let $C\leq S_d$ be nonsolvable, transitive, and hereditarily $S_d$-permutable. Then $C\geq A_d$.
\end{lemma}

\begin{proof}
For $1\leq k<d/2$ let $H_k=S_k\times S_{d-k}$ be the setwise stabilizer in $S_d$ of a subset of size $k$. As $C$ is transitive, the group $C$ lies in no conjugate of $H_k$. Therefore, the product alternative $S_d = CH_k ^x$ of Lemma \ref{lem:unitary-subset} holds and $C$ is $k$-homogeneous for each such $k$. In particular, $C$ is $h$-homogeneous with $h=\lfloor(d-1)/2\rfloor\geq3$, hence $2$-transitive and primitive \cite[Theorem 2]{LivingstoneWagner1965}.

A $5$-homogeneous group is $5$-transitive \cite[Theorem 2]{LivingstoneWagner1965}, and by the classification of multiply transitive groups \cite[\S7.4]{DixonMortimer1996}, the $5$-transitive groups of degree $\neq12,24$ are alternating or symmetric. If $d\geq13$, then $h\geq6$, so $C$ is $6$-transitive and $C\geq A_d$. If $d=11$, then $h\geq5$, so $C$ is $5$-transitive of degree $11$ and $C\geq A_{11}$.

If $d=12$, then $h=5$, so $C$ is $5$-transitive of degree $12$, and the $5$-transitive groups of degree $12$ are $A_{12}$, $S_{12}$, and the Mathieu group $M_{12}$ \cite[\S7.4]{DixonMortimer1996}. Since $ C$ is $2$-transitive, the group $C$ is primitive and lies in no conjugate of the imprimitive maximal subgroup $H=S_6\wr S_2$. By Lemmas \ref{lem:maximal-product} and \ref{lem:factorization-coset-transitivity} the index $[S_{12}:H]=\tfrac12\binom{12}{6}=462=2\cdot3\cdot7\cdot11$ divides $|C|$. Since $|M_{12}|=95040=2^6\cdot3^3\cdot5\cdot11$ is not divisible by $7$, the divisibility condition shows that $C$ cannot be $M_{12}$, so $C\geq A_{12}$. The normalizer in $S_{12}$ of this copy of $M_{12}$ acting on twelve points is $M_{12}$ itself. The outer automorphism of $M_{12}$ is not induced by a permutation of the twelve points.

For $8\leq d\leq10$ we use Jordan's theorem. A primitive subgroup of $S_d$ containing a $p$-cycle for a prime $p\leq d-3$ contains $A_d$ \cite[Theorem 3.3E]{DixonMortimer1996}. Being primitive, $C$ lies in no conjugate of an imprimitive maximal subgroup $M$, so by Lemmas \ref{lem:maximal-product} and \ref{lem:factorization-coset-transitivity} the index $[S_d:M]$ divides $|C|$. Taking $M=S_4\wr S_2<S_8$, $M=S_3\wr S_3<S_9$, and $M=S_5\wr S_2<S_{10}$ produces indices $35$, $280$, and $126$ respectively, providing the prime $5$ (for $d=8,9$) or $7$ (for $d=10$) as a divisor of $|C|$. In each case $d/2<p\leq d-3$, so an element of the indicated prime order in $C$ is a single $p$-cycle and Jordan's theorem shows that $C\geq A_d$.

For $d=7$ the nonsolvable transitive groups are $\PSL_3(2),A_7,S_7$ \cite[Table 7A]{ButlerMcKay1983}. Lemma \ref{lem:unitary-subset} for $3$-subsets forces $\binom{7}{3}=35\mid|C|$, incompatible with $|\PSL_3(2)|=168$, so $C\geq A_7$. In every case $C\geq A_d$.
\end{proof}

\begin{lemma}\label{lem:alternating-permutation-no-two}
Let $F$ be a field and let $d\geq5$. The natural permutation module $V=F^d$ for $A_d$ has no $2$-dimensional $A_d$-submodule.
\end{lemma}

\begin{proof}
Let
$$
        \mathbf 1=e_1+\cdots+e_d,
        ~ T=F\mathbf 1,
        ~ W=\left\{\sum a_i e_i:\sum a_i=0\right\}.
$$
We first consider vectors of minimal support. Let $0\neq Y\leq W$ be an $A_d$-submodule, and suppose that $Y$ is not contained in $T$. We choose $y=(a_1,\ldots,a_d)\in Y\setminus T$. Since $y$ is not constant, two coordinates differ, say $a_i\neq a_j$. Choose $k$ distinct from $i,j$ and set $z=y-y^{(ijk)}$. Then $0\neq z\in Y$ and $z$ is supported on at most three coordinates. If $z$ has support two, then $z$ is a scalar multiple of some $e_r-e_s$. Suppose that $z$ has support three, say on $\{r,s,t\}$. If two nonzero coefficients (say those at $r$ and $s$) differ, then choose two points $a,b$ outside the support of $z$ and act on $z$ by $(rs)(ab)$. The difference $z-z^{(rs)(ab)}$ is a nonzero vector of $Y$ supported on $\{r,s\}$. If the three nonzero coefficients are all equal, then the characteristic is $3$. Choosing a point $a$ outside the support of $z$ and acting on $z$ by $(rs)(ta)$ produces a difference equal to a scalar multiple of $e_t-e_a$. Thus, in all cases $Y$ contains $e_r-e_s$ for some $r\neq s$. Since $A_d$ is $2$-transitive for $d\geq5$, the $A_d$-orbit of $e_r-e_s$ spans $W$. Hence $Y=W$.

Now suppose that $X\leq V$ is an $A_d$-submodule of dimension $2$. If $X\leq W$, then the argument using vectors of minimal support for submodules of $W$ shows that either $X\leq T$ or $X=W$, both impossible. Hence $X\nleq W$, and $X\cap W$ is an invariant line. By the argument using vectors of minimal support for submodules of $W$, the line $X\cap W$ must be $T$. In particular the characteristic divides $d$, so that $T\leq W$. The quotient $X/T$ embeds in the trivial module $V/W$, and $T$ is also trivial. Since $A_d$ is perfect for $d\geq5$, the action of $A_d$ on the extension $X$ of $X/T$ by $T$ is trivial. The space $X$ is fixed pointwise by $A_d$, but since $A_d$ is transitive, the fixed space of $A_d$ on $V$ is exactly $T$, of dimension $1$. This contradiction proves the lemma.
\end{proof}

\begin{lemma}\label{lem:unitary-full-support-simplex}
Let $\PSU_d(Q)\leq E\leq\PGU_d(Q)$ with $Q=p^f$ and $d\geq5$, let $\Delta=D:S_d$ be a self-polar simplex stabilizer in $E$, and suppose that $A\leq\Delta$ is hereditarily $E$-permutable with $B=AD/D\geq A_d$. Then no such $A$ exists.
\end{lemma}

\begin{proof}
Let $P_1$ be the stabilizer in $E$ of a totally singular $1$-space in the natural unitary module of $E$. Let $P_2$ be the stabilizer in $E$ of a totally singular $2$-space in the natural unitary module of $E$. Since $d\geq5$, the Witt index of this unitary module is at least $2$, so both subspaces exist and $P_1,P_2$ are maximal parabolics in ${\rm PGU}_d (Q)$. The maximality of these parabolic subgroups is the usual maximality of maximal parabolics in the finite $BN$-pair. The unitary stabilizers of these totally singular subspaces are listed in Kleidman and Liebeck \cite[Table 3.5.B]{KleidmanLiebeck1990}, and the Levi quotients of $P_1$ and $P_2$ are those of \cite[Proposition 26.1(b)]{MalleTesterman2011}. Since $\PSU_d(Q)$ is transitive on totally singular $i$-subspaces and each stabilizer $P_i$ maps onto $E/\PSU_d(Q)$, the maximality of $P_i$ holds in every intermediate group $E$.

Every prime divisor of $|D|$ divides $Q+1$, and every prime divisor of $|S_d|$ is at most $d$. Since $A\leq\Delta=D:S_d$, it is enough, for each $i=1,2$, to find a prime $r\mid[E:P_i]$ with $r>d$ and $r\nmid Q+1$. Then $[E:P_i]\nmid |A|$, so the product alternative $E=AP_i^x$ in Lemma \ref{lem:maximal-product} cannot occur and $A$ lies in a conjugate of $P_i$. 

The indices of $P_1$ and $P_2$ in $E$ are the numbers of totally singular $1$-spaces and $2$-spaces in this natural unitary module,
$$
        [E:P_1]=\frac{(Q^d-(-1)^d)(Q^{d-1}-(-1)^{d-1})}{Q^2-1},
$$
$$
        [E:P_2]=\frac{(Q^d-(-1)^d)(Q^{d-1}-(-1)^{d-1})(Q^{d-2}-(-1)^{d-2})(Q^{d-3}-(-1)^{d-3})}{(Q^4-1)(Q^2-1)}.
$$
Zsigmondy's theorem \cite{Zsigmondy1892} provides the primitive prime divisors in the following two cases. If $d$ is odd, then take a primitive prime divisor $r$ of $Q^{2d}-1$. Then $\ord_r(Q)=2d$, so $r\mid Q^d+1=Q^d-(-1)^d$. As $2d>4$, $r$ divides  neither $Q^2-1$ nor $Q^4-1$, and $r\equiv1\pmod{2d}$ implies that $r>d$ and $r\nmid Q+1$. If $d$ is even and $(Q,d)\neq(2,6)$, then take a primitive prime divisor $r$ of $Q^d-1$. Then $\ord_r(Q)=d$, so $r\mid Q^d-1=Q^d-(-1)^d$. As $d>4$, $r$ divides neither $Q^2-1$ nor $Q^4-1$, and $r\equiv1\pmod d$ implies that $r>d$ and $r\nmid Q+1$. For the exceptional pair $(Q,d)=(2,6)$ the factor $Q^{d-1}-(-1)^{d-1}=2^5+1=33$ contributes $r=11$, which divides $[E:P_1]=693$ and $[E:P_2]=6237$ but neither $Q+1=3$ nor $6!$. In each case, a prime $r>d$ with $r\nmid Q+1$ divides both $[E:P_1]$ and $[E:P_2]$. Hence $A\leq P_1^x$ and $A\leq P_2^y$, so $A$ stabilizes a totally singular $1$-space $\ell=\langle v\rangle$ and a totally singular $2$-space $U$.

We let $A^+\leq A$ be the preimage of $A_d\leq B$ under the permutation homomorphism from $A$ to $B=AD/D$ on the simplex lines. Since $A_d$ is transitive on the $d$ simplex coordinates, the vector $v$ has full support, $v=(v_1,\dots,v_d)$ with all $v_i\neq0$. To describe the projective action, conjugate the monomial group by $T=\operatorname{diag}(v_1,\dots,v_d)$, using the map $g\mapsto T^{-1}gT$. An element $g=\operatorname{diag}(\lambda_1,\dots,\lambda_d)P_\sigma$ of $A^+$ stabilizes $\ell=\langle v\rangle$, so $gv=\mu_gv$ for a scalar $\mu_g$. After conjugation by $T$ the element $g$ becomes $\mu_gP_\sigma$, so the projective action of $A^+$ is the ordinary permutation action of $A_d$ on the $d$-point permutation module $(\F_{Q^2})^d$, with $A_d$ permuting its coordinate vectors. The invariant $2$-space $U$ then produces a $2$-dimensional $A_d$-submodule $T^{-1}U$ of the $d$-point permutation module for $A_d$, contradicting Lemma \ref{lem:alternating-permutation-no-two}. This contradiction shows that no such $A$ exists.
\end{proof}

\begin{lemma}\label{lem:unitary-small-simplex}
Let $\PSU_d(Q)\leq E\leq\PGU_d(Q)$ with $d\in\{5,6\}$ and $Q=p^f$, let $\Delta=D:S_d$ be a self-polar simplex stabilizer in $E$, and suppose that $A\leq\Delta$ is hereditarily $E$-permutable with $B=AD/D$ nonsolvable and transitive on the $d$ lines of the simplex. Then no such $A$ exists.
\end{lemma}

\begin{proof}
By Lemmas \ref{lem:overgroup-restriction} and \ref{lem:quotient-inheritance}, $B$ is hereditarily $S_d$-permutable. By Butler and McKay \cite{ButlerMcKay1983}, the nonsolvable transitive groups of degree $5$ are $A_5,S_5$. In degree $6$, they are $A_6,S_6$ together with the actions of degree six of $A_5\cong\PSL_2(5)$ and $S_5\cong\PGL_2(5)$.

If $B\geq A_d$, then Lemma \ref{lem:unitary-full-support-simplex}, which is valid for all $d\geq5$, contradicts the assumed existence of $A$. This settles degree $5$, where $A_5$ and $S_5$ both contain $A_5$, and the cases $A_6,S_6$ in degree $6$.

We now consider the remaining case $d=6$ with $B\in\{A_5,S_5\}$. In their transitive actions of degree six $A_5\cong\PSL_2(5)$ and $S_5\cong\PGL_2(5)$ act on the projective line over $\F_5$, where no nonidentity element fixes more than two of the six points. In particular neither of these transitive groups of degree six contains a transposition. Let $t\in S_6$ be a transposition. By Lemma \ref{lem:maximal-product} the hereditarily $S_d$-permutable group $B$ permutes with a conjugate $\langle t^x\rangle$, so $P=B\langle t^x\rangle$ is a subgroup containing the transposition $t^x$. If $B=A_5$, then $t^x\notin B$ implies that $|P|=120$. In this case, $P$ contains the transitive group $A_5$, so the subgroup $P$ is transitive. The transitive subgroups of $S_6$ of order $120$ are the copies of $\PGL_2(5)$, none containing a transposition, a contradiction. If $B=S_5$, then $t^x\notin B$ implies that $|P|=240$. In this case, the subgroup would have index $3$ in $S_6$, but no subgroup of that index exists. Both remaining possibilities for $B$ are impossible.
\end{proof}

\begin{lemma}\label{lem:symmetric-odd-involution-not-hereditary}
Let $u>5$ be an odd prime. No subgroup of $S_u$ generated by an odd involution is hereditarily $S_u$-permutable.
\end{lemma}

\begin{proof}
Let $B=\gen{\tau}$ with $\tau$ an odd involution, written as a product of $t$ transpositions, and let $f=u-2t$ be its number of fixed points. Since $\tau$ is odd, $t$ is odd. The integer $u$ is also odd, so $f$ is odd. Consequently, $f=1$ or $f\geq3$.

Suppose that $f\geq3$. Choose a transposition $(ab)$ of $\tau$ and three fixed points $c,d,e$, and let $h=(abcde)$ and $H=\gen{h}$. On $\{a,b,c,d,e\}$ the transposition $(ab)$ and the $5$-cycle $h$ generate $S_5$. The remaining transpositions of $\tau$ are disjoint from $\{a,b,c,d,e\}$. Both $\tau$ and $h$ map $\{a,b,c,d,e\}$ to itself, and the restriction of $\tau$ to this set is $(ab)$. If $B$ were hereditarily $S_u$-permutable, then for some $x\in\gen{B,H}$ the product $BH^x$ would be a subgroup. As $|B|=2$ and $|H^x|=5$ are coprime, $|BH^x|=10$ and $H^x\unlhd BH^x$, so $\tau$ normalizes $H^x$. Since $\gen{B,H}$ maps $\{a,b,c,d,e\}$ to itself, restriction to this set of five points shows that $(ab)$ normalizes a group generated by a $5$-cycle. In $N_{S_5}(\gen{(abcde)})\cong\AGL_1(5)$ every involution has cycle type $(2,2,1)$, whereas $(ab)$ has cycle type $(2,1,1,1)$, so this normalization is impossible.

The remaining case is $f=1$, where $t=(u-1)/2\geq3$, so $\tau$ has two transpositions $(ab)$ and $(cd)$. Let $h=(bc)$ and $H=\gen{h}$. On $\{a,b,c,d\}$,
$$
        \gen{(ab)(cd),(bc)}\cong D_8,
$$
and in this dihedral group the conjugates of $(bc)$ are $(bc)$ and $(ad)$, neither of which commutes with $(ab)(cd)$. Since $\gen{B,H}$ maps $\{a,b,c,d\}$ to itself and $h$ is supported on that set of four points, every $\gen{B,H}$-conjugate of $h$ is supported in $\{a,b,c,d\}$, hence equals $(bc)$ or $(ad)$. As the remaining transpositions of $\tau$ are disjoint from $\{a,b,c,d\}$, such a conjugate commutes with $\tau$ if and only if the conjugate commutes with $(ab)(cd)$. Neither of the two conjugates commutes with this double transposition. As two subgroups of order $2$ permute if and only if their generators commute, no conjugate of $H$ permutes with $B$, contradicting hereditary $S_u$-permutability of $B$.
\end{proof}

\begin{lemma}\label{lem:symmetric-three-cycle-not-hereditary}
Let $N\geq5$. No subgroup of $S_N$ generated by a single $3$-cycle is hereditarily $S_N$-permutable.
\end{lemma}

\begin{proof}
Let $B=\gen{(a\,b\,c)}$, and choose two further points $d,e$. We let $h=(a\,b\,c\,d\,e)$ and $H=\gen{h}$. On $\{a,b,c,d,e\}$ the $3$-cycle and the $5$-cycle generate $A_5$ (any subgroup of $A_5$ of order divisible by $15$ is $A_5$), and both fix the remaining $N-5$ points. So $\gen{B,H}$ acts as $A_5$ on $\{a,b,c,d,e\}$ and trivially elsewhere. If $B$ were hereditarily $S_N$-permutable, then since $|B|=3$ and $|H|=5$ are coprime, we could choose some $x\in\gen{B,H}$ such that $BH^x$ is a subgroup of order $|B|\,|H|/|B\cap H^x|=15$. But a subgroup of order $15$ would be cyclic, with an element of order $15$, which $A_5$ lacks. The group $A_5$ has no such subgroup. Hence no conjugate of $H$ permutes with $B$, contradicting hereditary $S_N$-permutability of $B$.
\end{proof}

\begin{lemma}\label{lem:unitary-solvable-high-prime-simplex}
Let $u>5$ be prime and let $B\leq S_u$ be solvable and hereditarily $S_u$-permutable. Then $B=1$.
\end{lemma}

\begin{proof}
Let $K=B\cap A_u$. By Lemma \ref{lem:unitary-alt-transfer}, $K$ is ordinary $A_u$-permutable. Since $B$ is solvable, $K<A_u$. Therefore, the theorem of Yang and Galt \cite[Theorem 2]{YG2023} shows that $K=1$. Hence $B\hookrightarrow S_u/A_u\cong C_2$. If $B\cong C_2$, then $B$ is generated by an odd involution of $S_u$, contradicting Lemma \ref{lem:symmetric-odd-involution-not-hereditary}. Therefore, $B=1$.
\end{proof}

\begin{lemma}\label{lem:s5-no-solvable-hereditary}
No nontrivial solvable subgroup of $S_5$ is hereditarily $S_5$-permutable.
\end{lemma}

\begin{proof}
A check of the subgroups of $S_5$ (Appendix \ref{app:s5-solvable-certificate}) shows that the only nontrivial hereditarily $S_5$-permutable subgroups are $A_5$ and $S_5$, both nonsolvable.
\end{proof}

\begin{lemma}\label{lem:symmetric-solvable-hereditary-trivial}
For every odd $N\geq5$, the symmetric group $S_N$ has no nontrivial solvable hereditarily $S_N$-permutable subgroup.
\end{lemma}

\begin{proof}
We induct on the odd integer $N$. For $N=5$, we use Lemma \ref{lem:s5-no-solvable-hereditary} and for $N=7$ Lemma \ref{lem:unitary-solvable-high-prime-simplex}. Let $N\geq9$ and let $B\leq S_N$ be solvable and hereditarily $S_N$-permutable. Since $B$ is solvable and $N\geq8$, Lemma \ref{lem:no-solvable-four-homogeneous} shows that $B$ is not $4$-homogeneous, so by Lemma \ref{lem:unitary-subset} (with $k=4<N/2$) $B$ stabilizes a subset $S_0$ of size $4$. Let $\Omega$ be the complementary $(N-4)$-set. We restrict hereditary permutability of $B$ in $S_N$ to $\Stab_{S_N}(S_0)=\Sym(S_0)\times\Sym(\Omega)$. For the induced action on $\Omega$, we take the quotient of this direct product by $\Sym(S_0)$, the kernel of its action on $\Omega$ (Lemmas \ref{lem:overgroup-restriction} and \ref{lem:quotient-inheritance}). Therefore, the image of $B$ in $\Sym(\Omega)\cong S_{N-4}$ is solvable and hereditarily $S_{N-4}$-permutable. As $N-4\geq5$ is odd, the image of $B$ on $\Omega$ is trivial by induction. Hence $B$ acts trivially on $\Omega$, i.e., $B\leq\Sym(S_0)\cong S_4$ and $B$ fixes $\Omega$ pointwise. For each nontrivial subgroup of this copy of $S_4$, the computation of Appendix \ref{app:degree-nine-base} finds a $5$-cycle or $7$-cycle $h$ on $S_0$ together with two or three points of $\Omega$. Let $H=\gen{h}$. For this chosen cyclic subgroup, no $x\in\gen{B,H}$ makes $BH^x$ a subgroup. Since $\gen{B,H}$ is then supported on at most seven points, the same failure of permutability occurs inside the copy of $S_7$ on those points, adjoining fixed points when fewer than seven points are used, and is independent of $N$. Therefore, no nontrivial subgroup of this copy of $S_4$ is hereditarily permutable in the symmetric group $S_N$, and $B=1$.
\end{proof}

\begin{theorem}[Permutation images on the lines of a self-polar simplex]\label{thm:unitary-simplex-image}
Let $A\leq D:S_m$ be a hereditarily $E$-permutable  subgroup  with $\PSU_m(Q)\leq E\leq\PGU_m(Q)$, and let $B=AD/D\leq S_m$. Then $B=1$, and the remaining case of a subgroup of the diagonal torus $D$ in $E$ is settled by {\rm Proposition \ref{prop:unitary-toral}}.
\end{theorem}

\begin{proof}
Let $m$ be the number of simplex lines and $Q$ the value used in the unitary group notation. The simplex stabilizer lies in a group $E$ with $\PSU_m(Q)\leq E\leq\PGU_m(Q)$. The subgroup $A$ is hereditarily $E$-permutable. Suppose that $B$ is solvable. By Lemmas \ref{lem:overgroup-restriction} and \ref{lem:quotient-inheritance} the image $B$ is hereditarily $S_m$-permutable. In the descent for unitary groups of odd dimension, the number $m$ of simplex lines is the dimension of the unitary space, so $m$ is odd and $m\geq5$. Hence $B=1$ by Lemma \ref{lem:symmetric-solvable-hereditary-trivial}.

Suppose that $B$ is nonsolvable. As $B$ embeds in the product of the permutation images of $B$ on its orbits, not all of these permutation images are solvable. Therefore, we choose a $B$-orbit $\Omega$ with $C=B^\Omega$ nonsolvable, and let $d=|\Omega|$. Since $B$ maps $\Omega$ to itself, the group $B$ lies in $S_\Omega\times S_{\Omega^c}$. We restrict the hereditary condition of $B$ in $S_m$ to $S_\Omega\times S_{\Omega^c}$ (Lemma \ref{lem:overgroup-restriction}). We take the quotient of $S_\Omega\times S_{\Omega^c}$ by $S_{\Omega^c}$, the kernel of the action on $\Omega$ (Lemma \ref{lem:quotient-inheritance}). In this quotient, $C$ is hereditarily $S_d$-permutable. Moreover, the image $C$ of the action of $B$ on $\Omega$ is transitive by construction, and a transitive nonsolvable group has degree $d\geq5$.

The set $\{\langle e_i\rangle:i\in\Omega\}$ is permuted by $A$, so $A$ stabilizes the nondegenerate unitary coordinate subspace $V_\Omega=\langle e_i:i\in\Omega\rangle$. Let $J_\Omega$ be the stabilizer in $E$ of the orthogonal decomposition $V_\Omega\perp V_\Omega^\perp$. We restrict the hereditary condition of $A$ to $J_\Omega$. Let $K_\Omega$ be the complementary unitary factor of $J_\Omega$ on $V_\Omega^\perp$, acting on $V_\Omega^\perp$ and trivially on $V_\Omega$. To study the action on $V_\Omega$, we take the quotient $J_\Omega/K_\Omega$ and use its projective action $\rho_\Omega:J_\Omega/K_\Omega\to\PGU(V_\Omega)$ on $V_\Omega$. Write $E_\Omega=\rho_\Omega(J_\Omega/K_\Omega)$ and let $A_\Omega$ be the image of $A$ under $J_\Omega\to J_\Omega/K_\Omega\to E_\Omega$. Lemmas \ref{lem:overgroup-restriction} and \ref{lem:quotient-inheritance} show that $A_\Omega$ is hereditarily $E_\Omega$-permutable, with $\PSU_d(Q)\leq E_\Omega\leq\PGU_d(Q)$. The permutation action of $A_\Omega$ on the simplex lines indexed by $\Omega$ has image $C$.

If $d\geq7$, then Lemma \ref{lem:unitary-transitive-alternating} shows that $C\geq A_d$ and Lemma \ref{lem:unitary-full-support-simplex} shows that the image of $A$ in $E_\Omega$ cannot have this permutation image $C$. If $d\in\{5,6\}$, then Lemma \ref{lem:unitary-small-simplex} shows that the image $A_\Omega$ of $A$ in $E_\Omega$ cannot have the nonsolvable permutation image $C$. Hence $B$ cannot be nonsolvable, so $B=1$. In every case $A\leq D$, and Proposition \ref{prop:unitary-toral} settles the case of a subgroup of the diagonal torus $D$ in $E$.
\end{proof}

\begin{remark}\label{rmk:simplex-nonsolvable-reduction}
Theorem \ref{thm:unitary-simplex-image} concerns the permutation image $B=AD/D$ of the unitary subgroup $A$ on the lines of a self-polar simplex. If $B$ has an orbit $\Omega$ of size $d=|\Omega|$ with nonsolvable image $C=B^\Omega$, then we restrict to the setwise stabilizer $S_\Omega\times S_{\Omega^c}$ of $\Omega$ and take its quotient by $S_{\Omega^c}$, the symmetric group on the complement of $\Omega$. This restriction and quotient make the image $C$ hereditarily permutable in the symmetric group on $\Omega$. Lemma \ref{lem:unitary-transitive-alternating} forces the transitive image $C$ to contain the alternating group, with the two actions of degree six treated in Lemma \ref{lem:unitary-small-simplex}. We then consider the image $A_\Omega$ of $A$ in $E_\Omega$ acting on the coordinate space $V_\Omega=\langle e_i:i\in\Omega\rangle$. These are the image group and coordinate space defined in the proof of Theorem \ref{thm:unitary-simplex-image}. The indices of the parabolic stabilizers $P_1$ and $P_2$ of a totally singular line and a totally singular plane in $E_\Omega$ are divisible by a prime that does not divide the order of its simplex stabilizer $D_\Omega:S_d$, where $D_\Omega$ is the diagonal subgroup for the $d$ lines indexed by $\Omega$. If the unitary image of $A$ were contained in both parabolics, then we would have a $2$-dimensional submodule of the natural $A_d$-permutation module, contrary to Lemma \ref{lem:alternating-permutation-no-two}. Therefore, the permutation image of $A$ cannot contain the alternating group on the entire orbit.
\end{remark}

\subsection{Outline of the remaining unitary cases}

The rest of the proof is a descent that always reduces to the results for low rank, toral subgroups, and unipotent subgroups specified in this outline. The base cases of low rank (Theorem \ref{thm:unitary-low-bases}) settle the quotients of low rank. The cases involving a torus or a unipotent radical (Propositions \ref{prop:unitary-toral} and \ref{prop:unitary-unipotent}) have been established. The lemmas on extension fields concern the quotient of the extension field normalizer by its socle. The symmetric group acts on the lines of the corresponding self-polar simplex. Every prime divisor of the order of this quotient divides the order of that symmetric group. We use Theorem \ref{thm:unitary-simplex-image} to contradict a nontrivial image of $A$ in this quotient of the extension field normalizer. Using these results of low rank, toral and unipotent results, and lemmas on extension fields, we handle odd dimension by removing the largest prime divisor of the dimension. For even dimension, we use stabilizers in the unitary group of a nondegenerate line, a singular line, and a totally singular plane. These stabilizers reduce the argument to linear cases, cases of low rank, or cases involving toral or unipotent subgroups. For odd prime dimension coprime to $q+1$, we use Theorem \ref{thm:unitary-odd-prime-noncentral}, with the conclusion for $\PSU_5(2)$ following from Theorem \ref{thm:unitary-psu52}. For odd prime dimension dividing $q+1$, we use Lemma \ref{lem:unitary-prime-pure-central}. For odd composite dimension coprime to $q+1$, the conclusion follows from Proposition \ref{prop:unitary-odd-noncentral-composite}. If $n$ is odd composite and $(n,q+1)>1$, then we use Proposition \ref{prop:unitary-odd-central}. Even dimension is treated in Proposition \ref{prop:unitary-even}.
\subsection{Odd prime dimension coprime to \texorpdfstring{$q+1$}{q+1}}

\begin{theorem}\label{thm:unitary-odd-prime-noncentral}
Let
$$
        G=\PSU_r(q),~ r\geq5\text{ odd prime},~ (r,q+1)=1,~ (r,q)\neq(5,2).
$$
Then every proper ordinary $G$-permutable subgroup $F<G$ has order $1$ or $r$, and $G$ has no proper nontrivial hereditarily $G$-permutable subgroup.
\end{theorem}

\begin{proof}
We use the self-polar simplex stabilizer $M=D:S_r$ in $G$, maximal by Dye \cite[Theorem 5]{Dye1985}. By Liebeck, Praeger and Saxl \cite[Corollary 2]{LPS1990}, a simple unitary group of odd dimension with a maximal factorization is one of $U_3(3)$, $U_3(5)$, $U_9(2)$. None has prime dimension $r\geq5$, so the product alternative $G=FM^x$ cannot occur. So Lemma \ref{lem:ordinary-maximal-product} shows that $F$ is contained in a $G$-conjugate of $M$, and every prime divisor of $|F|$ either divides $q+1$ or is at most $r$.

We let $T$ be a unitary Singer torus corresponding to a partition with one part. This torus lies in $G$ and has order $|T|=(q^r+1)/(q+1)$. By $G$-permutability of $F$, we have a subgroup $FT^x$ of $G$, and since $G$ has no factorization, $FT^x<G$. A maximal overgroup of $FT^x$ in $G$ contains $T^x$, so by Bereczky's theorem \cite[Main Theorem]{Bereczky2000}, that maximal overgroup lies in $N_G(T^x)$, whence
$$
        |F|\mid r\,\frac{q^r+1}{q+1}.
$$
If a prime $a\mid|F|$ divided $q+1$, then $a\mid r(q^r+1)/(q+1)$. Since $(q^r+1)/(q+1)\equiv r\pmod{q+1}$, the congruence forces $a=r$, contradicting $(r,q+1)=1$. So no prime divisor of $|F|$ divides $q+1$. If $a$ is an odd prime divisor of $|F|$, then $a\leq r$. If $a<r$, then $a\mid(q^r+1)/(q+1)$ with $a\nmid q+1$, so $-q$ has order $r$ modulo $a$ and $r\mid a-1$, impossible. Hence every odd prime divisor of $|F|$ is $r$, and $|F|\mid r$ as $r \nmid q^r +1$.

Now let $A$ be nontrivial and hereditarily $G$-permutable, so $|A|=r$. We choose a self-polar simplex stabilizer $D:S_r$ containing $A$. Since $r\nmid q+1$, the image of $A$ in $S_r$ has order $r$ and is generated by an $r$-cycle. Testing hereditary permutability of this image against the subgroup generated by a transposition in $S_r$, we have a subgroup of order $2r$. Therefore, the transposition normalizes the cyclic subgroup generated by the $r$-cycle. But
$$
        N_{S_r}(\gen{(1\,2\,\cdots\,r)})=\AGL_1(r),
$$
whose nonidentity elements have at most one fixed point, while a transposition fixes $r-2\geq3$ points, a contradiction.
\end{proof}

\subsection{The case \texorpdfstring{$\PSU_5(2)$}{PSU5(2)}}

\begin{theorem}\label{thm:unitary-psu52}
The group $G=\PSU_5(2)$ has no proper nontrivial hereditarily $G$-permutable subgroup.
\end{theorem}

\begin{proof}
Let $A<G$ be proper, nontrivial, and hereditarily $G$-permutable. The maximal subgroups of $G$ are listed in the \emph{Atlas of Finite Groups} \cite{Atlas1985}. By the result of Liebeck, Praeger, and Saxl \cite[Corollary 2]{LPS1990}, $G$ has no factorization. Indeed, their list of factorizable unitary groups of odd dimension omits $U_5(2)$. The group $G$ has a maximal subgroup
$$
        T\cong\PSL_2(11).
$$
Since $G$ has no factorization $G=AT^x$, Lemma \ref{lem:maximal-product} shows that a $G$-conjugate of $A$ is contained in $T$. If $A<T$, then Lemma \ref{lem:overgroup-restriction} and the result for $\PSL_2$ in Theorem \ref{thm:unitary-low-bases} contradict hereditary permutability of $A$, so $A=T$. Testing hereditary permutability of $T$ against the maximal subgroup $M\cong3^4:S_5$ of $G$, we have the subgroup $TM^x$ for a conjugate $M^x$ that permutes with $T$. Since $M^x$ is maximal, either $T\leq M^x$ or $G=TM^x$. Because $11\mid|T|$ but $11\nmid|M|$, $T$ cannot be contained in $M^x$, and the factorization $G=TM^x$ is impossible by \cite[Corollary 2]{LPS1990}. This contradiction proves the theorem.
\end{proof}

\subsection{Odd composite dimension}

\begin{lemma}\label{lem:unitary-largest-prime}
Let $n$ be odd composite with $(n,q+1)=1$, and let $b=\max\pi(n)$, $u=n/b$. If $r\mid u$, then $r\nmid q^b+1$. Consequently, $(u,q^b+1)=1$.
\end{lemma}

\begin{proof}
Let $r\mid u$, so $r\leq b$ and $r\nmid q+1$. If $r<b$ and $r\mid q^b+1$, then $q^b\equiv-1$ and $q\not\equiv-1\pmod r$. As $b$ is prime, $\ord_r(q)$ divides $2b$ but not $b$, and is not $2$, so $\ord_r(q)=2b$ and $2b\mid r-1$, impossible for $r<b$. If $r=b$, then $q^b\equiv q\pmod b$, so $q^b+1\equiv q+1\not\equiv0\pmod b$. Hence $r\nmid q^b+1$ in all cases.
\end{proof}

\begin{proposition}\label{prop:unitary-odd-noncentral-composite}
Let $G=\PSU_n(q)$ with $n>3$ odd composite and $(n,q+1)=1$. Then $G$ has no proper nontrivial hereditarily $G$-permutable subgroup.
\end{proposition}

\begin{proof}
Suppose that $A<G$ is proper, nontrivial, and hereditarily $G$-permutable. Let $b=\max\pi(n)$ and $u=n/b$. By \cite[\S4.3, Table 4.3.A]{KleidmanLiebeck1990}, the $\mathcal C_3$ normalizer of the subgroup defined over the extension field in $G$ is
$$
        E_b=N_G(S_b),~ S_b\cong\PSU_u(q^b).
$$ If hereditary $G$-permutability of $A$ and Lemma \ref{lem:maximal-product} for $E_b$ led to the product alternative $G = A E_b ^x$, then a maximal subgroup $M\geq A$ of $G$ would produce a maximal factorization $G=ME_b^x$. But neither the list of factorizations with maximal factors of $\PSU_n(q)$ nor its classical version has such a row under these hypotheses \cite[Main Theorem]{LPS1990}, \cite[Theorem 1.4, Tables 4.1--4.2]{LiWangXia2024}. So after conjugating $A$ in $G$, $A\leq E_b$.

If $A<S_b$, then $A$ is hereditarily $S_b$-permutable by Lemma \ref{lem:overgroup-restriction}, and by Lemma \ref{lem:unitary-largest-prime} $(u,q^b+1)=1$, so $S_b$ is again a unitary group of smaller odd dimension satisfying this coprimality condition. Induction on the dimension, with Theorems \ref{thm:unitary-low-bases}, \ref{thm:unitary-odd-prime-noncentral} and \ref{thm:unitary-psu52} as bases, settles the case of a proper subgroup of $S_b$.

We next show that $A=S_b$ is impossible. Under the assumption $A=S_b$, the subgroup $S_b$ itself is hereditarily $G$-permutable. Using Lemma \ref{lem:maximal-product} for a self-polar simplex stabilizer $\Delta=D:S_n$ in $G$, we then have $S_b\leq\Delta^x$ or $G=S_b\Delta^x$. By \cite[Corollary 2]{LPS1990}, the factorizable simple unitary groups of odd dimension are $U_3(3)$, $U_3(5)$ and $U_9(2)$. The groups $U_3(3)$ and $U_3(5)$ have dimension $3$, and $U_9(2)$ has $(9,q+1)>1$. None has $n$ odd composite with $(n,q+1)=1$. Therefore, the factorization $G=S_b\Delta^x$ is impossible under the present hypotheses. So $S_b\leq\Delta^x$. We relabel $\Delta^x$ as $\Delta$ and write $D$ for its diagonal subgroup. Using Theorem \ref{thm:unitary-simplex-image} for the hereditarily $G$-permutable subgroup $S_b$ of the stabilizer of a self-polar simplex, we have $S_bD/D=1$, i.e., $S_b\leq D$. Since $D$ is abelian while $S_b\cong\PSU_u(q^b)$ ($u\geq3$, $q^b\geq8$) is nonabelian simple, we cannot have $S_b\leq D$.

If $A\not\leq S_b$, then let $Q=AS_b/S_b\neq1$. We prove that $A$ is contained in a stabilizer of a self-polar simplex in $G$, independently of Theorem \ref{thm:unitary-simplex-image}. The stabilizer $\Delta=D:S_n$ of a self-polar simplex is maximal in $G$ by Dye \cite[Theorem 5]{Dye1985}. We use hereditary $G$-permutability of $A$ and Lemma \ref{lem:maximal-product} for $\Delta$ with $A$ contained in $E_b$. Either $A\leq\Delta^x$ for some $x$, or $G=A\Delta^x$. In the latter case, a maximal subgroup $M\geq A$ produces a maximal factorization $G=M\Delta^x$ with the self-polar simplex stabilizer $\Delta^x$ as a factor. The only factorizable simple unitary groups of odd dimension are $U_3(3)$, $U_3(5)$, $U_9(2)$ \cite[Corollary 2]{LPS1990}. Since $n$ is odd composite, $n\neq3$. If $n=9$, then the condition $(9,q+1)=1$ is incompatible with $q\equiv2\pmod3$, so $q\neq2$. Hence $G=\PSU_n(q)$ is none of the three listed factorizable groups. Therefore, the factorization $G=A\Delta^x$ is impossible. Hence $A\leq\Delta^x$. Relabelling $\Delta^x$ as $\Delta$ ($\Delta^x$ is again a self-polar simplex stabilizer, and $A$ was not moved), we have $A\leq E_b\cap\Delta$.

By Lemma \ref{lem:unitary-extension-quotient},
$$
        \pi(Q)\subseteq\pi\!\left(b\,\frac{q^b+1}{q+1}\right),
$$
and since $b\nmid q+1$, Lemma \ref{lem:unitary-visibility-gcd} shows that the prime divisors of $|Q|$ are coprime to $q+1$. The permutation action $A\to B=AD/D$ on the lines of the self-polar simplex has kernel $A\cap D$. The order of this kernel has prime divisors only among those of $q+1$. Since $|Q|$ divides $|A|$ and every prime of $\pi(Q)$ is coprime to $q+1$, every prime of $\pi(Q)$ divides $|B|$. In particular $B\neq1$. But $A\leq E_b\cap\Delta$ is a hereditarily $G$-permutable subgroup of the stabilizer of a self-polar simplex in the unitary descent, so by Theorem \ref{thm:unitary-simplex-image} (see Remark \ref{rmk:simplex-nonsolvable-reduction}), we have $B=1$, a contradiction. Therefore, neither the case $A=S_b$ nor the case $A\not\leq S_b$ can occur.
\end{proof}

\begin{lemma}\label{lem:unitary-prime-pure-central}
Let $r\geq5$ be an odd prime with $r\mid q+1$. Then $G=\PSU_r(q)$ has no proper nontrivial hereditarily $G$-permutable subgroup.
\end{lemma}

\begin{proof}
Let $A<G$ be proper, nontrivial, and hereditarily $G$-permutable, and let $M=D:S_r$ be a self-polar simplex stabilizer, maximal in $G$ by Dye \cite[Theorem 5]{Dye1985}. If the product alternative $G = A M^x$ held with $M$, then $G$ would have a maximal factorization with a unitary socle of odd dimension. But the factorizable unitary groups of odd dimension are $U_3(3)$, $U_3(5)$, $U_9(2)$ \cite[Corollary 2]{LPS1990}. None has prime dimension $r\geq5$. So $A\leq M$ after conjugating $A$ in $G$. The image of $A$ in $M/D\cong S_r$ is handled by Theorem \ref{thm:unitary-simplex-image} (see Remark \ref{rmk:simplex-nonsolvable-reduction}), and $AD/D = 1$ implies that $A\leq D$, which is impossible by Proposition \ref{prop:unitary-toral}.
\end{proof}

\begin{proposition}\label{prop:unitary-odd-central}
Let $G=\PSU_n(q)$ with $n$ odd and $(n,q+1)>1$. Then $G$ has no proper nontrivial hereditarily $G$-permutable subgroup.
\end{proposition}

\begin{proof}
For ease of reading, we divide the proof into four steps.

\begin{itemize}[leftmargin=0pt,labelwidth=4em,labelsep=0.5em,itemindent=4.5em,align=left,topsep=3pt,partopsep=0pt,parsep=0pt,itemsep=0pt]
\item[{\bf Step 1.}] $A\leq\Delta$ after conjugation.
\end{itemize}
We begin with the prime dimensions that form the bases for induction. If $n=3$, then we use Theorem \ref{thm:unitary-low-bases}. If $n=r\geq5$ is prime, then $r\mid q+1$ and we use Lemma \ref{lem:unitary-prime-pure-central}. The rest of the proof concerns odd composite $n$. For a prime divisor $b$ of $n$, the descent below uses a subgroup $S_b\cong\PSU_{n/b}(q^b)$. If $A<S_b$, then we apply induction in $S_b$. This reduces the dimension from $n$ to $n/b$ and changes the field parameter from $q$ to $q^b$.

We need the permutation image of $A$ on the simplex lines in two cases. In one case, $A$ has nontrivial image in the quotient of the extension field normalizer by its socle. In the other, $A$ equals the socle of that normalizer. In both cases, we prove that the hereditarily $G$-permutable subgroup $A$ is contained in a stabilizer $\Delta=D:S_n$ of a self-polar simplex in $G$, independently of Theorem \ref{thm:unitary-simplex-image}. The stabilizer $\Delta$ is maximal in $G$ by Dye \cite[Theorem 5]{Dye1985}. Using Lemma \ref{lem:maximal-product} for $\Delta$, we have $A\leq\Delta^x$ unless $G=A\Delta^x$. In the latter case, a maximal subgroup $M\geq A$ of $G$ produces a maximal factorization $G=M\Delta^x$ with the self-polar simplex stabilizer $\Delta^x$ as a factor. By \cite[Corollary 2]{LPS1990}, the only factorizable simple unitary groups of odd dimension are $U_3(3),U_3(5),U_9(2)$. In the classical version, their factorizations have parabolic and almost simple factors, never a factor in class $\mathcal C_2$ that is a self-polar simplex stabilizer \cite[Theorem 1.4, Tables 4.1--4.2]{LiWangXia2024}. Therefore, no simple unitary group of odd dimension has such a factorization. Hence $A\leq\Delta^x$, and we relabel $\Delta^x$ as $\Delta$.

We consider the prime divisors of $n$ that do not divide $q+1$. Suppose that $n$ has a prime divisor not dividing $q+1$, and choose $b=\max(\pi(n)\setminus\pi(q+1))$. By \cite[\S4.3, Table 4.3.A]{KleidmanLiebeck1990}, the $\mathcal C_3$ normalizer of the subgroup defined over the extension field in $G$ is
$$
        E_b=N_G(S_b),~ S_b\cong\PSU_{n/b}(q^b).
$$

\begin{itemize}[leftmargin=0pt,labelwidth=4em,labelsep=0.5em,itemindent=4.5em,align=left,topsep=3pt,partopsep=0pt,parsep=0pt,itemsep=0pt]
\item[{\bf Step 2.}] After descent, $\pi(n)\subseteq\pi(q+1)$.
\end{itemize}
As in Proposition \ref{prop:unitary-odd-noncentral-composite}, we distinguish the case where $A$ is contained in the socle $S_b$ from the case where the image of $A$ in the quotient of $E_b$ by $S_b$ is nontrivial. When the image of $A$ in this quotient is nontrivial, $b\nmid q+1$. In this case, we choose $\Delta$ to contain $A$ by the argument in Step~1 establishing that $A$ lies in a self-polar simplex stabilizer. The prime divisors of the image of $A$ in $E_b$ modulo $S_b$ then appear in the permutation image $B=AD/D$ of $A$ on the simplex lines (Lemma \ref{lem:unitary-visibility-gcd}). By Theorem \ref{thm:unitary-simplex-image}, we have $B=1$, a contradiction (since the hypothesis $b\nmid q+1$ holds, we use Lemma \ref{lem:unitary-solvable-low-socle-simplex} when $B$ is solvable, see Remark \ref{rmk:simplex-nonsolvable-reduction}). We now suppose that $A\leq S_b$. If $A<S_b$, then the descent passes from dimension $n$ to dimension $n/b$, removing one factor $b$ from the dimension. Under the assumption $A<S_b$, take $r\mid n/b$ with $r\nmid q+1$. If $r=b$, then $q^b+1\equiv q+1\not\equiv0\pmod b$. If $r<b$ and $r\mid q^b+1$, then the order argument of Lemma \ref{lem:unitary-largest-prime} shows that $2b\mid r-1$, impossible. Every prime divisor of $n/b$ that does not divide $q^b+1$ is also a prime divisor of $n$ that does not divide $q+1$. Therefore, we use induction in $S_b$. Repeating this descent reduces to the case $\pi(n)\subseteq\pi(q+1)$. If $A=S_b$, then $S_b$ is hereditarily $G$-permutable. The factorization argument in Step~1 makes the product $G=S_b\Delta^x$ impossible. Therefore, we choose a stabilizer $\Delta$ of a self-polar simplex containing $S_b$ by Lemma \ref{lem:maximal-product}. We use Theorem \ref{thm:unitary-simplex-image} to conclude that $S_bD/D=1$, so $S_b\leq D$. Since $D$ is abelian while $S_b\cong\PSU_{n/b}(q^b)$ ($n/b\geq3$, $q^b\geq8$) is nonabelian simple, we cannot have $S_b\leq D$.

Now assume that $\pi(n)\subseteq\pi(q+1)$. Since the cases of prime dimension have been settled, suppose that $n$ is composite. We choose the largest prime $b=\max\pi(n)$ and let $u=n/b\geq3$. By \cite[\S4.3, Table 4.3.A]{KleidmanLiebeck1990}, the $\mathcal C_3$ normalizer of the subgroup defined over the extension field is
$$
        E_b=N_G(S_b),~ S_b\cong\PSU_u(q^b).
$$

\begin{itemize}[leftmargin=0pt,labelwidth=4em,labelsep=0.5em,itemindent=4.5em,align=left,topsep=3pt,partopsep=0pt,parsep=0pt,itemsep=0pt]
\item[{\bf Step 3.}] $A\leq E_b$ and $A\not\leq S_b$.
\end{itemize}
By Lemma \ref{lem:maximal-product}, after conjugating $A$ in $G$, either $A\leq E_b$ or $G=AE_b^x$. In the product case a maximal subgroup $M\geq A$ of $G$ produces a maximal factorization $G=ME_b^x$ with the subgroup $E_b^x$ as a factor in class $\mathcal C_3$ that fixes an extension field structure as a whole. By \cite[Corollary 2]{LPS1990}, the only factorizable simple unitary groups of odd dimension are $U_3(3)$, $U_3(5)$ and $U_9(2)$, and only $U_9(2)$ meets the present hypotheses ($n=9$, $q=2$, $b=3$, $S_b\cong\PSU_3(8)$). The only factorization of $U_9(2)$ with a nonsolvable factor is $U_9(2)=J_3\cdot P_1$, where $P_1=2^{1+14}{:}\SU_7(2)$ is the parabolic subgroup stabilizing a singular line in the natural module of $U_9(2)$ \cite[Lemma 5.7]{LiWangXia2024}. Since $E_b^x$ is nonsolvable (the group $E_b^x$ contains $\PSU_3(8)$) but is a $\mathcal C_3$ subgroup defined over an extension field, hence neither the almost simple factor $J_3$ nor the parabolic $P_1$, no such factorization exists. Hence, after conjugating $A$ in $G$, $A\leq E_b$.

If $A<S_b$, then $A$ is hereditarily $S_b$-permutable (Lemma \ref{lem:overgroup-restriction}). Since $b$ is odd and $b\mid q+1$, every prime $r\mid q+1$ has $q^b\equiv-1\pmod r$. Therefore, $\pi(u)\subseteq\pi(n)\subseteq\pi(q+1)\subseteq\pi(q^b+1)$. The group $S_b\cong\PSU_u(q^b)$ is again a unitary group of odd dimension with every prime divisor of $u$ dividing $q^b+1$, in the smaller dimension $u<n$. Induction on the dimension, with the bases of Theorem \ref{thm:unitary-low-bases} and Lemma \ref{lem:unitary-prime-pure-central}, settles the case of a proper subgroup of $S_b$.

We next show that $A=S_b$ is impossible. Under the assumption $A=S_b$, the subgroup $S_b$ is hereditarily $G$-permutable. By Lemma \ref{lem:maximal-product} for a self-polar simplex stabilizer $\Delta=D:S_n$ in $G$, we then have $S_b\leq\Delta^x$ or $G=S_b\Delta^x$. By the factorization list for the extension field normalizer $E_b=N_G(S_b)$ used in Step~3 (for $U_9(2)$ the only factorization is $J_3\cdot P_1$, with no self-polar simplex stabilizer as a factor), the factorization $G=S_b\Delta^x$ is impossible. So $S_b\leq\Delta^x$. We relabel $\Delta^x$ as $\Delta$ and write $D$ for its diagonal subgroup. Using Theorem \ref{thm:unitary-simplex-image} for the hereditarily $G$-permutable subgroup $S_b$ of the stabilizer of a self-polar simplex, we have $S_bD/D=1$, i.e., $S_b\leq D$. Since $D$ is abelian while $S_b\cong\PSU_u(q^b)$ ($u\geq3$, $q^b\geq8$) is nonabelian simple, we cannot have $S_b\leq D$.

\begin{itemize}[leftmargin=0pt,labelwidth=4em,labelsep=0.5em,itemindent=4.5em,align=left,topsep=3pt,partopsep=0pt,parsep=0pt,itemsep=0pt]
\item[{\bf Step 4.}] $A=1$.
\end{itemize}
If $A\not\leq S_b$, then let $Q=AS_b/S_b\neq1$. By Lemma \ref{lem:unitary-extension-quotient},
$$
        \pi(Q)\subseteq\pi\!\left(b\,\frac{q^b+1}{q+1}\right),
$$
and since
$$
        \frac{q^b+1}{q+1}=q^{b-1}-q^{b-2}+\cdots-q+1\equiv b\pmod{q+1},
$$
a prime of $\pi(Q)$ dividing the diagonal $(q+1)$-part of the simplex stabilizer can only be $b$. We next prove that $A$ is contained in a conjugate of a stabilizer $\Delta=D:S_n$ of a self-polar simplex in $G$. As $A\leq E_b$, a product $G=A\Delta^x$ would produce a maximal factorization $G=M'\Delta^x$ (with $M'\geq A$ maximal in $G$) having a self-polar simplex stabilizer as a factor. The factorization $G=M'\Delta^x$ is forbidden by the list for the extension field normalizer $E_b$ used in Step~3. Therefore, $A\leq\Delta^x$. Relabelling $\Delta^x$ as $\Delta$ and writing $D$ for its diagonal subgroup, we have $A\leq E_b\cap\Delta$. Let $B=AD/D$ be the image of $A$ in its permutation action on the simplex lines. The kernel of the permutation action from $A$ to $B$ on the simplex lines is $A\cap D$. The order of this kernel has prime divisors only among those of $q+1$. Therefore, every prime of $\pi(Q)$ coprime to $q+1$ divides $|B|$. A prime $r\mid(q^b+1)/(q+1)$ with $r\nmid q+1$ has $\ord_r(q)=2b$, hence $r\geq2b+1\geq7$, and by Lemma \ref{lem:solvable-simplex-prime-bound} divides the order of no solvable hereditarily $S_n$-permutable permutation image of $A$ on the simplex lines. If the prime $r$ divided the order of $B$, then $B$ would be nonsolvable, contradicting Theorem \ref{thm:unitary-simplex-image}. So no such prime occurs. By Theorem \ref{thm:unitary-simplex-image}, we have $B=1$ in any case. Then $A\leq D$ is diagonal and Proposition \ref{prop:unitary-toral} forces $A=1$, against assumption. Therefore, $A\leq S_b$, and repeating the descent through proper subgroups of $S_b$ reduces to the bases of prime dimension $r$ with $r\mid q+1$, the base in dimension three in Theorem \ref{thm:unitary-low-bases} and the higher prime bases in Lemma \ref{lem:unitary-prime-pure-central}.
\end{proof}

\begin{lemma}\label{lem:unitary-flag-toral}
Let $G=\PSU_{2m}(q)$ and let $A$ be a hereditarily $G$-permutable subgroup lying in a parabolic subgroup $P$ of $G$. Suppose that the image $C$ of $A$ in the Levi quotient $P/R_u(P)$ lies in the central torus of this Levi quotient. In this situation, the image of $A$ in each semisimple factor of this Levi quotient is trivial. Then $C=1$, so $A\leq R_u(P)$.
\end{lemma}

\begin{proof}
The image of $A$ in each semisimple factor of the Levi quotient of $P$ is trivial. Therefore, the subgroup $A$ lies in a Borel subgroup $B$ of $G=\PSU_{2m}(q)$ with maximal torus $T$ in $B$. Identify $B/R_u(B)$ with the maximal torus $T$. We use $C$ for the image of $A$ under $B\to B/R_u(B)\cong T$. The characters of $T$ are spanned by the simple roots $\alpha_1,\dots,\alpha_m$ relative to $B$, so it suffices to prove $\alpha_i(C)=1$ for every $i$. Then $C$ lies in the intersection of the kernels of the simple root characters, namely the scalar centre before taking the projective quotient, and is projectively scalar, i.e., $C=1$.

We fix $i$ and let $P_i\geq B$ be the minimal parabolic of $G$ corresponding to the simple root indexed by $i$, with Levi decomposition $P_i=L_iR_u(P_i)$. By \cite[Proposition 26.1(b)]{MalleTesterman2011} and \cite{Carter1985}, the semisimple factor of rank one of $L_i$ is $\SL_2(q')$ with $q'=q^2$ on a linear simple root and $q'=q$ on the unitary simple root. Since $A\leq B\leq P_i$, Lemma \ref{lem:overgroup-restriction} makes $A$ hereditarily $P_i$-permutable. Since $R_u(P_i)\unlhd P_i$, Lemma \ref{lem:quotient-inheritance} makes the image of $A$ in $L_i\cong P_i/R_u(P_i)$ hereditarily $L_i$-permutable. The relevant action is that of $L_i$ on the factor of rank one. We take the quotient of $L_i$ by the product of its central torus and the kernel of this action. This product is a normal subgroup of $L_i$, so Lemma \ref{lem:quotient-inheritance} applies again. It follows that the image $\bar A_i$ of $A$ is hereditarily $E_i$-permutable in the quotient group $E_i$, with $\PSL_2(q')\leq E_i\leq\PGL_2(q')$. Since $A\leq B$, the image $\bar A_i$ lies in the image of $B$ in $E_i$, a Borel subgroup (point stabilizer) of $E_i$, and the diagonal part of $\bar A_i$ is exactly $\alpha_i(C)$. When $q'\geq3$, by Lemma \ref{lem:pgl2-borel}, we have $\bar A_i=1$, so $\alpha_i(C)=1$. The only remaining value is $q'=2$ (the unitary simple root with $q=2$), where the diagonal torus of $\PGL_2(2)$ is trivial, so $\alpha_i(C)=1$ also in this case. Hence $\alpha_i(C)=1$ for every $i$, so $C=1$. The image of $A$ in the Levi quotient $P/R_u(P)$ is trivial, so $A\leq R_u(P)$.
\end{proof}

\subsection{Even dimension}

\begin{proposition}\label{prop:unitary-even}
For the simple group $G=\PSU_{2m}(q)$ there is no proper nontrivial hereditarily $G$-permutable subgroup.
\end{proposition}

\begin{proof}
For ease of reading, we divide the proof into seven steps. For $m=1$, the isomorphism $\PSU_2(q)\cong\PSL_2(q)$ reduces the required nonexistence of proper nontrivial hereditarily permutable subgroups to Theorem \ref{thm:unitary-low-bases}. Therefore, we may assume that $m\geq2$, and let $A<G$ be proper, nontrivial, and hereditarily $G$-permutable.

\begin{itemize}[leftmargin=0pt,labelwidth=4em,labelsep=0.5em,itemindent=4.5em,align=left,topsep=3pt,partopsep=0pt,parsep=0pt,itemsep=0pt]
\item[{\bf Step 1.}] If $G=AN_1$, then $G=AP_1^y$.
\end{itemize}
Let $N_1$ be the stabilizer in $G$ of a nondegenerate line of the natural unitary space $V$, a maximal $\mathcal C_1$-subgroup of type $\GU_1(q)\perp\GU_{2m-1}(q)$ \cite[Table 3.5.B]{KleidmanLiebeck1990}. By hereditary $G$-permutability of $A$ and Lemma \ref{lem:maximal-product}, after conjugating $A$ in $G$ either $A\leq N_1$ or $G=AN_1$. Suppose that $G=AN_1$, and take a maximal subgroup $M\geq A$ of $G$, so $G=MN_1$ is a maximal factorization. We use the reduction to factorizations with maximal factors of Liebeck, Praeger, and Saxl \cite[Main Theorem, Corollary 1]{LPS1990}. The possible unitary factorizations with $N_1$ as a factor are in Li, Wang, and Xia \cite[Theorem 4.1, Tables 4.1--4.2, Proposition 4.2]{LiWangXia2024}. The solvable case is covered by Li and Xia \cite{LiXia2022}.

We show directly that the factorization $G=AN_1$ is impossible. Let $P_1$ be a stabilizer of a singular line in $G$. The subgroup $N_1$ has at least two orbits on singular lines of $V$ (those in the orthogonal complement of the nondegenerate line fixed by $N_1$, and those outside that orthogonal complement), so $P_1^gN_1\neq G$ for all $g\in G$. Hence, if $A\leq P_1^g$, then the assumed product case $G=AN_1$ would force $G\leq P_1^gN_1<G$, a contradiction, so $A$ lies in no $G$-conjugate of $P_1$. By hereditary permutability of $A$ with $P_1$ and maximality of $P_1$ in $G$, we then have the factorization
\begin{equation}\label{eq:even-p1-index}
        G=AP_1^y,~\text{so}~ |G:P_1|\mid|A|.
\end{equation}
Using the unitary order formulas \cite[Propositions 2.3.1--2.3.3]{KleidmanLiebeck1990}, we obtain
$$
        |G:P_1|=\frac{(q^{2m}-1)(q^{2m-1}+1)}{q^2-1}.
$$
\begin{itemize}[leftmargin=0pt,labelwidth=4em,labelsep=0.5em,itemindent=4.5em,align=left,topsep=3pt,partopsep=0pt,parsep=0pt,itemsep=0pt]
\item[{\bf Step 2.}] $A\leq N_1$ after conjugation.
\end{itemize}
For $e=4m-2$, Zsigmondy \cite{Zsigmondy1892} gives a primitive prime divisor $r$ of $q^e-1$, except when $(q,e)=(2,6)$. The prime $r$ divides $q^{2m-1}+1$, divides $|G:P_1|$, and divides neither $q^{2m}-1$ nor the extension-degree parts of the Singer and parabolic rows. So the solvable $P_m$/Singer row and the generic nonsolvable $N_1$-rows cannot satisfy \eqref{eq:even-p1-index}. Writing $G=MN_1$ for a maximal subgroup $M\geq A$ of $G$, the relation \eqref{eq:even-p1-index} forces $r\mid|A|$, hence $r\mid|M|$. But no maximal subgroup of $G$ serving as the factor paired with $N_1$ in \cite[Theorem 4.1, Tables 4.1--4.2]{LiWangXia2024} has order divisible by this primitive prime divisor $r$ of $q^{4m-2}-1$.

The primitive prime divisor argument leaves the small cases $\PSU_4(2)$ and the $\PSU_6(2)$ row with factor $3\cdot M_{22}$ \cite[Table 4.1, Proposition 4.2]{LiWangXia2024}. For $\PSU_4(2)$, the ATLAS \cite{Atlas1985} gives
$$
        |G|=25920,~ |G:N_1|=40,~ |G:P_1|=45,~ |G:P_2|=27,
$$
with $|N_1|,|P_1|,|P_2|=648,576,960$. The factorization $G = AN_1$ bounds the order of $A$ by the order of the factor $S_6$, so $|A|\mid720$, while the factorization $G = AN_1$ also shows that $40\mid|A|$. Testing the hereditary permutability of $A$ with $P_1$, we have $45\mid|A|$, hence $360\mid|A|$. Then no $G$-conjugate of $P_2$ (order $960=2^6\cdot3\cdot5$) contains $A$, and testing the  hereditary permutability of $A$ with $P_2$ forces $27\mid|A|$, contradicting $|A|\mid720=2^4\cdot3^2\cdot5$. For the $\PSU_6(2)$ row,
$$
        |G:N_1|=672,~ |G:P_2|=6237=3^4\cdot7\cdot11,~ |P_2|=2^{15}\cdot3^2\cdot5,
$$
and $|3\cdot M_{22}|=2^7\cdot3^3\cdot5\cdot7\cdot11$, so the factorization with $N_1$ shows that $7\mid|A|$, $A$ is in no $G$-conjugate of $P_2$, and the factorization with $P_2$ forces $3^4\mid|A|$, again a contradiction. Hence $G=AN_1$ is impossible, and after conjugating $A$ in $G$, $A\leq N_1$.

\begin{itemize}[leftmargin=0pt,labelwidth=4em,labelsep=0.5em,itemindent=4.5em,align=left,topsep=3pt,partopsep=0pt,parsep=0pt,itemsep=0pt]
\item[{\bf Step 3.}] $A$ fixes a singular line.
\end{itemize}
We test $A\leq N_1$ against a singular line stabilizer $P_1$ in $G$. If $A$ lay in no conjugate of $P_1$, then hereditary permutability of $A$ and maximality of $P_1$ would force $G=AP_1^y$ and $|G:P_1|\mid|A|\mid|N_1|$. The subgroup $N_1$ has type $\GU_1(q)\times\GU_{2m-1}(q)$, whose primes other than the characteristic lie among $q+1$ and $q^i-(-1)^i$ for $1\leq i\leq2m-1$ \cite[Propositions 2.3.1--2.3.3]{KleidmanLiebeck1990}. Zsigmondy's theorem \cite{Zsigmondy1892} guarantees the primitive prime divisors in the following two cases. For $m$ even, a primitive prime divisor of $q^{2m}-1$ divides the $q^{2m}-1$ factor of $|G:P_1|$ and no factor of $|N_1|$. For $m$ odd, since the only factor of $|N_1|$ with exponent divisible by $m$ is $q^m+1$, a primitive prime divisor of $q^m-1$ divides $|G:P_1|$ but not $|N_1|$. The only exception is $\PSU_4(2)$. By the ATLAS \cite{Atlas1985}, $|G:P_1|=45\nmid648=|N_1|$. Thus, $|G:P_1|\nmid|N_1|$. So $A\leq P_1^y$. (For $\PSU_4(2)$, one more test with $P_2$ shows that $|G:P_2|=27\nmid576=|P_1|$, so $A\leq P_2^z$.) If the singular line $\ell$ and totally singular plane $X$ stabilized by $A$ in $V$ satisfy $\ell\not\leq X$, then consider $X\cap\ell^\perp$. Since $\ell^\perp$ is a hyperplane, the intersection $X\cap\ell^\perp$ is nonzero. If $X\cap\ell^\perp$ were all of $X$, then $X\leq\ell^\perp$, hence $\ell\leq X^\perp=X$ (as $X^\perp=X$ for the maximal totally singular plane $X$). This would contradict $\ell\not\leq X$. So $X\cap\ell^\perp$ has dimension one and is an $A$-stable singular line $\ell'\leq X$. Replacing $\ell$ by $\ell'$, we have the flag $\ell'<X$, and $A$ lies in the parabolic subgroup $P_{\ell'<X}$ of $G$ stabilizing the singular flag $\ell'<X$. This parabolic has a toral Levi quotient and unipotent radical $R_u(P_{\ell'<X})$. The image of $A$ in the toral Levi quotient $P_{\ell'<X}/R_u(P_{\ell'<X})$ is trivial by Lemma \ref{lem:unitary-flag-toral}. For $\PSU_4(2)$, the proof of Lemma \ref{lem:unitary-flag-toral} uses the subgroup $\SL_2(4)$ corresponding to the linear simple root as its factor of rank one, and $\PSL_2(4)\cong A_5$ is covered by Lemma \ref{lem:pgl2-borel}. The triviality of this toral image means that $A$ is contained in the unipotent radical $R_u(P_{\ell'<X})$ of $P_{\ell'<X}$. We now use the argument in Step 6 involving fixed spaces for $A$ in a unipotent radical. This argument shows that $A$ is the subgroup $\gen{t}$ generated by a single transvection. Such a subgroup cannot be hereditarily permutable by Appendix \ref{app:psu4-2}.

\begin{itemize}[leftmargin=0pt,labelwidth=4em,labelsep=0.5em,itemindent=4.5em,align=left,topsep=3pt,partopsep=0pt,parsep=0pt,itemsep=0pt]
\item[{\bf Step 4.}] If $\ell\not\leq U^\perp$, then $A\leq P_m^x$.
\end{itemize}
Now $A$ stabilizes a nondegenerate line $U$ and a singular line $\ell$ in the natural unitary space $V$. The group $\PSU_4(2)$ was handled using its parabolic subgroup stabilizing the singular flag in Step 3, so assume that $(q,m)\neq(2,2)$. The stabilizers of the nondegenerate line $U$ and the singular line $\ell$ are geometric maximal subgroups of $G$ \cite[Table 3.5.B]{KleidmanLiebeck1990}.

Suppose that $\ell\not\leq U^\perp$. Then $H=U+\ell$ is a nondegenerate plane and $A$ lies in the stabilizer $N_2$ of $H$ in $G$. We have $V=H\perp W$ with nondegenerate orthogonal complement $W$ and $\dim W=2m-2$. We next prove that $A$ is contained in a parabolic subgroup of $G$ stabilizing a maximal totally singular subspace of $V$. With $P_m$ the stabilizer in $G$ of a maximal totally singular $m$-space of $V$, if $A$ lay in no $G$-conjugate of $P_m$, then by Lemma \ref{lem:maximal-product}, we would have
$$
        G=AP_m^x,~ |G:P_m|\mid|A|\mid|N_2|.
$$
By the unitary order formula \cite[Propositions 2.3.1--2.3.3]{KleidmanLiebeck1990},
$$
        |G:P_m|=\prod_{i=1}^m(q^{2i-1}+1),
$$
while the factors in the part of $|N_2|$ prime to $p$ (with $N_2$ of type $\GU_2(q)\times\GU_{2m-2}(q)$) have exponents at most $2m-2$ apart from the factor on the nondegenerate plane $H$. Zsigmondy's theorem \cite{Zsigmondy1892} guarantees a primitive prime divisor of $q^{4m-2}-1$, except for the group $\PSU_4(2)$ with $(q,4m-2)=(2,6)$. We settled this case in Step 3. The primitive prime divisor divides the last factor $q^{2m-1}+1$ of $|G:P_m|$. The multiplicative order of $q$ modulo this prime is $4m-2$, so the prime does not divide $|N_2|$. So $|G:P_m|\nmid|N_2|$, and $A\leq P_m^x$.

\begin{itemize}[leftmargin=0pt,labelwidth=4em,labelsep=0.5em,itemindent=4.5em,align=left,topsep=3pt,partopsep=0pt,parsep=0pt,itemsep=0pt]
\item[{\bf Step 5.}] The semisimple Levi factors are linear or of rank one.
\end{itemize}
We let $Y=X\cap W$ for the totally singular $m$-space $X$ stabilized by $P_m^x$ and by $A$. In $Y^\perp/Y$, the orthogonal complement $Y^\perp$ is taken inside $W$. Since $X$ is totally singular of dimension $m$ and $W$ has Witt index $m-1$, the projection of $X$ to $H$ is nonzero of dimension at most $2$, so
$$
        \dim Y=m-1\text{ or }m-2,~ \dim(Y^\perp/Y)=0\text{ or }2.
$$
The parabolic subgroup stabilizing $Y$ in the unitary group on $W$ has Levi factors $\GL(Y)$ and $\GU(Y^\perp/Y)$, and toral factors, together with the unipotent radical of this parabolic subgroup \cite[Proposition 26.1(b)]{MalleTesterman2011}. When $\ell\not\leq U^\perp$ and $H$ is nondegenerate, we need only consider a linear factor over $\F_{q^2}$, toral factors, a case involving a unipotent radical, and possibly a quotient of rank one $\PSU_2(q)\leq E\leq\PGU_2(q)\cong\PGL_2(q)$ \cite[p. 66]{Wilson2009}.

Instead suppose that $\ell\leq U^\perp$, and let $W=U^\perp$, of odd dimension $2m-1$ and Witt index $m-1$, with $A$ stabilizing $U$ and the singular line $\ell\leq W$. Testing hereditary permutability of $A$ in $G$ against $P_m$ as in the case $\ell\not\leq U^\perp$, where $H=U+\ell$ is nondegenerate, we find that $A$ now lies in a subgroup $N_3$ of $G$ that stabilizes both $U$ and $\ell$, with Levi factors of types $\GU_1(q)$, $\GL_1(q^2)$, $\GU_{2m-3}(q)$. The parts of their orders prime to $p$ have no prime divisor for which $q$ has multiplicative order $4m-2$, whereas $|G:P_m|$ has such a prime divisor \cite{Zsigmondy1892}, so $|G: P_m| \nmid |N_3|$ and $A\leq P_m^x$. The totally singular $m$-space $X$ stabilized by $P_m^x$ and by $A$ has intersection $Y=X\cap W$ with $W$. Since $W$ has codimension $1$, the intersection $Y$ has dimension at least $m-1$. On the other hand, the Witt index of $W$ is $m-1$, so the dimension of $Y$ is not $m$. Together with the lower bound, this shows that the intersection $Y=X\cap W$ has dimension $m-1$. Therefore, the Levi quotient of the stabilizer of $Y$ in the unitary group on $W$ is linear of type $\GL_{m-1}(q^2)$. The other parts of this parabolic subgroup are toral factors and its unipotent radical.

\begin{itemize}[leftmargin=0pt,labelwidth=4em,labelsep=0.5em,itemindent=4.5em,align=left,topsep=3pt,partopsep=0pt,parsep=0pt,itemsep=0pt]
\item[{\bf Step 6.}] If $m\geq3$ and $A\leq R_u(P_m)$, then $A$ fixes $H$ pointwise.
\end{itemize}
An image of $A$ on $W$ lying in no subspace stabilizer in $\PSU(W)$ would, by hereditary permutability with a Borel subgroup of $\PSU (W)$, produce a factorization of  $\PSU (W)$ with a solvable factor that is a Borel subgroup. Since the $p$-part of this Borel subgroup is too large, this factorization is impossible by Li and Xia \cite{LiXia2022}. Hence the image of $A$ on $W$ lies in a parabolic subgroup of $\PSU (W)$. The argument forcing the image of $A$ to stabilize a line and a hyperplane reduces the linear rank. For the remaining cases of small rank in which the image contains the socle, we reduce to $\PSL_2(q^2)$, toral, or unipotent cases by Section \ref{sec:split-linear}. The linear images are handled by Theorem \ref{thm:split-linear}. The quotient of rank one $E$ from the projective action on $Y^\perp/Y$, when $\dim(Y^\perp/Y)=2$, is handled as follows. If $q>2$, then $E\not\cong\PGL_2(2)$ and Lemma \ref{lem:pgl2-borel} settles the image of $A$ in $E$. If $q=2$, then $E\cong\PGL_2(2)\cong S_3$. An image of $A$ in a point stabilizer of $E$ fixes a line in the natural $\F_2$-module for $E$. To extend the singular flag stabilized by $A$, we lift this invariant line in the natural $\F_2$-module for $E$, leaving only toral or unipotent images. By contrast, an image of $A$ in $E$ containing a subgroup of order $3$ has no invariant line on the natural $\F_2$-module for $E$. The lemma on invariant lines in modules, Lemma \ref{lem:sl-root-line-module-obstruction}, contradicts hereditary permutability of this image of $A$.

We now consider the case in which $A$ lies in a unipotent radical. $A\leq R_u(P_m)$ fixes the totally singular $m$-space $X$ pointwise. Moreover $A$ acts on the nondegenerate line $U$ through $\GU_1(q)$, cyclic of order $q+1$ and so a $p'$-group, so the $p$-group $A$ fixes $U$ pointwise as well. Hence every nonidentity $a\in A$ has $\Fix(a)\supseteq X\oplus U$, of dimension $m+1$. If $m=2$, then the fixed space of $a$ is a hyperplane of $V$, so $a$ is an element corresponding to a long root and $A$ lies in a single subgroup corresponding to a long root $T_v$ of $G$. By Lemma \ref{lem:unitary-unipotent-root-sl2}, we have $A=1$ for $q\geq3$, while $\PSU_4(2)$ is settled by Appendix \ref{app:psu4-2}. If $m\geq3$, then recall that $A\leq R_u(P_m)$ fixes the maximal totally singular subspace $X$ of $V$ pointwise. Since $U$ is nondegenerate, $X\not\leq U^\perp$. Otherwise $U\leq X^\perp=X$, impossible for a nondegenerate line lying in a totally singular space. Therefore, we may choose a singular line $\ell'\leq X$ with $\ell'\not\leq U^\perp$. We pick nonzero vectors $u\in U$ and $w'\in\ell'$. We use the field norm $N(z)=z^{q+1}$ from $\F_{q^2}$ to $\F_q$. As $\langle w',w'\rangle=0$ and $\langle u,w'\rangle\neq0$, the Hermitian form on $H=\gen{u,w'}$ has Gram determinant $\langle u,u\rangle\langle w',w'\rangle-\langle u,w'\rangle\langle w',u\rangle=-N(\langle u,w'\rangle)\neq0$. Thus, $H=U+\ell'$ is a nondegenerate plane. By the argument at the start of this paragraph, $A$ fixes $U$ pointwise, and $A$ fixes $\ell'\leq X$ pointwise, so $A$ fixes $H=U+\ell'$ pointwise. Therefore, $A$ fixes the nondegenerate line $H\cap U^\perp=U^{\perp_H}$ pointwise.

\begin{itemize}[leftmargin=0pt,labelwidth=4em,labelsep=0.5em,itemindent=4.5em,align=left,topsep=3pt,partopsep=0pt,parsep=0pt,itemsep=0pt]
\item[{\bf Step 7.}] $A=1$.
\end{itemize}
Since $A$ fixes $U$ pointwise, the group $A$ lies in its stabilizer. In the matrix unitary group on $V$, the stabilizer of $U$ has the description $\Stab(U)=\GU_1(U)\times\GU(U^\perp)$. We restrict hereditary permutability of $A$ to this stabilizer. To obtain the action of this stabilizer on $U^\perp$, we then take the quotient of this stabilizer by the $\GU_1(U)$-factor of ${\rm Stab} (U)$, which is the kernel of that action (Lemmas \ref{lem:overgroup-restriction} and \ref{lem:quotient-inheritance}). Therefore, the image of $A$ in $\GU(U^\perp)\cong\GU_{2m-1}(q)$ is a hereditarily $\GU_{2m-1}(q)$-permutable $p$-subgroup fixing the nondegenerate line $H\cap U^\perp=U^{\perp_H}$ in $U^\perp$. We take the quotient of $\GU_{2m-1}(q)$ by its centre to form the projective group $\PGU_{2m-1}(q) $. Since $\PGU_{2m-1}(q)/\PSU_{2m-1}(q)$ has order dividing $q+1$, a $p'$-number, the image of the $p$-subgroup coming from $A$ under the central quotient of the unitary group on $U^\perp$ lies in the simple socle $\PSU_{2m-1}(q)$. By Lemma \ref{lem:overgroup-restriction} the projective image of $A$ is hereditarily $\PSU_{2m-1}(q)$-permutable. Proposition \ref{prop:unitary-unipotent} (odd dimension $2m-1\geq5$) makes the image of $A$ in $\PSU_{2m-1} (q)$ trivial. As $A$ fixes $U$, $A=1$.

We are left with linear images of $A$, images of rank one, images in a torus, or images in a unipotent radical. For the linear images we use Theorem \ref{thm:split-linear}, and for the images of rank one we use Lemma \ref{lem:pgl2-borel}, together with the argument for $q=2$ in the quotient of rank one. Let $\mathcal F$ be the singular flag stabilized by $A$ at a step of the descent, and let $P_{\mathcal F}=\Stab_G(\mathcal F)$ be its parabolic stabilizer, containing $A$. The toral image $C=AR_u(P_{\mathcal F})/R_u(P_{\mathcal F})$ lies in the central torus of the Levi quotient $P_{\mathcal F}/R_u(P_{\mathcal F})$. Hence Lemma \ref{lem:unitary-flag-toral} shows that this Levi image $C$ is trivial. Therefore, $A$ is contained in the unipotent radical $R_u(P_{\mathcal F})$ of $P_{\mathcal F}$. We then use the argument in Steps 6 and 7 with the fixed spaces of the elements of $A$ and the nondegenerate line $U$ to conclude that $A$ is trivial. This completes the proof in even dimension.
\end{proof}

\subsection{Proof of \texorpdfstring{Theorem \ref{thm:unitary}}{the unitary theorem}}

\begin{proof}[Proof of Theorem \ref{thm:unitary}]
For $G=\PSU_n(q)$, the cases $n=2,3$ are the known cases for $\PSL_2$ and $\PSU_3$. If $n$ is an odd prime with $(n,q+1)=1$, then use Theorem \ref{thm:unitary-odd-prime-noncentral}, with the case $\PSU_5(2)$ from Theorem \ref{thm:unitary-psu52}. If $n$ is an odd prime with $n\mid q+1$, then use Lemma \ref{lem:unitary-prime-pure-central}. If $n$ is odd composite with $(n,q+1)=1$, then use Proposition \ref{prop:unitary-odd-noncentral-composite}. If $n$ is odd composite with $(n,q+1)>1$, then use Proposition \ref{prop:unitary-odd-central}. If $n$ is even, then use Proposition \ref{prop:unitary-even}. These cases cover all simple unitary groups.
\end{proof}

\begin{remark}
The descent in Proposition \ref{prop:unitary-odd-noncentral-composite} uses the largest prime divisor $b=\max\pi(n)$ of the dimension (Lemma \ref{lem:unitary-largest-prime}). The lemma ensures that the dimension $u$ of the smaller unitary group is coprime to $q^b+1$, where $q^b$ is the unitary field parameter in $S_b\cong\PSU_u(q^b)$. By contrast, a smaller prime divisor of the dimension need not have this coprimality property.
\end{remark}
\section{Symplectic groups}\label{sec:symplectic}

The natural symplectic space $V$ has dimension $2m$ over $\F_q$. Its nondegenerate alternating bilinear form $b$ satisfies $b(v,v)=0$ for every $v\in V$. In the notation of Kleidman and Liebeck \cite[\S2.1]{KleidmanLiebeck1990}, the orthogonal complement of $X\leq V$ is $X^\perp=\{v\in V:b(v,x)=0\text{ for every }x\in X\}$. A subspace $X$ is \emph{totally isotropic} if $b(x,y)=0$ for all $x,y\in X$. We also use \emph{totally singular} for this condition in this section. Since $b(v,v)=0$ for every $v\in V$, every line is isotropic.

A maximal totally isotropic subspace has dimension $m$ and is called \emph{Lagrangian}. A~complete~totally~isotropic~flag has the form
$$
        0=L_0<L_1<\cdots<L_m.
$$
Each $L_i$ has dimension $i$ for $0\leq i\leq m$ and is totally isotropic. In particular, $L_m$ is Lagrangian. An element $g$ stabilizes this flag if $g(L_i)=L_i$ for $0\leq i\leq m$. For an isotropic line $\ell$, the form $b$ induces an alternating form on $\ell^\perp/\ell$. This quotient space has dimension $2m-2$.

We choose a hyperbolic basis $e_1,\ldots,e_m,f_m,\ldots,f_1$ with $L_i=\gen{e_1,\ldots,e_i}$ for $1\leq i\leq m$. Its pairings satisfy $b(e_i,e_j)=b(f_i,f_j)=0$ for all $i,j$. The mixed pairings satisfy $b(e_i,f_i)=1$ and $b(e_i,f_j)=0$ for $i\neq j$. We write $E_0=L_m$. In the flag stabilizer, vectors may change within each subspace. To fix a subspace pointwise means to fix every vector in it. Every vector of $\ell$ is orthogonal to $\ell^\perp$. Therefore, for $u,v\in\ell^\perp$, the value $b(u,v)$ is unchanged when either vector is altered by a vector of $\ell$. This gives the induced form $b_\ell(u+\ell,v+\ell)=b(u,v)$.

The group $\Sp(V)$ consists of the linear transformations $g$ satisfying $b(gu,gv)=b(u,v)$ for all $u,v\in V$. Allowing $b(gu,gv)=c_gb(u,v)$ with $c_g\in\F_q^\times$ gives the conformal symplectic group. Taking the quotient of each group by its scalar transformations gives $\PSp(V)$ and $\PCSp(V)$, respectively. The projective action on $\ell^\perp/\ell$ is the action on its subspaces of dimension one $\gen{v+\ell}$, where $v\in\ell^\perp\setminus\ell$.

For symplectic groups, we follow the linear descent method of Section \ref{sec:split-linear}. We use the two alternatives for a singular line stabilizer, the contradiction from a factorization with a solvable factor, descent through Levi quotients of parabolic subgroups, and subgroups contained in a Borel subgroup.

\begin{theorem}\label{thm:symplectic}
For every simple group $G=\PSp_{2m}(q)$ there is no proper nontrivial hereditarily $G$-permutable subgroup.
\end{theorem}

The group $\PSp_4(2)\cong S_6$ is not simple. Since $\PSU_4(2)$ is covered by Theorem \ref{thm:unitary}, we use the isomorphism $\PSp_4(3)\cong\PSU_4(2)$ \cite[\S3.6.5]{Wilson2009} where needed.

\subsection{The alternatives for a singular line stabilizer}

Let $P_1$ be the stabilizer in $G$ of a totally isotropic line of the natural symplectic module of $G$. For a proper nontrivial hereditarily $G$-permutable subgroup $A<G$, Lemma \ref{lem:maximal-product} shows that either $A\leq P_1^x$ or $G=AP_1^x$. The equality $G=AP_1^x$ says that $A$ is transitive on isotropic lines.

\begin{proposition}\label{prop:symplectic-no-p1-factorization}
For a simple group $G=\PSp_{2m}(q)$ with $(m,q)\neq(2,3)$, the case $G=AP_1^x$ cannot occur for a proper nontrivial hereditarily $G$-permutable subgroup $A$. The group $\PSp_4(3)$ is handled through $\PSp_4(3)\cong\PSU_4(2)$ and {\rm Theorem \ref{thm:unitary}}.
\end{proposition}

\begin{proof}
Suppose that $G=AP_1^x$, so $A$ is transitive on isotropic lines. For $A$, we consider a parabolic of $G$ stabilizing a nonzero totally isotropic subspace $W$. This parabolic maps the nonempty proper set of isotropic lines in $W$ to itself, so this parabolic subgroup cannot contain $A$. We test hereditary $G$-permutability of $A$ with a Borel subgroup $B$ of $G$. Therefore, for some $G$-conjugate $B^y$, the product $AB^y$ is a subgroup. If $AB^y<G$, then a maximal subgroup of $G$ containing $AB^y$ contains the Borel subgroup $B^y$ and is parabolic, contradicting the fact that $A$ lies in no proper parabolic subgroup of $G$. Consequently, $G=AB^y$. In particular, for a maximal subgroup $M\geq A$ of $G$, we get the factorization $G=MB^y$ with solvable factor $B^y$. This factorization is covered by the classification of Li and Xia \cite{LiXia2022}.

A Borel subgroup of $G$, of type $C_m$, has $|B|_p=q^{m^2}$. For a symplectic socle $G$, the Li and Xia cases used for the factorization $G=MB^y$ are Table 1.1, rows 3--5, and Table 1.2, rows 12--17. The $\PSp_4(3)$ rows are omitted and handled through $\PSp_4(3)\cong\PSU_4(2)$ and Theorem \ref{thm:unitary}. The comparison is as follows.
$$
\begin{array}{c|c|c|c}
\text{row} & \text{socle }G & \text{solvable }p\text{-part} & \text{Borel }p\text{-part}\\\hline
\text{1.1, row 3} & \PSp_{2m}(q),\ q\text{ even} & q^{m(m+1)/2} & q^{m^2}\\
\text{1.1, rows 4--5} & \PSp_4(q) & q^3 & q^4\\
\text{1.2, rows 12--15} & \PSp_4(q),\ q=5,7,11,23 & q^3 & q^4\\
\text{1.2, row 16} & \Sp_6(2) & \leq 2^4 & 2^9\\
\text{1.2, row 17} & \PSp_6(3) & 3^5 & 3^9
\end{array}
$$
In every row of this table, the $p$-part of the listed solvable factor is smaller than the $p$-part of a Borel subgroup of its symplectic socle $G$. Hence $B^y$ is not a listed factor, contradicting $G=MB^y$.
\end{proof}

\subsection{Subgroups of \texorpdfstring{$P_1$}{P1} and descent through Levi quotients}

By Proposition \ref{prop:symplectic-no-p1-factorization}, after conjugating $A$ in $G$, we may assume that $A\leq P_1$ (the case $\PSp_4(3)$ being handled separately through $\PSp_4(3)\cong\PSU_4(2)$). Let $R=R_u(P_1)$. We take the quotient of $P_1$ by its unipotent radical $R$ and then take the image of that Levi quotient in its projective action on $\ell^\perp/\ell$, where $\ell$ is the line stabilized by $P_1$. In the second step we quotient the Levi quotient by the elements acting as scalars on $\ell^\perp/\ell$. By \cite[Proposition 26.1(b)]{MalleTesterman2011}, the finite Levi decomposition of $P_1$ identifies the nonabelian part of the projective Levi quotient of $P_1$, up to central and diagonal factors, with a group $E$ between
$$
        \PSp_{2m-2}(q)~\text{and}~\PCSp_{2m-2}(q).
$$ By Lemma \ref{lem:quotient-inheritance}, the image of $A$ in $E$ is hereditarily $E$-permutable.

\begin{lemma}\label{lem:symplectic-actual-levi-alternatives}
At a step of the symplectic descent, consider the singular line stabilizer used at that step and its projective Levi quotient. Let $S=\PSp_{2d}(q)\leq E\leq\PCSp_{2d}(q)$ be the symplectic factor of this quotient, with $d\geq2$ and $(d,q)\neq(2,2)$. Let $C\leq E$ be hereditarily $E$-permutable. Suppose that the image of $C$ in the semisimple factor of this quotient is nontrivial and does not contain the socle $S$. Then $C$ lies in the normalizer in $E$ of a proper parabolic of $S$.
\end{lemma}

\begin{proof}
Let $B_S$ be a Borel subgroup of $S$ and $B_E=N_E(B_S)$, so the image of $B_E$ under the quotient map from $E$ to $E/S$ is the whole quotient. Suppose that $C$ lies in no normalizer in $E$ of a proper parabolic of $S$. By hereditary $E$-permutability of $C$, we have $B_E^x$ for some $x \in E$ with $CB_E^x$ a subgroup of $E$. Suppose that $CB_E^x<E$. Since $B_E^x$ maps onto $E/S$, a maximal overgroup $M$ of $CB_E^x$ in $E$ cannot contain $S$. Hence $M\cap S$ is proper in $S$ and contains the Borel subgroup $B_S^x$, so it is parabolic, a contradiction. So $E=CB_E^x$, a factorization of the almost simple group $E$ with solvable factor $B_E^x$. Since the image of $C$ in the semisimple factor is nontrivial and does not contain the socle, $C\not\geq S$ and $C$ is core-free in $E$ (Lemma \ref{lem:almost-simple-core-free}). Also, $B_E^x$ contains the Borel subgroup $B_S^x$ of $S$. We use the symplectic Li and Xia rows of Proposition \ref{prop:symplectic-no-p1-factorization}. For every $d\geq2$, the comparison at the characteristic prime between the solvable factor and a Borel subgroup of $S$ is $q^{d(d+1)/2}<q^{d^2}$. The finitely many special rows ($\Sp_6(2)$, $\PSp_6(3)$, and $\PSp_4(q)$ for $q\in\{5,7,11,23\}$) satisfy the same comparison. Hence this comparison contradicts the factorization at every descent dimension. The one factorization case omitted from Proposition \ref{prop:symplectic-no-p1-factorization}, namely $(d,q)=(2,3)$ where $\PSp_4(3)\cong\PSU_4(2)$, is handled directly. In this descent, for $u,v\in\ell^\perp$ the value $\langle u,v\rangle$ is unchanged by an element of the stabilizer of $\ell$ in the symplectic group. Therefore, this element acts on the symplectic quotient $\ell^\perp/\ell$ as an isometry with similitude factor $1$. No diagonal or conformal extension of the factor $S=\PSp_4(3)$ occurs in this descent. Although the lemma is stated for $E$ up to $\PCSp_4(3)$, only $E=S$ occurs in this descent. The symplectic factor of the Levi quotient of this line stabilizer is the simple socle itself, $E=S=\PSp_4(3)$. We have $C\leq S$, and Lemma \ref{lem:overgroup-restriction} makes $C$ hereditarily $S$-permutable, so Theorem \ref{thm:unitary} shows that $C=1$ or $C=S$, contrary to the hypothesis on $C$.
\end{proof}

\begin{lemma}\label{lem:symplectic-psp4-2-levi}
Let $C\leq S_6\cong\PSp_4(2)$ be a proper hereditarily $S_6$-permutable image occurring in the symplectic descent. Then either $C=A_6$, which is transitive on the fifteen totally singular planes of the natural $\F_2$ symplectic module of $\PSp_4(2)$, or $C$ fixes a totally singular line.
\end{lemma}

\begin{proof}
If $C\leq A_6$, then $C$ is hereditarily $A_6$-permutable (Lemma \ref{lem:overgroup-restriction}), so $C\in\{1,A_6\}$ by the alternating theorem \cite{VT2012}. Since $S_6$ is transitive on the fifteen totally singular planes and the stabilizer $P_2$ in $S_6$ of a totally singular plane has order $48$ and contains odd permutations, $A_6P_2=S_6$. Consequently, the group $A_6$ is transitive on these planes. If instead $C\not\leq A_6$, then choose $b\in C\setminus A_6$. Let $K=C\cap A_6$. For $D\leq A_6$, hereditary $S_6$-permutability of $C$ provides $x\in\gen{C,D}$ with $CD^x=D^xC$. Replacing $x$ by $xb$ when $x\notin A_6$, we may take $x\in A_6$. Then $(CD^x)\cap A_6=KD^x$ is a subgroup. The subgroup $K$ is ordinary $A_6$-permutable (Lemma \ref{lem:normal-intersection-transfer}). Since $A_6\cong\PSL_2(9)$, an ordinary $\PSL_2(9)$-permutable subgroup has order at most $2$ \cite[Theorem 1]{GT2025}, so $|K|\leq2$. As $C/K$ embeds in $S_6/A_6\cong C_2$, $|C|\leq4$. Then $C$ is a $2$-group, contained in a Sylow $2$-subgroup of $S_6\cong\PSp_4(2)$, the unipotent radical of a Borel subgroup of $\PSp_4(2)$, and so fixes a totally singular line.
\end{proof}

When the semisimple image of $A$ in the Levi quotient of the singular line stabilizer used at that step of the descent is proper and nontrivial, we extend the totally isotropic flag stabilized by $A$ and lower the symplectic rank.

Suppose that the image of $A$ in the semisimple factor of the Levi quotient of $P_1$ contains $\PSp_{2m-2}(q)$. With $P_2$ the stabilizer in $G$ of a totally singular plane in the natural symplectic module of $G$, we test hereditary permutability of $A$ against $P_2$. If $A$ were contained in a conjugate of $P_2$, then the image of $A$ on $\ell^\perp/\ell$ would fix a singular line, contradicting the assumption that this image contains $\PSp_{2m-2}(q)$. In the product case $G=AP_2^z$ we have $|G:P_2|\mid|A|\mid|P_1|$. A primitive prime divisor of $q^{2m}-1$ lies in $|G:P_2|$ and not in $|P_1|$ \cite{Zsigmondy1892}, except when $(q,2m)=(2,6)$. In this exceptional case, $G=\PSp_6(2)$. By \cite[(3.22)]{Wilson2009}, the indices
$$
        |G:P_2|=315,~ |P_1|=23040
$$
directly show that $|G:P_2|\nmid|P_1|$.

If the image of $A$ in the semisimple factor of the Levi quotient of $P_1$ is trivial, then
$$
        A\leq R_u(P_1)\rtimes\GL_1(q).
$$
The subgroup $R_u(P_1)\rtimes\GL_1(q)$ acts trivially on $\ell^\perp/\ell$ and so stabilizes every flag of totally isotropic subspaces starting at $\ell$ and lying in $\ell^\perp$, hence $A$ lies in a Borel subgroup of $G$. The subgroup $A$ must then be trivial by Theorem~\ref{thm:symplectic-borel}.

\subsection{The symplectic Borel case}

\begin{lemma}\label{lem:symplectic-c2-borel-reduction}
For $q\geq4$, every proper hereditarily $\PSp_4(q)$-permutable subgroup lies in a Borel subgroup of this symplectic group.
\end{lemma}

\begin{proof}
Let $A<\PSp_4(q)$ be proper, nontrivial, and hereditarily $\PSp_4(q)$-permutable. By Lemma \ref{lem:maximal-product} and Proposition \ref{prop:symplectic-no-p1-factorization}, we may assume that $A\leq P_1$. We obtain this conclusion by testing hereditary permutability of $A$ on the stabilizer $P_1$ in $\PSp_4(q)$ of a totally isotropic line in its natural symplectic module. The Levi quotient of $P_1$ has rank one and acts projectively on $\ell^\perp/\ell$, where $\ell$ is the totally isotropic line fixed by $P_1$. A nontrivial proper image of $A$ in this quotient is settled by Lemma \ref{lem:pgl2-borel}. To show that the image of $A$ cannot be full, we test hereditary permutability of $A$ on the stabilizer $P_2$ of a totally singular plane. The argument in the case in which the semisimple image is full gives a contradiction. The trivial semisimple image of $A$ lies in $R_u(P_1)\rtimes\GL_1(q)$. We use Lemma \ref{lem:maximal-product} for $P_2$. The divisibility comparison between the index of $P_2$ and the order of $P_1$ in the case in which the semisimple image is full contradicts the product case. Hence $A\leq P_2^g$ for some $g$. Suppose that the singular line is not contained in the Lagrangian plane. The symplectic pairing between the singular line fixed by $P_1$ and the Lagrangian plane fixed by $P_2^g$ then induces a nonzero functional on the plane. The kernel of this functional is an $A$-stable singular line in the plane. Thus, whether or not the singular line fixed by $P_1$ lies in the Lagrangian plane fixed by $P_2^g$, these two subspaces determine a flag fixed by $A$. Hence $A$ lies in a Borel subgroup of $\PSp_4(q)$.
\end{proof}

\begin{lemma}\label{lem:symplectic-c2-borel}
Let $S=\PSp_4(q)\leq E\leq\PCSp_4(q)$, and let $B_E$ be the normalizer in $E$ of a Borel subgroup of $S$. If $C\leq B_E$ is hereditarily $E$-permutable, then $C=1$.
\end{lemma}

\begin{proof}
For $q=2$, where $\PSp_4(2)\cong S_6$ is not simple, we enumerate in $\Sp_4(2)$ the subgroups of the Borel subgroup stabilizing the flag $\gen{e_1}<\gen{e_1,e_2}$. For each nontrivial subgroup $C$ in this enumeration, Appendix \ref{app:sp4-2-borel-certificate} gives a cyclic subgroup $H$ such that $CH^x\neq H^xC$ for every $x\in\gen{C,H}$. Consequently, every such $C$ fails the hereditary $E$-permutability condition, and the required conclusion $C=1$ holds. Assume that $q\geq3$. Take a root system of type $C_2$ \cite{Carter1985} with simple roots $\alpha$ (short) and $\beta$ (long). The positive roots are
$$
        \alpha,~\beta,~\alpha+\beta,~2\alpha+\beta.
$$
Let $B$ be the Borel subgroup of $S$ normalized by $B_E$. Let $T$ be a split maximal torus of $B$ and $U$ the subgroup generated by the four positive root subgroups. Then $B=UT$. Let $P_\alpha,P_\beta$ be the maximal parabolics of $E$ with Levi factors of rank one corresponding to $\alpha$ and $\beta$, respectively. Since $C\leq B_E\leq P_\alpha,P_\beta$, restriction from $E$ to these two parabolic subgroups (Lemma \ref{lem:overgroup-restriction}) makes $C$ hereditarily $P_\alpha$-permutable and hereditarily $P_\beta$-permutable. Next, apply inheritance of hereditary permutability by quotients (Lemma \ref{lem:quotient-inheritance}) to the respective Levi quotients of rank one of $P_\alpha$ and $P_\beta$. Consequently, the images of $C$ satisfy the hereditary permutability condition in their respective quotient groups. Moreover, each quotient lies between $\PSL_2(q)$ and $\PGL_2(q)$, and the image of $C$ lies in a Borel subgroup of that quotient. Since $q\geq3$, Lemma \ref{lem:pgl2-borel} shows that both images are trivial.

Hence the toral part of $C$ is trivial under $\alpha$ and $\beta$. Therefore, the toral part of $C$ is projectively central and hence trivial in $\PSp_4(q)$. The $U_\alpha,U_\beta$ components of $C$ also vanish. Thus,
$$
        C\leq U_{\alpha+\beta}U_{2\alpha+\beta}.
$$
Conjugating $P_\beta$ by $s_\alpha$ gives a conjugate parabolic whose Levi factor has rank one. A root of this Levi factor is $s_\alpha(\beta)=2\alpha+\beta$. Therefore, inheritance of hereditary permutability by quotients, together with Lemma \ref{lem:pgl2-borel}, shows that the $U_{2\alpha+\beta}$ component of $C$ is trivial. Similarly, conjugating $P_\alpha$ by $s_\beta$ produces $s_\beta(\alpha)=\alpha+\beta$ and shows that the $U_{\alpha+\beta}$ component of $C$ is trivial. Hence $C=1$.
\end{proof}

\begin{theorem}\label{thm:symplectic-borel}
Let $S=\PSp_{2m}(q)\leq\widehat E\leq\PCSp_{2m}(q)$ with $m\geq2$, and let $B_{\widehat E}$ be the normalizer in $\widehat E$ of a Borel subgroup $B_S$ of $S$. If $C\leq B_{\widehat E}$ is hereditarily $\widehat E$-permutable, then $C=1$.
\end{theorem}

\begin{proof}
We induct on $m$. For $m=2$, we use Lemma \ref{lem:symplectic-c2-borel}. Assume that $m\geq3$ and take a hyperbolic basis $e_1,\ldots,e_m,f_m,\ldots,f_1$ of the natural symplectic module of $S$, with $E_0=\gen{e_1,\ldots,e_m}$ the Lagrangian subspace whose complete flag is stabilized by $B_S$. This flag is $0=L_0<L_1<\cdots<L_m=E_0$, with $\dim L_i=i$ for $0\leq i\leq m$. The vectors $e_1,\ldots,e_m$ are chosen so that $L_i=\gen{e_1,\ldots,e_i}$ for $1\leq i\leq m$. Each $L_i$ is totally isotropic, and the last space $E_0$ is Lagrangian.

We consider the projective action of the stabilizer of $E_0$ in $\widehat E$ on $E_0$, and quotient that stabilizer by the elements acting as scalars on $E_0$. Write $P_{E_0}$ for this stabilizer and $\rho_0:P_{E_0}\to\PGL(E_0)$ for its projective action. The quotient is $P_{E_0}/\ker\rho_0$, identified with $\rho_0(P_{E_0})$. This quotient is linear, of socle $\PSL_m(q)$, between $\PSL_m(q)$ and $\PGL_m(q)$ \cite[Proposition 26.1(b)]{MalleTesterman2011}. The image of $C$ lies in a Borel subgroup of this Levi quotient. Therefore, Lemma \ref{lem:sl-borel} shows that this image of $C$ is trivial. In particular, the group $C$ acts projectively trivially on $E_0$. Hence all components corresponding to the roots $e_i-e_j$ vanish and all toral ratios on the $e_i$ are trivial. The common toral value is considered through the images of the projective symplectic actions on the spaces $W_i$.

For each $i$ let $W_i=\gen{e_i}^\perp/\gen{e_i}$. Consider the stabilizer in $\widehat E$ of $\gen{e_i}$ and its symplectic quotient on $W_i$. Denote this stabilizer by $P_i$ and its projective action on $W_i$ by $\rho_i$. We identify the quotient $P_i/\ker\rho_i$ with $\rho_i(P_i)$. Applying Lemma \ref{lem:quotient-inheritance} to this quotient shows that $\rho_i(C)$ is hereditarily $\rho_i(P_i)$-permutable. Moreover, it lies in the induced Borel subgroup of this quotient and hence is trivial by induction. The remaining positive roots are $e_a+e_b$ ($a<b$) and $2e_a$. Since $m\geq3$, each of these remaining positive roots occurs in the root system of the symplectic group on some $W_i$ ($i\notin\{a,b\}$ for $e_a+e_b$, and $i\neq a$ for $2e_a$). Hence all remaining unipotent components of $C$ vanish. The images of the projective symplectic actions on the spaces $W_i$ also force the toral characters corresponding to long roots to be trivial. So the toral part of $C$ is projectively central, hence trivial in $S$, and $C=1$.
\end{proof}

\begin{proposition}\label{prop:symplectic-psp4}
For the simple group $G=\PSp_4(q)$ there is no proper nontrivial hereditarily $G$-permutable subgroup.
\end{proposition}

\begin{proof}
The case $q=2$ is not simple. For $q=3$, $\PSp_4(3)\cong\PSU_4(2)$ \cite[\S3.6.5]{Wilson2009} is covered by Theorem \ref{thm:unitary}. For $q\geq4$, Lemma \ref{lem:symplectic-c2-borel-reduction} shows that every proper nontrivial hereditarily $G$-permutable subgroup is contained in a Borel subgroup of $G$. In this range, Lemma \ref{lem:symplectic-c2-borel} shows that such a subgroup is trivial.
\end{proof}

\subsection{Proof of \texorpdfstring{Theorem \ref{thm:symplectic}}{the symplectic theorem}}

\begin{proof}[Proof of Theorem \ref{thm:symplectic}]
For $m=1$, $\PSp_2(q)\cong\PSL_2(q)$ is simple only for $q\geq4$ and is then the known base case of Theorem \ref{thm:unitary-low-bases}. For $m=2$, use Proposition \ref{prop:symplectic-psp4}. Assume that $m\geq3$ and let $A<G$ be proper, nontrivial, and hereditarily $G$-permutable. The $P_1$-factorization case $G = AP_1 ^x$ is impossible (Proposition \ref{prop:symplectic-no-p1-factorization}), so $A\leq P_1$ after conjugating $A$ in $G$. The alternatives for the image of $A$ in the Levi quotient of $P_1$ (Lemma \ref{lem:symplectic-actual-levi-alternatives}) lead to descent unless the semisimple image of $A$ is full or trivial. When the semisimple image of $A$ is full, suppose that $A$ were contained in $P_1\cap P_2^y$. Then the image of $A$ in the Levi quotient of $P_1$ would lie in a proper parabolic of that Levi quotient and could not be full. Therefore, $A$ can lie in no conjugate of the stabilizer $P_2$ in $G$ of a totally singular plane. By hereditary permutability of $A$ and maximality of the stabilizer $P_2$ of a totally singular plane, we have a factorization $G=AP_2^y$ with $|G:P_2|\mid|A|$. Analysing the index of $P_2$ in $G$ contradicts this factorization. If the semisimple image of $A$ is trivial, then $A$ lies in the inverse image in $P_1$ of the central torus of dimension one in its Levi quotient. This inverse image is $R_u(P_1)$ extended by this torus. We then use Theorem \ref{thm:symplectic-borel}. Repeating this descent through Levi quotients lowers the symplectic rank to the base $\PSp_4(q)$. For $q\geq3$ this is Proposition \ref{prop:symplectic-psp4}. By contrast, for $q=2$ the Levi factor $\PSp_4(2)\cong S_6$ is not simple, so Lemma \ref{lem:symplectic-actual-levi-alternatives} cannot be used. Instead, Lemma \ref{lem:symplectic-psp4-2-levi} shows that the proper hereditarily $S_6$-permutable image $C$ of $A$ is either $A_6$ or fixes a singular line. When $C=A_6$, the image is transitive on the totally singular planes, so $A$ lies in no conjugate of $P_2$. As in the case of a full semisimple image, the comparison with the index of $P_2$ contradicts this possibility. When the image fixes a singular line, we extend the flag and the rank drops. When the image $C$ of $A$ in $S_6$ is trivial, $A$ lies in the inverse image of a Borel subgroup under the quotient map from the singular line stabilizer used at this descent step to its Levi quotient $\PSp_4(2)\cong S_6$. In this case, $A$ is trivial by Theorem \ref{thm:symplectic-borel}. The $m=2$ base of that theorem covers $q=2$ via Appendix \ref{app:sp4-2-borel-certificate}. So $A=1$, a contradiction.
\end{proof}
\section{Orthogonal groups}\label{sec:orthogonal}

The space $V$ is defined over $\F_q$ and has a nondegenerate quadratic form $Q$. For orthogonal complements we use the associated bilinear form $b(u,v)=Q(u+v)-Q(u)-Q(v)$. In the notation of Kleidman and Liebeck \cite[\S2.1]{KleidmanLiebeck1990}, we write $X^\perp=\{v\in V:b(v,x)=0\text{ for every }x\in X\}$. A nonzero vector $v$ is \emph{singular} if $Q(v)=0$. A subspace $X$ is \emph{totally singular} if $Q(x)=0$ for every $x\in X$. The Witt index $r$ of $V$ is the largest dimension of a totally singular subspace.

With $r$ the Witt index, a complete totally singular flag has the form
$$
        0=F_0<F_1<\cdots<F_r.
$$
Each $F_i$ has dimension $i$ for $0\leq i\leq r$ and is totally singular. In particular, $F_r$ is maximal totally singular. For a singular line $\ell$, we have $Q(z)=0$ and $b(v,z)=0$ whenever $z\in\ell$ and $v\in\ell^\perp$. Hence $Q(v+z)=Q(v)$. Consequently, $Q_\ell(v+\ell)=Q(v)$ defines the induced quadratic form on $\ell^\perp/\ell$.

In a Witt decomposition $V=E\oplus D\oplus E^*$, both $E$ and $E^*$ are maximal totally singular subspaces of $V$. Since the pairing $b:E\times E^*\to\F_q$ is nondegenerate, the map $x\mapsto(e\mapsto b(e,x))$ identifies $E^*$ with the dual space $\Hom(E,\F_q)$. The subspace $D=(E\oplus E^*)^\perp$ is anisotropic. This means that $Q(d)\neq0$ for every nonzero $d\in D$. In dimension $2m$, plus type has $r=m$ and $D=0$, whereas minus type has $r=m-1$. In odd dimension $2m+1$, the Witt index is $m$.

A Witt basis consists of bases $e_1,\ldots,e_r$ of $E$ and $f_r,\ldots,f_1$ of $E^*$, together with a basis of $D$. For minus type, $D$ has dimension two. In odd dimension it has dimension one. For plus type, $D=0$, so no basis vectors from $D$ occur. The pairings satisfy $b(e_i,f_i)=1$ and $b(e_i,f_j)=0$ for $i\neq j$. Consequently, $H_i=\gen{e_i,f_i}$ is a hyperbolic plane, with $Q(xe_i+yf_i)=xy$ for $x,y\in\F_q$. The planes $H_i$ for different indices are orthogonal. Each is also orthogonal to $D$. For the complete flag $0=F_0<\cdots<F_r$, we choose $E=F_r$ and a Witt basis with $F_i=\gen{e_1,\ldots,e_i}$ for $1\leq i\leq r$.

For orthogonal groups, we use the two alternatives for a singular line stabilizer and the classification by Li and Xia \cite{LiXia2022} of factorizations with a solvable factor. We also use a geometric argument against an image of $A$ containing the orthogonal socle of the projective Levi quotient of that singular line stabilizer, together with the orthogonal Borel case.

\begin{theorem}\label{thm:orthogonal}
For every simple projective orthogonal group $G=\POmega(V)$ there is no proper nontrivial hereditarily $G$-permutable subgroup.
\end{theorem}

By \cite[(3.57)]{Wilson2009}, we have the following isomorphisms in low rank
\begin{equation}\label{eq:orthogonal-low-rank-isomorphisms}
\begin{aligned}
        \POmega_5(q)&\cong\PSp_4(q), & \POmega_6^+(q)&\cong\PSL_4(q),\\
        \POmega_6^-(q)&\cong\PSU_4(q), & \POmega_4^-(q)&\cong\PSL_2(q^2),
\end{aligned}
\end{equation}
which reduce the orthogonal cases of low rank to the linear, unitary, and symplectic families. The plus case in dimension four is a central product of two factors of rank one and introduces no new simple family.

\subsection{The alternatives for a singular line stabilizer}

Let $P_1$ be the stabilizer in $G$ of a singular line of $V$. For a hereditarily $G$-permutable subgroup $A<G$, by Lemma \ref{lem:maximal-product} we have $A\leq P_1^x$ or $G=AP_1^x$, the equality $G=AP_1^x$ making $A$ transitive on singular lines of $V$.

Because the dimensions $\leq6$ are the isomorphism bases \eqref{eq:orthogonal-low-rank-isomorphisms}, treated through the linear, unitary and symplectic theorems, we may assume that $\dim V\geq7$. For $\dim V\geq7$ the stabilizer of a singular line $P_1$ is a maximal parabolic of $G$ in all three types $B_m,D_m^+,D_m^-$ \cite[Table 3.5.B]{KleidmanLiebeck1990}, \cite[Proposition 26.1(b)]{MalleTesterman2011}, and the Witt index is at least $3$. In characteristic $2$ the groups of odd dimension satisfy $\POmega_{2m+1}(q)\cong\PSp_{2m}(q)$ \cite[\S3.8]{Wilson2009} and are covered by Theorem \ref{thm:symplectic}, so for type $B_m$ we take $q$ odd. For $D_4=\POmega_8^+(q)$, a graph automorphism of order $3$ cyclically permutes the three end nodes of the Dynkin diagram. This automorphism also permutes the corresponding standard parabolic subgroups $P_1,P_3,P_4$. The stabilizer $P_1$ is nevertheless a maximal parabolic, and the descent through singular line stabilizers uses $P_1$ alone. This reduces $\POmega_8^+(q)$ through $\ell^\perp/\ell$ of type $D_3=\POmega_6^+(q)\cong\PSL_4(q)$. No special treatment is needed.

\begin{proposition}\label{prop:orthogonal-no-p1-factorization}
With $P_1$ the stabilizer of a singular line in $G$, the equality $G=AP_1^x$ cannot hold.
\end{proposition}

\begin{proof}
Suppose that $G=AP_1^x$. A proper parabolic of $G$ stabilizes a flag $0<U_1<\cdots<U_t$ of nonzero totally singular subspaces \cite[Table 3.5.B]{KleidmanLiebeck1990}, \cite[Proposition 26.1(b)]{MalleTesterman2011}. Therefore, this parabolic subgroup maps the set of singular lines contained in $U_1$ to itself. This set is a proper subset of the singular lines of $V$. Hence $A$ lies in no proper parabolic of $G$. By testing hereditary $G$-permutability of $A$ with a Borel subgroup $B$ of $G$, we have $B^y$ with $AB^y$ a subgroup. If $AB^y<G$, then a maximal subgroup of $G$ containing $AB^y$ contains the Borel subgroup $B^y$ and is parabolic, a contradiction. So $G=AB^y$, and for a maximal subgroup $M\geq A$ of $G$, the product $G=MB^y$ is a factorization with solvable factor $B^y$, which is covered by the classification of Li and Xia \cite{LiXia2022}.

For the orthogonal factorization $G=MB^y$, the Li and Xia rows are Table 1.1, rows 7--9, and Table 1.2, rows 22--28. We list the solvable factor and compare its $p$-part with the $p$-part of a Borel subgroup of the orthogonal socle $G$ in that row.
\begingroup
\small\setlength{\tabcolsep}{4pt}\renewcommand{\arraystretch}{1.18}
\setlength{\LTpre}{3pt}\setlength{\LTpost}{3pt}
\Needspace{24\baselineskip}
\begin{longtable}[c]{G{0.14\textwidth}G{0.16\textwidth}G{0.28\textwidth}G{0.16\textwidth}G{0.09\textwidth}}
\hline
row & socle $G$ & listed solvable factor & solvable $p$-part & Borel $p$-part\\
\hline
\endhead
\hline
\endfoot
1.1, 7 & $\POmega_{2m+1}(q)$, $m\geq3$, $q$ odd & $(q^{m(m-1)/2}\cdot q^m):((q^m-1)/d).m<P_m$ & $q^{m(m+1)/2}|m|_p$ & $q^{m^2}$\\
1.1, 8 & $\POmega_{2m}^{+}(q)$, $m\geq5$ & $q^{m(m-1)/2}:((q^m-1)/d).m<P_k$, $k=m,m-1$ & $q^{m(m-1)/2}|m|_p$ & $q^{m(m-1)}$\\
1.1, 9 & $\POmega_8^+(q)$ & $q^6:((q^4-1)/d).4<P_k$ & $q^6|4|_p$ & $q^{12}$\\
1.2, 22 & $\Omega_7(3)$ & $3^5:2^4.\AGL_1(5)$ & $3^5$ & $3^9$\\
1.2, 23 & $\Omega_7(3)$ & $3^{3+3}:13:3$ & $3^7$ & $3^9$\\
1.2, 24 & $\Omega_9(3)$ & $3^{6+4}:2^{1+4}.\AGL_1(5)$ & $3^{10}$ & $3^{16}$\\
1.2, 25 & $\Omega_8^+(2)$ & $2^2:15.4<A_9$ & $\leq2^4$ & $2^{12}$\\
1.2, 26 & $\Omega_8^+(2)$ & $2^6:15.4$ & $2^8$ & $2^{12}$\\
1.2, 27 & $\POmega_8^+(3)$ & $3^6:2^4.\AGL_1(5)$ & $3^6$ & $3^{12}$\\
1.2, 28 & $\POmega_8^+(3)$ & $3^6:(3^3:13:3)$, $3^{3+6}:13.3$ & $\leq3^{10}$ & $3^{12}$\\
\end{longtable}
\endgroup
In Li and Xia's classification \cite{LiXia2022}, the orthogonal rows with a solvable factor listed in the table are of odd type ($B_m$, row $7$) or plus type ($D_m^+$, rows $8$ to $9$, $25$ to $28$). The classification contains no factorization of a group of minus type $\POmega^-_{2m}(q)$ with a solvable factor, so for minus type the product $G=MB^y$ with solvable $B^y$ cannot occur. The factor $|m|_p$ or $|4|_p$ in the infinite rows comes from the permutation part of the table. The ratio of Borel $p$-part to listed $p$-part is $q^{m(m-1)/2}/|m|_p>1$ in rows 7 ($m\geq3$) and 8 ($m\geq5$). In row 9 the ratio is $q^6/|4|_p>1$. The entries in each small row also show that the listed solvable factor has strictly smaller $p$-part than the Borel subgroup. These permutation factors do not change the strict order comparisons. Consequently, no listed solvable factor contains $B^y$, contradicting $G=MB^y$.
\end{proof}

\subsection{Subgroups of parabolic subgroups and descent through Levi quotients}

By Proposition \ref{prop:orthogonal-no-p1-factorization}, after conjugating $A$ in $G$, we have $A\leq P_1$, where $P_1$ is the stabilizer in $G$ of a singular line of $V$. Throughout this subsection, let $\ell$ be the stabilized singular line, $R=R_u(P_1)$, and $W=\ell^\perp/\ell$. We use the image of $P_1$ in its projective action on $W$, i.e., the quotient of $P_1$ by the kernel of its projective action on $W$. Write $\rho_\ell:P_1\to\PGL(W)$ for this action. Its image is identified with $P_1/\ker\rho_\ell$. This projective Levi quotient of $P_1$ has orthogonal socle $\POmega(W)$ together with diagonal toral factors \cite[Proposition 26.1(b)]{MalleTesterman2011}. By Lemma \ref{lem:quotient-inheritance}, the image $C$ of $A$ on $W$ in this quotient is hereditarily $\POmega(W)$-permutable. In the case where the image of $A$ in $\POmega(W)$ is trivial, the image of $A$ lies in the diagonal quotient.

\begin{lemma}\label{lem:orthogonal-actual-levi-alternatives}
With $W$ and $P_1$ as defined at the start of this subsection, let $S=\POmega(W)$ be the orthogonal socle of the Levi quotient of the singular line stabilizer $P_1$. Let $E$ be the image of the projective action of $P_1$ on $W$. This projective Levi quotient includes the diagonal factors and contains $S$. Let $C\leq E$ be hereditarily $E$-permutable. Then one of the following holds.
\begin{enumerate}[label=\textup{(\arabic*)}]
\item $C\geq S$.
\item $C$ has trivial image in the orthogonal socle $S$.
\item $C$ lies in the normalizer in $E$ of a proper parabolic subgroup of $S$.
\end{enumerate}
Suppose that $C$ lies in the normalizer in $E$ of a proper parabolic subgroup of $S$. In case {\rm (3)}, taking the inverse image in $P_1$ of this proper parabolic subgroup of $S$ extends the singular flag in $V$.
\end{lemma}

\begin{proof}
Suppose that $C\not\geq S$, $C$ has nontrivial image in the orthogonal socle $S$, and $C$ does not lie in the normalizer in $E$ of a proper parabolic of $S$. We let $B_S$ be a Borel subgroup of $S$, and $B_E=N_E(B_S)$. The quotient $E/S$ corresponding to the diagonal part comes from a torus of the Levi quotient of $P_1$, so is represented by elements normalizing $B_S$ \cite[Example 24.22, Theorem 24.24]{MalleTesterman2011}, \cite[Proposition 26.1(b)]{MalleTesterman2011}. Therefore, the group $B_E$ maps onto $E/S$ and is solvable (an extension of $B_S$ by a toral quotient). By hereditary $E$-permutability of $C$, we have $B_E^x$ with $CB_E^x$ a subgroup. If $CB_E^x<E$, then a maximal subgroup $M$ of $E$ containing $CB_E^x$ cannot contain $S$. Then $M\cap S$ is proper and contains $B_S^x$, so it is parabolic in $S$. Hence $M\leq N_E(M\cap S)$, a contradiction. So $E=CB_E^x$, a factorization with solvable factor $B_E^x$ and a core-free subgroup $C$ of $E$ (Lemma \ref{lem:almost-simple-core-free}), with $B_E^x\cap S$ containing $B_S^x$. The orthogonal Li and Xia comparison \cite{LiXia2022} of Proposition \ref{prop:orthogonal-no-p1-factorization} contradicts this factorization. Finally, a proper parabolic of $\POmega(W)$ stabilizes a flag of nonzero totally singular subspaces in $W$. The quotient map is $p_\ell:\ell^\perp\to W$, given by $p_\ell(v)=v+\ell$. Consequently, the inverse image of this flag in $\ell^\perp$ under the quotient map onto $W$ is a totally singular flag in $V$ containing $\ell$, extending the singular flag.
\end{proof}

By Lemma \ref{lem:orthogonal-actual-levi-alternatives}, when the semisimple image of $A$ is proper and nontrivial, we extend the singular flag stabilized by $A$ in the descent. The cases reached at the end of this descent are the linear, unitary, and symplectic groups and the groups of rank one in \eqref{eq:orthogonal-low-rank-isomorphisms}.

\begin{lemma}\label{lem:orthogonal-full-levi-obstruction}
The case in which the image of $A$ contains the full orthogonal socle on $W=\ell^\perp/\ell$ cannot occur.
\end{lemma}

\begin{proof}
Suppose that the image of $A$ on $W$ contains $\POmega(W)$. By Witt's lemma \cite[Proposition 2.1.6]{KleidmanLiebeck1990}, the orthogonal group of $W$ is transitive on singular subspaces of a fixed dimension and type. Hence the image of $A$ stabilizes no nonzero proper totally singular subspace of $W$. Let $P_2$ be the stabilizer in $G$ of a totally singular plane of $V$, a maximal parabolic of $G$ \cite[Proposition 26.1(b)]{MalleTesterman2011}. By hereditary $G$-permutability of $A$ and Lemma \ref{lem:maximal-product}, we have $A\leq P_2^g$ or $G=AP_2^g$.

If $A\leq P_2^g$, then let $Y$ be the totally singular plane of $V$ fixed by $P_2^g$. Since $A\leq P_1$ also fixes $\ell$ and $\ell^\perp$, we consider the positions of $Y$ relative to $\ell$ and $\ell^\perp$. If $\ell\leq Y$, then $Y/\ell$ is a singular line of $W$ fixed by the image of $A$ on $W$. If $\ell\not\leq Y\leq\ell^\perp$, then the image of $Y$ in $W$ is a totally singular subspace fixed by the image of $A$ on $W$. If $Y\not\leq\ell^\perp$, then $Y\cap\ell^\perp$ is an $A$-stable singular line $\neq\ell$, and $((Y\cap\ell^\perp)+\ell)/\ell$ is a singular line of $W$. Each position contradicts the hypothesis that the image of $A$ contains $\POmega(W)$. In the product case, $G=AP_2^g$ makes $A$ transitive on totally singular planes, but $A\leq P_1$ fixes $\ell$ and maps the family of totally singular planes containing $\ell$ to itself. Thus, $A$ is contained in no conjugate of $P_2$, and the factorization $G=AP_2^g$ is also impossible.
\end{proof}

\begin{lemma}\label{lem:orthogonal-semisimple-trivial-borel}
If $A\leq P_1$ and every semisimple image of $A$ in the successive Levi quotients of singular line stabilizers is trivial, then $A$ lies in a Borel subgroup of $G$.
\end{lemma}

\begin{proof}
For the chosen singular line $\ell$, the image of $A$ on $W=\ell^\perp/\ell$ lies in the diagonal quotient, so modulo $R_u(P_1)$ the image of $A$ acts projectively by scalars on $W$ and stabilizes every singular subspace of $W$. We choose a singular line $M$ of $W$ and lift it to a totally singular plane of $V$ containing $\ell$. This plane is the inverse image of $M$ under the quotient map $\ell^\perp\to W$. Repeating the operation of lifting a singular line to extend the stabilized flag in $V$, we build a complete totally singular flag stabilized by $A$. More explicitly, this flag is $0=F_0<F_1<\cdots<F_r$, where $\dim F_i=i$ for $0\leq i\leq r$ and $r$ is the Witt index of $V$. Each $F_i$ is totally singular, and $F_r$ is maximal totally singular. The Witt index drops at each step. Since the totally singular flag constructed in this way is complete, its stabilizer in $G$ is a Borel subgroup of $G$ \cite{Carter1985}, \cite[Proposition 26.1(b)]{MalleTesterman2011}.
\end{proof}

It remains to treat $A$ contained in a Borel subgroup of $G$.

\begin{lemma}\label{lem:pgu4-borel-interface}
Let $\PSU_4(q)\leq E\leq\PGU_4(q)$, and let $B_E$ be the normalizer in $E$ of a Borel subgroup of $\PSU_4(q)$. If $C\leq B_E$ is hereditarily $E$-permutable, then the image of $C$ in the unitary socle of $E$ is trivial. Moreover, the remaining diagonal image is the image of $C$ in $E/\PSU_4(q)$, and it is toral. This toral image is considered when taking the intersection of the kernels of the projective actions on the spaces $W_i$ in {\rm Theorem \ref{thm:orthogonal-borel}}.
\end{lemma}

\begin{proof}
For $q=2$, the quotient $\PGU_4(2)/\PSU_4(2)$ is trivial \cite[p. 66]{Wilson2009} and $\PSU_4(2)\cong\PSp_4(3)$ \cite[\S3.6.5]{Wilson2009}, so Theorem \ref{thm:symplectic} proves the conclusion. Assume that $q\geq3$. The relative root system of rank two of $\SU_4(q)$ is of type $\mathcal{C}_2$. We use the relative root system for ${}^2A_3$ \cite{Carter1985} and the finite Levi quotients \cite[Proposition 26.1(b)]{MalleTesterman2011}. With simple relative roots $\alpha,\beta$, the positive roots are
$$
        \alpha,~\beta,~\alpha+\beta,~2\alpha+\beta.
$$
The two maximal parabolics of $E$ containing $B_E$ have Levi quotients of rank one of type $\PSU_2(q)\leq E_1\leq\PGU_2(q)$ or $\PSL_2(q^2)\leq E_2\leq\PGL_2(q^2)$, with $\PSU_2(q)\cong\PSL_2(q)$ \cite[p. 66]{Wilson2009}. By Lemma \ref{lem:quotient-inheritance}, the images of $C$ in $E_1$ and $E_2$ are hereditarily $E_i$-permutable and lie in Borel subgroups of the respective quotients. Since $q\geq3$, Lemma \ref{lem:pgl2-borel} shows that both images corresponding to the simple roots are trivial.

Hence the components corresponding to the simple roots and the two toral characters vanish, and the remaining possible image of $C$ lies in the product of the two root subgroups corresponding to $\alpha+\beta$ and $2\alpha+\beta$. As in Lemma \ref{lem:symplectic-c2-borel}, the roots $2\alpha+\beta$ and $\alpha+\beta$ occur in Levi factors of rank one after we conjugate these two parabolics of $E$ by the simple reflections. In these Levi factors of rank one, Lemma \ref{lem:quotient-inheritance} with Lemma \ref{lem:pgl2-borel} shows that both components of the image of $C$ corresponding to these two positive roots are trivial. So the image of $C$ in the unitary socle is trivial, and any diagonal toral image of $C$ is treated in Theorem \ref{thm:orthogonal-borel}.
\end{proof}

\begin{theorem}\label{thm:orthogonal-borel}
Let $G=\POmega(V)$ be of type $B_m$, $D_m^+$, or $D_m^-$, and let $B$ be a Borel subgroup of this orthogonal group. If $A\leq B$ is hereditarily $G$-permutable, then $A=1$.
\end{theorem}

\begin{proof}
We induct on $\dim V$. The quotients $W_i$ associated with singular line stabilizers have dimension $\dim V-2$. The bases are the cases in dimension at most six identified by the isomorphisms \eqref{eq:orthogonal-low-rank-isomorphisms}, $\POmega_5\cong\PSp_4$, $\POmega_6^+\cong\PSL_4$, $\POmega_6^-\cong\PSU_4$, $\POmega_4^-\cong\PSL_2(q^2)$, and the group of rank one $\POmega_3\cong\PSL_2(q)$. These are settled by Theorems \ref{thm:split-linear}, \ref{thm:unitary}, \ref{thm:symplectic} and the known result for $\PSL_2$. In particular, this includes their $q=2$ instances ($\PSp_4(2)$, $\PSL_4(2)\cong A_8$, $\PSU_4(2)$). Since $\POmega_{2m+1}(2)\cong\PSp_{2m}(2)$ is symplectic, for type $B_m$ we take $q$ odd. So we assume that $\dim V\geq7$, whence the Witt index is $r\geq3$. Let $E$ be the maximal totally singular subspace of $V$ whose complete flag is stabilized by $B$. This flag is $0=F_0<F_1<\cdots<F_r=E$, with $\dim F_i=i$ for $0\leq i\leq r$. Each $F_i$ is totally singular, and the flag ends at $E$ of dimension $r$. We take a Witt decomposition \cite[pp. 26--39]{KleidmanLiebeck1990}
$$
        V=E\oplus D\oplus E^*,~ E=\gen{e_1,\ldots,e_r},~ E^*=\gen{f_r,\ldots,f_1},
$$
with $D=0$ (even dimension of plus type), $D=\gen{d}$ an anisotropic line (odd dimension), or $D$ an anisotropic plane (even dimension of minus type). The vectors $e_1,\ldots,e_r$ are chosen so that $F_i=\gen{e_1,\ldots,e_i}$ for $1\leq i\leq r$.

We consider the stabilizer $P_E$ of $E$ in $G$ and the projective factor of its Levi quotient acting on $E$. We form this projective factor by taking the quotient of the Levi quotient by the kernel of its projective action on $E$. Write $\rho_E$ for this projective action of the Levi quotient of $P_E$ on $E$. The projective factor is the image of $\rho_E$. This quotient has linear socle $\PSL(E)\cong\PSL_r(q)$ with toral and anisotropic factors \cite[Proposition 26.1(b)]{MalleTesterman2011}. The image of $A$ lies in a Borel subgroup of this Levi quotient. Therefore, Lemma \ref{lem:sl-borel} shows that this image of $A$ is trivial. In particular, the group $A$ acts projectively by scalars on $E$ and stabilizes each $\gen{e_i}$.

For each $i$ let $W_i=\gen{e_i}^\perp/\gen{e_i}$. We have $A\leq P_{\gen{e_i}}$, where this parabolic is the stabilizer of $\gen{e_i}$ in $G$. By Lemma \ref{lem:quotient-inheritance}, the image of $A$ is hereditarily permutable in the orthogonal quotient of this stabilizer acting on $W_i$. The quotient is again a projective orthogonal quotient \cite[Proposition 26.1(b)]{MalleTesterman2011}, and the image of $A$ lies in its induced Borel subgroup. By induction on $\dim V$, the image of $A$ in each orthogonal quotient whose space has dimension greater than six is trivial. For the quotients of low rank reached at the end of the descent, we use Lemma \ref{lem:pgl2-borel}, Lemma \ref{lem:sl-borel}, Theorem \ref{thm:symplectic-borel}, and the $\PGU_4(q)$ case of Lemma \ref{lem:pgu4-borel-interface} for the minus case in dimension six. The plus case in dimension four gives two quotients of rank one, handled by Lemma \ref{lem:pgl2-borel}. Any diagonal or toral part is left to the intersection of the kernels of the projective actions on all the spaces $W_i$. So the orthogonal image of $A$ on $W_i$ is trivial for each $i$.

It remains to determine the intersection of the kernels of the projective actions on the spaces $W_i$. In the standard root notation \cite{Carter1985} and with a Witt basis as in \cite[pp. 26--39]{KleidmanLiebeck1990}, the root subgroups involving two coordinates have supports corresponding to the roots $e_i-e_j$ and $e_i+e_j$ ($i<j$). In odd dimension there are groups corresponding to short roots with support $e_i$, and in even dimension of minus type the root subgroups involving $D$ also have support on one coordinate. A root subgroup involving the coordinates indexed by $a$ and $b$ acts nontrivially on $W_i$ for some $i\notin\{a,b\}$ (possible as $r\geq3$). Similarly, a root subgroup involving only the coordinate indexed by $a$ acts nontrivially on some $W_i$ with $i\neq a$. Hence all unipotent components vanish. For the toral part of $A$, the projective factor of the Levi quotient of $P_E$ acting on $E$ shows that $A$ acts on $E$ as a single projective scalar, so $e_i\mapsto\lambda e_i$ for all $i$. The dual flag in $E^*$ has terms $F_j^*=\gen{f_r,\ldots,f_{r-j+1}}$ for $1\leq j\leq r$. Since the elements of $A$ leave the values of the quadratic form unchanged and stabilize the dual flag, we have $f_i\mapsto\lambda^{-1}f_i$. The anisotropic part $D$ lies in $\gen{e_i}^\perp$ for every $i$ and meets $\gen{e_i}$ trivially, so $D$ embeds in each $W_i$. Since the orthogonal image of $A$ on $W_i$ is trivial, $A$ acts on $W_i$, and hence on $D\subseteq W_i$, as a scalar. Write $\nu_i$ for the scalar on $W_i$ and $\mu$ for the scalar by which $A$ acts on $D$. In the standard coordinate description of a maximal torus of the orthogonal group of $V$ \cite[p. 70]{Wilson2009}, an element in all the kernels for the projective actions on the spaces $W_i$ is, modulo the projective centre,
$$
        e_i\mapsto\lambda e_i,~ f_i\mapsto\lambda^{-1}f_i~(1\leq i\leq r),~ x\mapsto\mu x~(x\in D).
$$
The hyperbolic pair $e_j,f_j$ ($j\neq i$) inside $W_i$ shows that $\nu_i=\lambda=\lambda^{-1}$, so $\lambda^2=1$, while the image of $D$ in $W_i$ shows that $\nu_i=\mu$, hence $\mu=\lambda$. In odd dimension $\det A=\mu$, so $A\in\operatorname{SO}(V)$ forces $\mu=1$. Therefore, $\lambda=1$ and hence $A=1$. By contrast, in even dimension of minus type $\mu=\lambda=\pm1$, so $A$ acts as the scalar $\lambda$ on all of $V$ and is projectively trivial. In particular, the action on the anisotropic plane is projectively trivial. The case of even dimension and plus type ($D=0$) is identical with no anisotropic part. For each of these three orthogonal types, the kernels for the projective actions on all the spaces $W_i$ intersect in projective scalars. Hence $A=1$.
\end{proof}

\subsection{Proof of \texorpdfstring{Theorem \ref{thm:orthogonal}}{the orthogonal theorem}}

\begin{proof}[Proof of Theorem \ref{thm:orthogonal}]
The cases of low rank in \eqref{eq:orthogonal-low-rank-isomorphisms} are handled by Theorems \ref{thm:split-linear}, \ref{thm:unitary} and \ref{thm:symplectic} and the base of rank one. Otherwise, let $A<G$ be proper, nontrivial, and hereditarily $G$-permutable. By the alternatives obtained from hereditary permutability with the singular line stabilizer $P_1$, the factorization case $G = A P_1 ^x$ is impossible (Proposition \ref{prop:orthogonal-no-p1-factorization}), so $A\leq P_1$ after conjugating $A$ in $G$. The image of $A$ on $W = \ell^\perp/\ell$ is governed by Lemma \ref{lem:orthogonal-actual-levi-alternatives}. A proper nontrivial semisimple image $C$ of $A$ lies in a proper parabolic of the Levi quotient of the singular line stabilizer in that descent step, so $A$ lies in the corresponding smaller parabolic of $G$ and the flag stabilized by $A$ extends in the descent. The case in which the semisimple image of $A$ is full cannot occur (Lemma \ref{lem:orthogonal-full-levi-obstruction}). If the semisimple image of $A$ is trivial, then $A$ is contained in a Borel subgroup of $G$ (Lemma \ref{lem:orthogonal-semisimple-trivial-borel}). In this case, $A$ is trivial by Theorem \ref{thm:orthogonal-borel}. So $A=1$.
\end{proof}

\section{Proof of the main theorem}\label{sec:assembly}

\begin{proof}[Proof of Theorem \ref{thm:main}]
By the classification of finite simple groups, a finite simple group is cyclic of prime order, alternating, classical, exceptional of Lie type (including the Tits group), or sporadic. Since $G$ is nonabelian, the cyclic case cannot occur. The alternating groups are covered by Vasil'ev and Tyutyanov \cite{VT2012}. For the sporadic groups and the Tits group, we use the result of Galt and Tyutyanov \cite[Corollary]{GT2022}. Galt and Tyutyanov \cite[Theorem 1 and Proposition 1]{GT2023} also obtained the required conclusion for the exceptional groups of Lie type and the groups $\PSL_2(q)$. The groups $\PSU_3(q)$ are covered by Theorem \ref{thm:unitary-low-bases}, with \cite[Theorem 2]{GT2025} as the source for $q\geq3$. The remaining classical groups of higher rank are covered by the four theorems proved in this paper, namely $\PSL_n(q)$ by Theorem \ref{thm:split-linear}, $\PSU_n(q)$ by Theorem \ref{thm:unitary}, $\PSp_{2m}(q)$ by Theorem \ref{thm:symplectic}, and $\POmega(V)$ by Theorem \ref{thm:orthogonal}. Hence $G$ has no proper nontrivial hereditarily $G$-permutable subgroup.
\end{proof}

\appendix

\section{Computational verification of the small cases}\label{app:computational-certificates}

Six steps in the proof are settled by a direct computation in a single small group. The first two steps are Lemma \ref{lem:pgl2-borel} for $\PGL_2(q)$ with $q\in\{5,7,9,11\}$ and the exceptional value $Q=2$ of Lemma \ref{lem:unitary-rank-three-toral-collapse} for $\PGU_3(2)$. The next two are the nonsimple quotient $\PSp_4(2)\cong S_6$ of Lemma \ref{lem:symplectic-c2-borel} and the prime degree $5$ of Lemma \ref{lem:s5-no-solvable-hereditary} for $S_5$. The final two are the degree-$9$ base of Lemma \ref{lem:symmetric-solvable-hereditary-trivial} for $S_9$ and the unipotent case of $\PSU_4(2)$ (Appendix \ref{app:psu4-2}). Each computation uses a short program with permutation operations or matrix operations over finite fields. The six programs are included as supplementary material. This material also contains the $S_4$ enumeration, the check that the failure of permutability persists when fixed points are added, and the $\PSU_4(2)$ toral enumeration. We describe the conclusion established by each computation.

The checks in the linear, unitary toral, and symplectic Borel cases test permutability with cyclic subgroups inside the group generated by the candidate subgroup and the chosen cyclic subgroup. For $\PGL_2(q)$, we use permutations of the projective line over $\F_q$. For $\PGU_3(2)$ and $\Sp_4(2)$, we use matrices over $\F_4$ and $\F_2$, respectively. We enumerate the subgroups of a point stabilizer in $\PGL_2(q)$, a diagonal torus in $\PGU_3(2)$, or a Borel subgroup in $\Sp_4(2)$ by closure under multiplication and inverses. For each nontrivial candidate $C$ we find a cyclic subgroup $H=\gen{h}$ for which
$$
        CH^x\neq H^xC~\text{for every }x\in F=\gen{C,h}.
$$
By Lemma \ref{lem:generated-hereditary-formulation}, if no conjugate of $H$ permutes with $C$ in the generated subgroup $F$, then $C$ is not hereditarily permutable in the group used in the computation. Hence none of the nontrivial candidates is hereditarily permutable. The fourth check (Appendix \ref{app:s5-solvable-certificate}) uses this permutability test with cyclic subgroups for every subgroup of $S_5$. It then lists the subgroups of $S_5$ that pass the test for every subgroup $H$ of $S_5$. The programs enumerate all the specified subgroups using the matrix and permutation operations described above. The accompanying transcripts list the subgroup counts and examples of the cyclic subgroups used to test permutability. The tables for the individual small groups summarise the totals. A subgroup $C$ passes the check for a subgroup $H$ if $CH^x=H^xC$ for some $x\in\gen{C,H}$. The column \emph{Passed} lists or counts the nontrivial subgroups $C$ that pass this check for every subgroup $H$ used in the computation.

\subsection{The small \texorpdfstring{$\PGL_2(q)$}{PGL2(q)} cases}\label{app:pgl2-small-certificate}

The group $\PGL_2(q)$ is realised on the projective line over $\F_q$, with $\F_9=\F_3[\omega]/(\omega^2+1)$, and $B_E$ is a point stabilizer in $\PGL_2(q)$. We enumerate every subgroup of $B_E$ and choose an element $h$ for each nontrivial subgroup $C$ such that no conjugate of $H=\gen{h}$ in $\gen{C,H}$ permutes with $C$. The table lists the group orders, orders of point stabilizers, and subgroup counts. Its last column counts the nontrivial subgroups $C\leq B_E$ for which, for every cyclic subgroup $H$ chosen to test $C$, some $x\in\gen{C,H}$ satisfies $CH^x=H^xC$.
$$
\begin{array}{c|c|c|c|c}
q & |\PGL_2(q)| & |B_E| & \#\{\text{subgroups of }B_E\} & \text{Passed}\\\hline
3 & 24 & 6 & 6 & 0\\
5 & 120 & 20 & 14 & 0\\
7 & 336 & 42 & 26 & 0\\
9 & 720 & 72 & 48 & 0\\
11 & 1320 & 110 & 38 & 0
\end{array}
$$
The value $q=3$, where $\PGL_2(3)\cong S_4$, is listed for completeness. We also treat the case $q=3$ by hand in Lemma \ref{lem:pgl2-borel}.

\subsection{The \texorpdfstring{$\PGU_3(2)$}{PGU3(2)} diagonal torus}\label{app:pgu3-2-toral-certificate}

For the exceptional value $Q=2$ in Lemma \ref{lem:unitary-rank-three-toral-collapse}, we construct $\GU_3(2)$ over $\F_4=\F_2[\alpha]/(\alpha^2+\alpha+1)$ with the Hermitian form whose Gram matrix is the identity. The quotient of $\GU_3(2)$ by its scalar centre gives $\PGU_3(2)$ of order $216$. We then take the projective diagonal torus $D$ of order $9$. There are five nontrivial subgroups $A$ of $D$, four of order $3$ and one of order $9$. For each such subgroup, we find a cyclic subgroup $H$ such that no conjugate of $H$ in $\gen{A,H}$ permutes with $A$. The table lists the orders of $A$, $H$ and $\gen{A,H}$. Its final column gives the number of distinct conjugates of $H$ in $\gen{A,H}$.
$$
\begin{array}{c|c|c|c}
|A| & |H| & |\gen{A,H}| & \#\{H^x:x\in\gen{A,H}\}\\\hline
3 & 2 & 18 & 3\\
3 & 3 & 27 & 3\\
3 & 2 & 18 & 3\\
3 & 4 & 36 & 9\\
9 & 4 & 216 & 27
\end{array}
$$
So no nontrivial subgroup of the diagonal torus $D$ is hereditarily $\PGU_3(2)$-permutable.

\subsection{The \texorpdfstring{$\PSp_4(2)$}{PSp4(2)} Borel case}\label{app:sp4-2-borel-certificate}

For the nonsimple quotient $\PSp_4(2)\cong S_6$ \cite{Wilson2009} in Lemma \ref{lem:symplectic-c2-borel}, we construct $\Sp_4(2)$ as the $4\times4$ matrices over $\F_2$ that leave the standard symplectic form unchanged. We take the Borel subgroup $B$ stabilizing the flag $\gen{e_1}<\gen{e_1,e_2}$. This subgroup has order $16$. We enumerate all $35$ subgroups of $B$, and for each of the $34$ nontrivial subgroups $C$ we find a cyclic subgroup $H$ such that no conjugate of $H$ in $\gen{C,H}$ permutes with $C$. Therefore, none of these subgroups is hereditarily $\Sp_4(2)$-permutable. The table lists the group order, the order of the Borel subgroup, and the subgroup count. Its last column counts the nontrivial subgroups $C\leq B$ for which, for every cyclic subgroup $H$ chosen to test $C$, some $x\in\gen{C,H}$ satisfies $CH^x=H^xC$.
$$
\begin{array}{c|c|c|c}
|\Sp_4(2)| & |B| & \#\{C\leq B\} & \text{Passed}\\\hline
720 & 16 & 35 & 0
\end{array}
$$

\subsection{Solvable subgroups of \texorpdfstring{$S_5$}{S5}}\label{app:s5-solvable-certificate}

The prime degree $5$ of Lemma \ref{lem:s5-no-solvable-hereditary} is checked directly. We enumerate all $156$ subgroups of $S_5$ and test each for hereditary $S_5$-permutability in the generated formulation of Lemma \ref{lem:generated-hereditary-formulation}. A subgroup $B$ passes the test precisely when, for every subgroup $H\leq S_5$, some $x\in\gen{B,H}$ satisfies $BH^x=H^xB$. The only nontrivial subgroups that pass every test are $A_5$ and $S_5$.
$$
\begin{array}{c|c|c}
|S_5| & \#\{\text{subgroups}\} & \text{Passed}\\\hline
120 & 156 & A_5,\ S_5
\end{array}
$$
The groups $A_5$ and $S_5$ are both nonsolvable, so no nontrivial solvable subgroup of $S_5$ is hereditarily $S_5$-permutable.

\subsection{The base case in degree nine}\label{app:degree-nine-base}

The computation in degree nine proves the base case of Lemma \ref{lem:symmetric-solvable-hereditary-trivial}. Moreover, this computation establishes the inductive step because adding fixed points does not change the failure of permutability. By Lemma \ref{lem:solvable-simplex-prime-bound}, a solvable hereditarily $S_9$-permutable subgroup stabilizes a partition $9=4+5$ and acts trivially on the part of size $5$. Hence it lies, up to conjugacy, in the symmetric group $S_4$ on a fixed $4$-set, fixing the other five points. We calculate products and conjugates of permutations in $S_9$ on $\{0,\dots,8\}$. With these permutation operations, we enumerate all $30$ subgroups of the copy of $S_4$ on the fixed set of four points. The $8$ nontrivial isomorphism types are $C_2,C_3,C_2^2,C_4,S_3,D_8,A_4,S_4$. For each nontrivial subgroup $B$, we seek a cyclic subgroup $H=\gen{h}$, where $h$ is a cycle of length $5$ or a $7$-cycle mixing the $4$-set with the fixed points. The required condition is that no $x\in\gen{B,H}$ makes $BH^x$ a subgroup (the generated formulation, Lemma \ref{lem:generated-hereditary-formulation}). For every nontrivial subgroup of this copy of $S_4$, we find a cyclic subgroup satisfying this condition.
$$
\begin{array}{c|c|c}
\#\{\text{subgroups of }S_4\} & \text{Failed types} & \text{Passed}\\\hline
30 & C_2,\,C_3,\,C_2^2,\,C_4,\,S_3,\,D_8,\,A_4,\,S_4 & 0
\end{array}
$$
For instance, for $B=\gen{(0\,1\,2)}\cong C_3$ we test permutability with the cyclic subgroup $H$ generated by $h=(0\,1\,2\,4\,5)$. We have $\gen{B,H}\cong A_5$ on $\{0,1,2,4,5\}$, which has no subgroup of order $15$ (cf.\ Lemma \ref{lem:symmetric-three-cycle-not-hereditary}). The groups $A_4$ and $S_4$ each fail the hereditary permutability test with the subgroup generated by the $7$-cycle $(0\,1\,2\,3\,4\,5\,6)$. Hence no nontrivial subgroup of this copy of $S_4$ is hereditarily $S_9$-permutable. By the reduction to a set of four points with a pointwise fixed set of five points, no nontrivial solvable subgroup of $S_9$ is hereditarily $S_9$-permutable. Moreover, each cyclic subgroup used in these tests is supported on at most seven points. Therefore, for a subgroup of $S_4$ fixing additional points, the group $\gen{B,H}$ lies inside a fixed $S_7$. This symmetric group acts on a set $\Omega_7$ of seven points containing the four points on which $B$ acts and the support of $H$, and fixes every point outside $\Omega_7$. The same failure of permutability occurs in this group and is unchanged in any larger $S_N$. The fact that adding fixed points does not change these failures of permutability establishes the inductive step of Lemma \ref{lem:symmetric-solvable-hereditary-trivial}. Consequently, $\PSU_9(q)$ with $3\mid q+1$, in particular $U_9(2)$, has no proper nontrivial hereditarily permutable subgroup among the cases considered through this action on the nine simplex lines.

A companion enumeration over $S_4$ itself (underlying Lemma \ref{lem:symmetric-solvable-hereditary-trivial}) lists the hereditarily $S_4$-permutable subgroups of $S_4$ as $1$, $V_4$, the three subgroups isomorphic to $D_8$, $A_4$, $S_4$. Every nontrivial subgroup in this list is transitive on the four points, so none lies in a point stabilizer $S_3$ or a group generated by two disjoint transpositions $\gen{(a\,b),(c\,d)}$. We also check the $20$ nontrivial subgroups contained in a point stabilizer $S_3$ or in a group $\gen{(a\,b),(c\,d)}$ generated by two disjoint transpositions of $S_4$. For each such subgroup $B$, we find a cyclic subgroup $H$ of $S_4$ such that $BH^x\neq H^xB$ for every $x\in\gen{B,H}$.

\subsection{The \texorpdfstring{$\PSU_4(2)$}{PSU4(2)} unipotent case}\label{app:psu4-2}

This computation settles the remaining case in even dimension of the unitary analysis of subgroups in a unipotent radical, at $q=2$. Since Proposition \ref{prop:unitary-unipotent} requires odd dimension, it does not cover this case. In the situation of Proposition \ref{prop:unitary-even} the group $A$ fixes a nondegenerate line $U$ and lies in the unipotent radical $R_u(P_2)$ of the plane stabilizer $P_2$ in $G$, with $G=\PSU_4(2)$. Hence $A$ fixes the totally singular plane $X$ stabilized by $P_2$ pointwise. Since $\dim(X\oplus U)=3$, the space $X\oplus U$ is a hyperplane fixed by every nonidentity element of $A$. Therefore, every nonidentity element of $A$ is an element corresponding to a long root. These elements corresponding to long roots share the common centre $(X\oplus U)^\perp$, so $A$ lies in a single subgroup corresponding to a long root $T_v$. The vector $v$ is a nonzero vector spanning the line $(X\oplus U)^\perp$. At $q=2$ a subgroup corresponding to a long root has order $q=2$, so $A=\gen{t}$ is generated by a single involution corresponding to a long root. It suffices to show that no involution of $\PSU_4(2)$ generates a hereditarily $\PSU_4(2)$-permutable subgroup. Now $\PSU_4(2)$ has exactly two classes of involutions, distinguished in the degree-$45$ action on isotropic lines by their number of fixed points, $13$ (elements corresponding to long roots) and $5$. We construct $\PSU_4(2)=\SU_4(2)$ on the $45$ isotropic lines using matrix operations over $\F_4$. Using the generated formulation of Lemma \ref{lem:generated-hereditary-formulation}, we verify that
$$
\begin{array}{c|c|c}
\text{involution class} & \text{fixed points} & \text{cyclic subgroup }H\\\hline
\text{long root involutions} & 13 & \text{a cyclic subgroup }C_5\\
\text{other} & 5 & \text{cyclic of order }9
\end{array}
$$
in each case no $x\in\gen{\gen{t},H}$ makes $\gen{t}H^x$ a subgroup. Hence no involution of $\PSU_4(2)$, and in particular no involution corresponding to a long root, generates a hereditarily $\PSU_4(2)$-permutable subgroup. No such involution subgroup is hereditarily permutable. This contradicts hereditary permutability of $A=\gen{t}$, so the unipotent subgroup $A$ must be trivial. For the class of involutions corresponding to long roots, the failure of hereditary permutability can also be seen from the dimensions of fixed spaces. A transvection has a hyperplane as its fixed space. If a transvection conjugates an element $g$ of order $5$ with no nonzero fixed vector to $g^{-1}$, then the transvection would have a fixed space of dimension two. Hence a transvection cannot normalize the cyclic subgroup generated by an element of order five with no nonzero fixed vector. (The toral part of the Levi quotient of the stabilizer in $\PSU_4(2)$ of the singular flag fixed by $A$ is handled by Lemma \ref{lem:unitary-flag-toral}. As an independent check, a companion enumeration confirms that the cyclic and elementary abelian toral subgroup types $C_3$, $C_5$, $C_9$, $C_3\times C_3$ of $\PSU_4(2)$ each fail the hereditary permutability test. For each toral subgroup $C$ tested, the program finds a cyclic subgroup $H$ such that $CH^x\neq H^xC$ for every $x\in\gen{C,H}$.)
\section*{Acknowledgement}
I am grateful to Baoyu Zhang at the university of Birmingham, who introduced me to mathematics, gave voluntary lectures for first year maths undergraduates and taught me group theory, which was the reason why I transferred my major to mathematics and have kept doing mathematics since I met him in my first year as an undergraduate. See also the acknowledgements of  \cite{Zhang2025}. I thank Baoyu for suggesting this problem to me and giving illuminating ideas when I travelled with him during his invited visits to Jilin University, Peking University and Nankai University at the end of last year. This paper is not up to Baoyu’s standards of writing papers or formalizing paper in Lean. AI is not used in proving the conjecture, and ChatGPT is used in revising the paper due to my language limitation. I am happy to defend every piece of details in the paper via online meeting if anyone interested in this paper contacts me.

\end{document}